\documentclass[11pt, twoside, leqno]{article}

\usepackage{amsmath}
\usepackage{amssymb}
\usepackage{titletoc}
\usepackage{mathrsfs}
\usepackage{amsthm}

\usepackage{indentfirst}
\usepackage{color}
\usepackage{txfonts}
\usepackage{enumerate}
\usepackage{anysize}

\allowdisplaybreaks

\newtheorem{theorem}{Theorem}[section]
\newtheorem{lemma}[theorem]{Lemma}
\newtheorem{corollary}[theorem]{Corollary}
\newtheorem{proposition}[theorem]{Proposition}
\theoremstyle{definition}

\newtheorem{remark}[theorem]{Remark}
\newtheorem{definition}[theorem]{Definition}

\numberwithin{equation}{section}

\begin{document}

\title{\bf\Large
Maximal Function and Atomic Characterizations of
Matrix-Weighted Hardy Spaces with Their
Applications to Boundedness of Calder\'on--Zygmund Operators
\footnotetext{\hspace{-0.35cm} 2020 {\it Mathematics Subject Classification}.
Primary 42B30; Secondary 42B25, 42B20, 42B35, 46E40, 47A56.\endgraf
{\it Key words and phrases.} Hardy space, matrix weight, reducing operator, maximal function,
atom, Calder\'on--Zygmund operator.\endgraf
This project is partially supported by the National
Key Research and Development Program of China
(Grant No. 2020YFA0712900),
the National Natural Science Foundation of China
(Grant Nos. 12431006 and 12371093),
and the Fundamental Research Funds
for the Central Universities (Grant No. 2233300008).}}
\date{}
\author{Fan Bu, Yiqun Chen, Dachun Yang\footnote{Corresponding author,
E-mail: \texttt{dcyang@bnu.edu.cn}/{\color{red} January 30, 2025}/Final Version.}\
\ and Wen Yuan}

\maketitle

\vspace{-0.8cm}

\begin{center}
\begin{minipage}{13.8cm}
{\small {\bf Abstract}\quad
Let $p\in(0,1]$ and $W$ be an $A_p$-matrix weight,
which in scalar case is exactly a
Muckenhoupt $A_1$ weight.
In this article, we introduce matrix-weighted Hardy spaces $H^p_W$
via the matrix-weighted grand non-tangential maximal function
and characterize them, respectively,
in terms of various other maximal functions
and atoms, both of which are closely related to matrix weights under consideration
and their corresponding reducing operators. As applications,
we first establish the finite atomic characterization of $H^p_W$,
then using it we give a criterion on the boundedness of sublinear operators
from $H^p_W$ to any $\gamma$-quasi-Banach space, and finally
applying this criterion we further
obtain the boundedness of Calder\'on--Zygmund
operators on $H^p_W$.
The main novelty of these results lies in
that the aforementioned maximal functions related to reducing operators
are new even in the scalar weight case
and we characterize these matrix-weighted Hardy spaces
by a fresh and natural variant of classical weighted atoms
via first establishing a Calder\'on--Zygmund decomposition
which is also new even in the scalar weight case.
}
\end{minipage}
\end{center}

\vspace{0.2cm}

\tableofcontents

\vspace{0.2cm}

\section{Introduction}

To develop the prediction theory of multivariate stochastic processes,
Wiener and Masani \cite[\S 4]{wm58} introduced the matrix-weighted Lebesgue
space $L^2_W$ on $\mathbb{R}^n$, where $W$ is a matrix weight.
In the 1990s, Nazarov, Treil, and Volberg
\cite{nt96,tv971,tv97,v97} originally studied the ingenious generalization of
scalar Muckenhoupt weights to matrix weights.
Actually, to solve the problem about the angle between past
and future of the multivariate random stationary process
and the problem about the boundedness of the inverse of Toeplitz operators,
Treil and Volberg \cite{tv97} found the proper analogue
of the Muckenhoupt $A_2$ weight condition
on $\mathbb R$ in this matrix-valued context such that
the Hilbert transform is bounded on $L^2_W$ if and only if $W$ is such
an $A_2$ matrix weight. Extensions of this to the matrix-weighted Lebesgue space
$L^p_W$ on $\mathbb R$
with $W\in A_p$ for general $p\in(1,\infty)$
were later given by Nazarov and Treil \cite{nt96}
and with a different approach by Volberg \cite{v97}.
Then Bownik \cite{b01} studied the self-improvement
property of $A_p$ matrix weights
and Christ and Goldberg \cite{cg01,Gold} obtained
the boundedness of both certain maximal operator and certain
convolutional Calder\'on--Zygmund operator appropriate for
$A_p$ matrix weights.
Recently, due to its significant role in the study
about systems of partial differential equations \cite{cmr,IM19}
and multivariate stationary random processes \cite{tv97,tv99},
matrix weights have attracted increasing attention.
Particularly, Nazarov et al. \cite{nptv17}
obtained the boundedness of Calder\'on--Zygmund operators
on the matrix-weighted $L^2_W$ over ${\mathbb R}^n$ with the operator norm
$C[W]^{\frac 32}_{A_2}$.
More surprisingly, Domelevo et al. \cite{dptv24}
further proved that this exponent $\frac{3}{2}$ is indeed sharp.
In addition, Bownik and Cruz-Uribe \cite{bc22}
extended the Jones factorization theorem and the Rubio de Francia
extrapolation theorem on scalar Muckenhoupt weights to matrix weights.
For more studies on the boundedness of classical operators on the
matrix-weighted Lebesgue space $L^p_W$, we refer to \cite{dhl20, llor23, llor24, v24}.
Moreover, matrix-weighted Besov--Triebel--Lizorkin spaces
were originally studied by Frazier and Roudenko
\cite{fr04,fr21,rou03,rou04}.
Furthermore, Bu et al. \cite{bhyy1,bhyy2,bhyy3,byy23}
investigated matrix-weighted Besov--Triebel--Lizorkin-type spaces.
For more studies on other matrix-weighted function spaces, we refer
to \cite{bx24-1,bx24-2,bx24-3,byy23,lyy24-1,lyy24-2,n25,wgx24,wyy23}.
However, so far a real-variable theory of matrix-weighted Hardy spaces
is still missing. Then a natural \emph{question} is whether
or not such a real-variable theory truly exists.

In this article, we give an affirmative answer to this question.
Recall that the classical Hardy space $H^p$ with $p\in(0,1]$
was originally introduced by Stein and Weiss \cite{SW60} and
further great significantly developed by
Fefferman and Stein in their seminal article \cite{FS72}.
Since then, the real-variable theory of Hardy-type spaces on ${\mathbb{R}^n}$
proves to play an important and irreplaceable role in both harmonic
analysis (see, for instance, \cite{clms,CW77,bdt,gnns19,gnns21,lyz,ns2012,shyy,yyz,zyyw})
and partial differential equations
(see, for instance, \cite{clms,bdl18,bdl20,m94,shyy}).
We refer also to the well-known monographs \cite{Duo01,gr85,g14m,Lu95,S93,t86} for
more detailed studies on classical Hardy spaces.
Particularly, it is worth mentioning that classical weighted Hardy spaces
have been detailedly studied by Garcia-Cuerva \cite{G79}, Bui \cite{B81}, and
Str\"omberg and Torchinsky \cite{st89}, and some recent variants of
classical weighted Hardy spaces and their applications
are given by Bonami et al. \cite{bcklyy,bgk12},
Bui \cite{B14}, Ky \cite{ky14}, Yang et al. \cite{ylk}, Ho \cite{h17,h19},
Cruz-Uribe and Wang \cite{cw14},
Nakai and Sawano \cite{ns2012}, Sawano et al. \cite{shyy},
and Izuki et al. \cite{INNS}.

In this article, let $p\in(0,1]$ and $W$ be an $A_p$-matrix
weight (or, more generally,
$A_{p,\infty}$-matrix weight), which in scalar cases is exactly the
Muckenhoupt $A_1$ weight (or, more generally, $A_\infty$ weight).
We first introduce matrix-weighted Hardy spaces
$H^p_W$ via the matrix-weighted grand non-tangential maximal function and
then characterize them, respectively, in terms of various other maximal functions
($W\in A_{p,\infty}$)
and atoms ($W\in A_p$), both of which are closely related to matrix weights under consideration
and their corresponding reducing operators. As applications, for any given $W\in A_p$,
we first establish the finite atomic characterization of $H^p_W$,
then using it we give a criterion on the boundedness of sublinear operators
from $H^p_W$ to any $\gamma$-quasi-Banach space, and finally
applying this criterion we further
obtain the boundedness of Calder\'on--Zygmund
operators both from $H^p_W$ to $L^p_W$ and from $H^p_W$ to itself.

The main novelty of these results lies in
that the aforementioned maximal functions related to reducing operators
are new even in the scalar weight case
and we characterize these matrix-weighted Hardy spaces
by a fresh and natural variant of classical weighted atoms
via first establishing a Calder\'on--Zygmund decomposition
which is also new even in the scalar weight case.

To establish the equivalence of various matrix-weighted maximal functions,
we propose several $\mathbb A$-matrix-weighted maximal functions,
which consists of a new and practical approach that connects
and unifies different types of matrix-weighted maximal functions.
Motivated by the observation \cite[p.\,454,\ Remark]{v97} of Volberg that
the set where the $A_{p,\infty}$ matrix weight
is much larger than its average is large,
we introduce an $\mathbb A$-matrix-weighted non-tangential
infimum maximal function, which provides a new infimum
characterization. This function allows us to use the ``large''
subsets to control the overall behavior and establish a
strong link between matrix weights and reducing operators.
The reducing operator, regarded as an average of matrix weights,
provides some kind of continuity that improves the properties
of the combined entity of matrix weights and functions,
making the analysis more manageable despite
the intrinsic inseparability of these two research objects.
Additionally, we present a new matrix-weighted
Calder\'on--Zygmund decomposition, which naturally overcomes the
fundamental difficulty of finding an appropriate definition
for matrix-weighted atoms.
This approach facilitates the atomic characterization of $H^p_W$.
Unlike the scalar case where estimates directly target the
function under consideration, in the matrix weight setting estimates focus on the
combined entity of the matrix weight and the vector function under consideration.
As a result, some straightforward conclusions,
such as pointwise estimates of ``good'' functions in the
Calder\'on--Zygmund decomposition, no longer hold.
To overcome this essential difficulty, we utilize the integrals of ``good''
functions (actually locally polynomials)
over cubes (some kind of averages) to control their coefficients, which
allows us to directly focus on the vector-valued function itself, independent of matrix weights.

The organization of the remainder of this article is as follows.

In Section \ref{equ-M},
let $p\in(0,\infty)$ and $W$ be an
$A_{p,\infty}$-matrix weight, which in scalar cases is exactly an
Muckenhoupt $A_\infty$ weight. Let $\mathbb A:=\{A_Q\}_{Q\in\mathscr{Q}}$
be a family of reducing operators of order $p$ for $W$.
We first give a brief review of matrix weights
and introduce the matrix-weighted Hardy space via the
matrix-weighted grand non-tangential maximal function;
see Definition \ref{de-HW}.
Subsequently, we introduce various types of $\mathbb A$-matrix-weighted maximal
functions via reducing operators (see Definition \ref{reducing hardy})
and establish their equivalences in $L^p$; see Theorem \ref{if and only if A}.
Furthermore, we prove the equivalence between
$\mathbb A$-matrix-weighted maximal functions and
matrix-weighted maximal functions (see Theorem \ref{weight and reducing}) and
use this to derive various equivalent characterizations of
matrix-weighted maximal functions;
see Theorem \ref{if and only if W}.
Finally, we obtain several fundamental properties of
matrix-weighted Hardy spaces,
such as embedding into the space of tempered distributions
and their completeness;
see Propositions \ref{embedding} and \ref{jwdj}.
It is worth noting that all the $\mathbb A$-matrix-weighted maximal
functions (see Definition \ref{reducing hardy}),
which can be regarded as the discrete variant of the
corresponding matrix-weighted maximal function in some sense
(see Remark \ref{2.29}), are new even in the scalar case,
where the matrix-weighted non-tangential infimum
maximal function [see Definition \ref{reducing hardy}(iii)] represents the
most significant novelty via providing a fresh infimum
characterization of maximal functions.

In Section \ref{ad}, let $p\in(0,1]$.
We first introduce a novel and natural variant of classical atoms
in the matrix-weighted setting (see Definition \ref{F-atom})
and then establish the atomic characterization
of the matrix-weighted Hardy space ($W\in A_p$); see Theorem \ref{W-F-atom}.
To achieve this goal, we establish a new variant of the Calder\'on--Zygmund decomposition
in the matrix-weighted setting ($W\in A_{p,\infty}$),
which is of independent interest; see Theorem \ref{l-CZ}.
Additionally, we show that the space of locally integrable functions
is dense in $H^p_W$ when $p\in(0,1]$ and $W\in A_p$; see Proposition \ref{dense}.
Finally, we also prove
that, when $p\in(1,\infty)$ and $W\in A_p$, the matrix-weighted Hardy space
coincides with the matrix-weighted Lebesgue space
(see Theorem \ref{L-H}), which further implies the reasonability
of this kind of the matrix-weighted Hardy space.
It is worth noticing that the Calder\'on--Zygmund decomposition
associated with matrix weights is also new even in the scalar case.

In Section \ref{fatom}, let $p\in(0,1]$ and $W\in A_p$,
we first introduce the finite atomic matrix-weighted Hardy space (see Definition
\ref{de-fin}) and
establish the finite atomic characterization of $H^p_W$;
see Theorem \ref{finatom}.
Then, using this, we give a criterion on the
boundedness of sublinear operators
from $H^p_W$ to any $\gamma$-quasi-Banach space;
see Theorem \ref{subl}.

In Section \ref{cz}, applying Theorem \ref{subl}, we show that, when $p\in (0,1]$
and $W\in A_p$, Calder\'on--Zygmund operators are
bounded from the matrix-weighted Hardy space $H^p_W$ to
the matrix-weighted Lebesgue space $L^p_W$ and from $H^p_W$ to itself;
see Theorem \ref{CZ}.

To limit the length of this article,
we leave the Littlewood--Paley characterizations of $H^p_W$
and its dual space in a forthcoming article.

At the end of this section, we make some conventions on notation.
In the whole article, we work on ${\mathbb R}^n$ and,
for simplicity of presentation, we will not indicate this
underlying space in related symbols.
Let $\mathbb{N}:=\{1,2,\dots\}$, $\mathbb{Z}_+:=\mathbb{N}\cup\{0\}$,
and $\mathbb{Z}_+^n:=(\mathbb{Z}_+)^n$ for any $n\in\mathbb{N}$.
Let $\mathscr M$ be the set of
all measurable functions.
For any set $K$, let $\#K$ denote its \emph{cardinality}.
The cubes $Q\subset\mathbb{R}^n$ we refer
in this article always have edges
parallel to the axes. For any cube $Q\subset\mathbb{R}^n$,
let $c_Q$ be its \emph{center} and $l(Q)$
denote its \emph{edge length}.
For any $r\in(0,\infty)$,
let $rQ$ denote the cube with the same
center as $Q$ and the edge length
$rl(Q)$.
We use $\mathbf{0}$ to
denote the \emph{origin} of ${\mathbb{R}^n}$.
For any measurable set $E$ in
$\mathbb{R}^n$ with $|E|\in(0,\infty)$ and for any
$f\in L^1_{\mathrm{loc}}$
(the set of all locally integrable functions on $\mathbb{R}^n$), let
$$
\fint_E f(x)\,dx:=\frac{1}{|E|}\int_{E}f(x)\,dx
$$
and we denote by $\mathbf{1}_E$ the
\emph{characteristic function} of $E$ and by $E^\complement$
its \emph{complementary set}.
For any $x\in\mathbb{R}^n$ and $r\in(0,\infty)$, let
$B(x,r):=\left\{y\in\mathbb{R}^n:\ |x-y|<r\right\}$
be the ball with center $x$ and radius $r$.
We use $t\to 0^+$
to denote that there exists some $c\in (0,\infty)$
such that $t\in(0,c)$ and $t\to 0$.
For any $p\in(0, \infty]$, let
$p':=\frac{p}{p-1}$ if $p\in(1, \infty]$
and $p':=\infty$ if $p\in(0, 1]$ be the \emph{conjugate index}.
We always denote by $C$ a \emph{positive constant}
which is independent of the main parameters involved,
but it may vary from line to line.
The symbol $A\lesssim B$ means that $A\le CB$ for
some positive constant $C$, while $A\sim B$ means $A\lesssim B\lesssim A$.
Also, for any $\alpha\in\mathbb{R}$, we use $\lfloor\alpha\rfloor$
(resp. $\lceil\alpha\rceil$) to denote the largest (resp. smallest) integer
not greater (resp. less) than $\alpha$.
Finally, when we prove a theorem
(and the like), in its proof we always use the same
symbols as those used in the statement itself of that theorem (and the like).

\section{Maximal Function Characterizations of $H_W^p$\label{equ-M}}

In this section, we introduce matrix-weighted Hardy spaces
and establish various maximal function characterizations of these spaces.
Let us begin with some concepts of matrices.

For any $m,n\in\mathbb{N}$, let $M_{m,n}(\mathbb{C})$ be
the set of all $m\times n$ complex-valued matrices
and we denote $M_{m,m}(\mathbb{C})$ simply by $M_{m}(\mathbb{C})$.
The zero matrix in $M_{m,n}(\mathbb{C})$ is denoted by $O_{m,n}$
and $O_{m,m}$ is simply denoted by $O_m$.
For any $A\in M_m(\mathbb{C})$, let
$\|A\|:=\sup_{\vec z\in\mathbb{C}^m,\,|\vec z|=1}|A\vec z|$.
Let $A^*$ denote the \emph{conjugate transpose} of $A$.
A matrix $A\in M_m(\mathbb{C})$ is said to be \emph{positive definite}
if, for any $\vec z\in\mathbb{C}^m\setminus\{\vec{0}\}$, $\vec z^*A\vec z>0$,
and $A$ is said to be \emph{nonnegative definite} if,
for any $\vec z\in\mathbb{C}^m$, $\vec z^*A\vec z\geq0$.
Here, and thereafter, $\vec 0$ denotes the zero vector of $\mathbb C^m$.
The matrix $A\in M_m(\mathbb C)$ is called a \emph{unitary matrix} if $A^*A=I_m$,
where $I_m$ is the identity matrix.

In the remainder of this article, we \emph{always fix} $m\in\mathbb{N}$.
Let $A\in M_m(\mathbb{C})$ be a positive definite matrix
and have eigenvalues $\{\lambda_i\}_{i=1}^m$.
Using \cite[Theorem 2.5.6(c)]{hj13},
we find that there exists a unitary matrix $U\in M_m(\mathbb{C})$ such that
\begin{equation}\label{500}
A=U\operatorname{diag}\,(\lambda_1,\ldots,\lambda_m)U^*.
\end{equation}
The following definition
can be found in \cite[(6.2.1)]{hj94}
(see also \cite[Definition 1.2]{h08}).

\begin{definition}
Let $A\in M_m(\mathbb{C})$ be a positive definite matrix
and have eigenvalues $\{\lambda_i\}_{i=1}^m$.
For any $\alpha\in\mathbb{R}$, define
$A^\alpha:=U\operatorname{diag}(\lambda_1^\alpha,\ldots,\lambda_m^\alpha)U^*,
$
where $U$ is the same as in \eqref{500}.
\end{definition}

\begin{remark}
From \cite[p.\,408]{hj94}, we infer that $A^\alpha$
is independent of the choices of the order of $\{\lambda_i\}_{i=1}^m$ and $U$,
and hence $A^\alpha$ is well defined.
\end{remark}

Now, we recall the concept of matrix weights
(see, for instance, \cite{nt96,tv97,v97}).

\begin{definition}\label{MatrixWeight}
A matrix-valued function $W:\ \mathbb{R}^n\to M_m(\mathbb{C})$ is called
a \emph{matrix weight} if $W$ is a nonnegative matrix function
with locally integrable entries such that $W(x)$ is invertible
for almost every $x\in{\mathbb{R}^n}$.
\end{definition}

It is well known that, for any given $p\in(0,\infty)$ and any given matrix weight $W$,
the matrix-weighted Lebesgue space $L^p_W$ is defined to be the
set of all $\vec f\in(\mathscr{M})^m$ such that
\begin{equation*}
\left\|\vec f\right\|_{L^p_W}:=
\left[\int_{{\mathbb{R}^n}}\left|W^{\frac1p}(x)\vec f(x)\right|^p\,dx
\right]^{\frac1p}<\infty,
\end{equation*}
which when $m=1$ and $W\equiv 1$ is exactly the Lebesgue space $L^p$.
In addition, for the Lebesgue space $L^p$ with $p=\infty$, we make the usual modification.

Next, we introduce matrix-weighted Hardy spaces.
We first give several concepts.
We denote by $\mathcal{S}$ the set of
all Schwartz functions and by
$\mathcal{S}'$ the set of
all tempered distributions.
Here, and thereafter, for any
$\alpha=(\alpha_1,\ldots,\alpha_n)\in\mathbb Z_+^n$,
let $|\alpha|:=\alpha_1+\cdots+\alpha_n.$
For any $N\in\mathbb Z_+$, let
\begin{align*}
\mathcal{S}_N:=\left\{\phi\in\mathcal{S}:\
\left\|\phi\right\|_{\mathcal{S}_N}\le1\right\},
\end{align*}
where
$\|\phi\|_{\mathcal{S}_N}:=
\sup_{\alpha\in\mathbb Z_+^n,\,|\alpha|\le N+1}
\sup_{x\in\mathbb R^n}(1+|x|)^{N+n+1}
|\partial^\alpha\phi(x)|$.
For any function $\phi$ on $\mathbb R^n$ and $t\in(0,\infty)$, define
$\phi_t(\cdot):=\frac{1}{t^n}\phi(\frac{\cdot}{t})$.
Let $a\in(0,\infty)$ and $N\in\mathbb Z_+$.
Then, for any $\vec f\in(\mathcal S')^m$,
the \emph{matrix-weighted grand non-tangential maximal function}
$(M_{a,N}^*)^p_{W}(\vec f)$ of $\vec f$ is defined by setting, for any $x\in\mathbb R^n$,
\begin{align*}
\left(M_{a,N}^*\right)^p_{W}\left(\vec f\right)(x):=
\sup_{\phi\in\mathcal{S}_N}\sup_{t\in(0,\infty)}\sup_{y\in B(x,at)}
\left|W^{\frac{1}{p}}(x)\phi_t*\vec f(y)\right|.
\end{align*}
By \cite[Theorem 2.3.20]{g14c}, we conclude that
$\phi_t*\vec f\in(C^\infty)^m$, where $C^\infty$
is the set of all infinitely differentiable functions on $\mathbb R^n$.

\begin{definition}\label{de-HW}
Let $p\in(0,\infty)$, $N\in\mathbb Z_+$, and $W$ be a matrix weight.
The \emph{matrix-weighted Hardy space} $H_{W,N}^p$
is defined to be the set of all $\vec f\in(\mathcal{S}')^m$ such that
$(M_{1,N}^*)^p_{W}(\vec f)\in L^p$,
equipped with the quasi-norm
\begin{align*}
\left\|\vec f\right\|_{H_{W,N}^p}
:=\left\|\left(M_{1,N}^*\right)^p_{W}\left(\vec f\right)\right\|_{L^p}.
\end{align*}
\end{definition}

Now, we introduce various matrix-weighted maximal functions.

\begin{definition}\label{HW}
Let $p\in(0,\infty)$ and $W$ be a matrix weight.
Let $\psi\in\mathcal{S}$, $N\in\mathbb Z_+$,
$a,l\in(0,\infty)$, and $\vec f\in(\mathcal{S}')^m$.
Then the \emph{matrix-weighted radial maximal function} $M^p_{W}(\vec f,\psi)$,
the \emph{matrix-weighted grand radial maximal function} $(M_N)^p_{W}(\vec f)$,
the \emph{matrix-weighted non-tangential maximal function} $(M^*_a)^p_{W}(\vec f,\psi)$,
the \emph{matrix-weighted maximal function
$(M^{**}_l)^p_{W}(\vec f,\psi)$ of Peetre type},
and the \emph{matrix-weighted grand maximal function
$(M^{**}_{l,N})_{W}^p(\vec f)$ of Peetre type} of $\vec f$
are defined, respectively, by setting, for any $x\in\mathbb R^n$,
\begin{align*}
M^p_{W}\left(\vec f,\psi\right)(x):=\sup_{t\in(0,\infty)}
\left|W^{\frac{1}{p}}(x)\psi_t*\vec f(x)\right|,\
\left(M_N\right)^p_{W}\left(\vec f\right)(x):=
\sup_{\phi\in\mathcal{S}_N}\sup_{t\in(0,\infty)}
\left|W^{\frac{1}{p}}(x)\phi_t*\vec f(x)\right|,
\end{align*}
\begin{align*}
\left(M^*_a\right)^p_{W}\left(\vec f,\psi\right)(x):=
\sup_{t\in(0,\infty)}\sup_{y\in B(x,at)}
\left|W^{\frac{1}{p}}(x)\psi_t*\vec f(y)\right|,
\end{align*}
\begin{align*}
\left(M^{**}_l\right)^p_{W}\left(\vec f,\psi\right)(x):=
\sup_{t\in(0,\infty)}\sup_{y\in\mathbb R^n}
\left|W^{\frac{1}{p}}(x)\psi_t*\vec f(x-y)\right|
\left(1+\frac{|y|}{t}\right)^{-l},
\end{align*}
and
\begin{align*}
\left(M^{**}_{l,N}\right)_{W}^p\left(\vec f\right)(x)
:=\sup_{\phi\in\mathcal{S}_N}\sup_{t\in(0,\infty)}\sup_{y\in\mathbb R^n}
\left|W^{\frac{1}{p}}(x)\phi_t*\vec f(x-y)\right|
\left(1+\frac{|y|}{t}\right)^{-l}.
\end{align*}
\end{definition}

To establish the equivalences of various matrix-weighted maximal functions,
we need some extra assumptions on matrix weights.
Corresponding to the scalar Muckenhoupt $A_p({\mathbb{R}^n})$ class,
Nazarov and Treil \cite{nt96} and Volberg \cite{v97} originally
independently introduced $A_p({\mathbb{R}^n},\mathbb{C}^m)$-matrix
weights with $p\in (1,\infty)$. The following version of
$A_p({\mathbb{R}^n},\mathbb{C}^m)$-matrix weights was originally introduced
by Frazier and Roudenko \cite{fr04} with $p\in (0,1]$ and
Roudenko \cite{rou03} with $p\in (1,\infty)$ (see also \cite[p.\,490]{fr21}).

\begin{definition}\label{def ap}
Let $p\in(0,\infty)$. A matrix weight $W$ on $\mathbb{R}^n$
is called an $A_p({\mathbb{R}^n},\mathbb{C}^m)$-\emph{matrix weight}
if $W$ satisfies that, when $p\in(0,1]$,
$$
[W]_{A_p({\mathbb{R}^n},\mathbb{C}^m)}
:=\sup_{\mathrm{cube}\,Q}\mathop{\mathrm{\,ess\,sup\,}}_{y\in Q}
\fint_Q\left\|W^{\frac{1}{p}}(x)W^{-\frac{1}{p}}(y)\right\|^p\,dx
<\infty
$$
or that, when $p\in(1,\infty)$,
$$
[W]_{A_p({\mathbb{R}^n},\mathbb{C}^m)}
:=\sup_{\mathrm{cube}\,Q}
\fint_Q\left[\fint_Q\left\|W^{\frac{1}{p}}(x)W^{-\frac{1}{p}}(y)\right\|^{p'}
\,dy\right]^{\frac{p}{p'}}\,dx
<\infty.
$$
\end{definition}

When $m=1$, $A_{p}({\mathbb{R}^n},\mathbb{C}^m)$-matrix weights
in this case reduce back to classical scalar Muckenhoupt
$A_{\max\{1,p\}}({\mathbb{R}^n})$ weights.
In what follows, if there exists no confusion, we denote
$A_p({\mathbb{R}^n},\mathbb{C}^m)$ simply by $A_p$.

When establishing various maximal function characterizations of
matrix-weighted Hardy spaces, we only need to impose a weaker
assumption on the matrix weight under consideration,
namely the matrix Muckenhoupt $A_{p,\infty}({\mathbb{R}^n},\mathbb{C}^m)$ class.
This concept was originally introduced by Volberg \cite[(2.2)]{v97}, and recently
Bu et al. \cite{bhyyNew} provided the following equivalent definition
(see \cite[Definition 3.1]{bhyyNew}) and some other equivalent characterizations.

\begin{definition}\label{def ap,infty}
Let $p\in(0,\infty)$. A matrix weight $W$ on $\mathbb{R}^n$
is called an $A_{p,\infty}({\mathbb{R}^n},\mathbb{C}^m)$-\emph{matrix weight}
if $W$ satisfies that, for any cube $Q\subset\mathbb{R}^n$,
$$\max\left\{0,\log\left(\fint_Q\left\|
W^{\frac{1}{p}}(x)W^{-\frac{1}{p}}(\cdot)\right\|^p
\,dx\right)\right\}\in L^1(Q)$$ and
\begin{align*}
[W]_{A_{p,\infty}({\mathbb{R}^n},\mathbb{C}^m)}
:=\sup_{\mathrm{cube}\,Q}\exp\left(\fint_Q\log\left(\fint_Q
\left\|W^{\frac{1}{p}}(x)W^{-\frac{1}{p}}(y)\right\|^p\,dx\right)\,dy\right)
<\infty.
\end{align*}
\end{definition}

When $m=1$, $A_{p,\infty}({\mathbb{R}^n},\mathbb{C}^m)$-matrix weights
in this case reduce back to the classical scalar
Muckenhoupt $A_\infty({\mathbb{R}^n})$ weights.
In what follows, if there exists no confusion, we denote
$A_{p,\infty}({\mathbb{R}^n},\mathbb{C}^m)$ simply by $A_{p,\infty}$.

Next, we recall the concept of $A_{p,\infty}$-dimensions of a matrix weight
originally from \cite[Definition 6.2]{bhyyNew}.

\begin{definition}\label{AinftyDim}
Let $p\in(0,\infty)$ and $d\in\mathbb{R}$.
A matrix weight $W$ is said to have \emph{$A_{p,\infty}$-lower dimension $d$}
if there exists a positive constant $C$ such that,
for any $\lambda\in[1,\infty)$ and any cube $Q\subset\mathbb{R}^n$,
\begin{align*}
\exp\left(\fint_{\lambda Q}\log\left(
\fint_Q\left\|W^{\frac{1}{p}}(x)W^{-\frac{1}{p}}(y)\right\|^p\,dx\right)\,dy\right)
\le C\lambda^d.
\end{align*}
A matrix weight $W$ is said to have \emph{$A_{p,\infty}$-upper dimension $d$}
if there exists a positive constant $C$ such that,
for any $\lambda\in[1,\infty)$ and any cube $Q\subset\mathbb{R}^n$,
\begin{align*}
\exp\left(\fint_Q\log\left(
\fint_{\lambda Q}\left\|W^{\frac{1}{p}}(x)W^{-\frac{1}{p}}(y)\right\|^p\,dx\right)\,dy\right)
\le C\lambda^d.
\end{align*}
\end{definition}

\begin{remark}\label{p012}
Let $p\in(0,1]$ and $W\in A_p$.
By Definitions \ref{def ap} and \ref{AinftyDim}, we conclude that
$W\in A_{p,\infty}$ and $W$ has $A_{p,\infty}$-upper dimension $0$.
\end{remark}

From \cite[Propositions 6.3(ii) and 6.4(ii)]{bhyyNew}, we deduce that,
for any matrix weight $W\in A_{p,\infty}$,
there exist $d_1\in[0,n)$ and $d_2\in[0,\infty)$ such that $W$
has $A_{p,\infty}$-lower dimension $d_1$
and $A_{p,\infty}$-upper dimension $d_2$.
For any matrix weight $W\in A_{p,\infty}$, we define
\begin{align*}
d_{p,\infty}^{\mathrm{lower}}(W) &:=\inf\{d\in[0,n):\ W\text{ has }A_{p,\infty}\text{-lower dimension }d\}, \\
d_{p,\infty}^{\mathrm{upper}}(W) &:=\inf\{d\in[0,\infty):\ W\text{ has }A_{p,\infty}\text{-upper dimension }d\},\notag \\
[\![d_{p,\infty}^{\mathrm{lower}}(W),n) &:=\begin{cases}
[d_{p,\infty}^{\mathrm{lower}}(W),n) &\text{if }W\text{ has }A_{p,\infty}\text{-lower dimension }d_{p,\infty}^{\mathrm{lower}}(W),\\
(d_{p,\infty}^{\mathrm{lower}}(W),n) &\text{otherwise}, \end{cases}\notag
\end{align*}
and
\begin{align*}
[\![d_{p,\infty}^{\mathrm{upper}}(W),\infty) &:=\begin{cases}
[d_{p,\infty}^{\mathrm{upper}}(W),\infty) &\text{if }W\text{ has }A_{p,\infty}\text{-upper dimension }d_{p,\infty}^{\mathrm{upper}}(W),\\
(d_{p,\infty}^{\mathrm{upper}}(W),\infty)&\text{otherwise}.\notag \end{cases}
\end{align*}

The following maximal function characterizations of matrix-weighted Hardy spaces
are the main result of this section.

\begin{theorem}\label{if and only if W}
Let $p\in(0,\infty)$ and $W\in A_{p,\infty}$.
Assume that $\psi\in\mathcal{S}$ satisfies $\int_{\mathbb R^n}\psi(x)\,dx\neq0$.
Let $a\in(0,\infty)$,
$l\in(n/p+[d_{p,\infty}^{\mathrm{lower}}(W)+d_{p,\infty}^{\mathrm{upper}}(W)]/p,\infty)$,
and $N\in\mathbb N$ satisfy $N> l+ d_{p,\infty}^{\mathrm{upper}}(W)/p$.
Then, for any $\vec f\in(\mathcal{S}')^m$,
\begin{align*}
\left\|\vec f\right\|_{H^p_{W,N}}
&\sim\left\|M_{W}^p\left(\vec f,\psi\right)\right\|_{L^p}
\sim\left\|\left(M^*_a\right)^p_{W}\left(\vec f,\psi\right)\right\|_{L^p}
\sim\left\|\left(M_{N}\right)^p_{W}\left(\vec f\right)\right\|_{L^p}\nonumber\\
&\sim\left\|\left(M^*_{a,N}\right)^p_{W}\left(\vec f\right)\right\|_{L^p}
\sim\left\|\left(M^{**}_l\right)^p_{W}\left(\vec f,\psi\right)\right\|_{L^p}
\sim\left\|\left(M^{**}_{l,N}\right)^p_{W}\left(\vec f\right)\right\|_{L^p}\notag,
\end{align*}
where the positive equivalence constants are independent of $\vec f$.
\end{theorem}

Based on Theorem \ref{if and only if W}, in what follows, we denote
$H_{W,N(W)}^p$ simply by $H_W^p$,
where $W\in A_{p,\infty}$ and
\begin{align}\label{N(W)}
N(W):=\left\lfloor\frac np+\frac{d_{p,\infty}^{\mathrm{lower}}(W)
+2d_{p,\infty}^{\mathrm{upper}}(W)}p\right\rfloor+1.
\end{align}

\begin{remark}
\begin{enumerate}[\rm(i)]
\item Let $m=1$, $p\in(0,\infty)$, and $W\equiv 1$.
Then, in this case, $H^p_W$ reduces back to the classical Hardy space
$H^p$ and Theorem \ref{if and only if W}
reduces back to \cite[Theorem 2.1.4]{g14m}, which establishes
various maximal function characterizations of the classical Hardy space $H^p$.

\item Let $m=1$ and $p\in(0,1]$. Then, in this case,
Theorem \ref{if and only if W} reduces back to
various maximal function characterizations of the scalar weighted
Hardy space $H^p_W$, which
is a special case of \cite[Theorem 3.1]{shyy} with
$X:=L^p_W$.
\end{enumerate}
\end{remark}

To show Theorem \ref{if and only if W}, in Subsection \ref{averaging}
we first introduce the corresponding $\mathbb A$-matrix-weighted Hardy spaces and
various $\mathbb A$-matrix-weighted maximal functions and establish their equivalences.
Then, in Subsection \ref{weighted}, we establish the relations
between the $\mathbb A$-maximal functions and the matrix-weighted maximal functions,
which further implies Theorem \ref{if and only if W}.

\subsection{$\mathbb A$-Matrix-Weighted Maximal Functions} \label{averaging}

We first recall the concept of reducing operators,
one of the most important tool in
the study of matrix weights, which was originally introduced by
Volberg in \cite[(3.1)]{v97}.

\begin{definition}\label{reduce}
Let $p\in(0,\infty)$, $W$ be a matrix weight,
and $E\subset\mathbb{R}^n$ a bounded measurable set satisfying $|E|\in(0,\infty)$.
The matrix $A_E\in M_m(\mathbb{C})$ is called a
\emph{reducing operator} of order $p$ for $W$
if $A_E$ is positive definite and,
for any $\vec z\in\mathbb{C}^m$,
\begin{equation}\label{equ_reduce}
\left|A_E\vec z\right|
\sim\left[\fint_E\left|W^{\frac{1}{p}}(x)\vec z\right|^p\,dx\right]^{\frac{1}{p}},
\end{equation}
where the positive equivalence constants depend only on $m$ and $p$.
\end{definition}

\begin{remark}\label{r2.13}
\begin{enumerate}[{\rm(i)}]
\item In Definition \ref{reduce}, the existence of $A_E$ is guaranteed by
\cite[Proposition 1.2]{Gold} and \cite[p.\,1237]{fr04}; we omit the details.
Obviously, for a given matrix weight $W$, its reducing operators are
not unique.

\item Let $m=1$ and all the other symbols be the same as in Definition \ref{reduce}.
In this case, for any bounded measurable set
$E\subset\mathbb{R}^n$ with $|E|\in(0,\infty)$,
$|A_E|\sim [\frac {W(E)}{|E|}]^{\frac 1p}$,
where $W(E):=\int_EW(x)\,dx$
and the positive equivalence constants depend only on $p$.
\end{enumerate}
\end{remark}

Based on the observation in Remark \ref{r2.13}(ii),
reducing operators can be regarded as the average of the matrix weight under
consideration. Moreover, we have the following useful equivalent
characterization of reducing operators.

\begin{proposition}\label{reduceM}
Let $p\in(0,\infty)$, $W$ be a matrix weight,
and $E\subset\mathbb{R}^n$ a bounded measurable set satisfying $|E|\in(0,\infty)$.
Then $A_E$ is a reducing operator of order $p$ for $W$
if and only if, for any matrix $M\in M_m(\mathbb{C})$,
\begin{align} \label{equ_reduceM}
\|A_EM\|\sim\left[\fint_E\left\|W^{\frac{1}{p}}(x)M\right\|^p\,dx\right]^{\frac{1}{p}},
\end{align}
where the positive equivalence constants depend only on $m$ and $p$.
\end{proposition}

\begin{proof}
The necessity of the present proposition is exactly \cite[Lemma 2.10]{bhyy1}.
Conversely, if \eqref{equ_reduceM} holds,
then, for any $\vec z\in\mathbb C^m$, taking $M:=[O_{m,m-1}\ \vec z]$,
which means the first
$m-1$ columns of
$M$ are all $\vec0$ and the
$m$th column is $\vec z$, we obtain
\eqref{equ_reduce}, that is, $A_E$ is a reducing operator of order $p$ for $W$.
This finishes the proof of the sufficiency and hence Proposition \ref{reduceM}.
\end{proof}

The following lemma is exactly \cite[Proposition 6.5]{bhyyNew}.

\begin{lemma}\label{sharp estimate}
Let $p\in(0,\infty)$, $W\in A_{p,\infty}$,
and $\{A_Q\}_{\mathrm{cube}\,Q}$ be a family of
reducing operators of order $p$ for $W$.
Assume that $d_1\in[\![d_{p,\infty}^{\mathrm{lower}}(W),n)$ and
$d_2\in[\![d_{p,\infty}^{\mathrm{upper}}(W),\infty)$.
Then there exists a positive constant $C$ such that,
for any cubes $Q,R\subset\mathbb{R}^n$,
\begin{equation}\label{0127}
\left\|A_QA_R^{-1}\right\|
\le C\max\left\{\left[\frac{l(R)}{l(Q)}\right]^{\frac{d_1}p},
\left[\frac{l(Q)}{l(R)}\right]^{\frac{d_2}p}\right\}
\left[1+\frac{|c_Q-c_R|}{\max\{l(Q),l(R)\}}\right]^{\frac{d_1+d_2}p}.
\end{equation}
\end{lemma}

The following corollary follows directly from
Lemma \ref{sharp estimate} and Remark \ref{p012}.

\begin{corollary}\label{p01}
Let $p\in(0,1]$, $W\in A_p$, and $\{A_Q\}_{\mathrm{cube}\,Q}$ be a family of
reducing operators of order $p$ for $W$.
Assume that $d\in[\![d_{p,\infty}^{\mathrm{lower}}(W),n)$.
Then there exists a positive constant $C$ such that,
for any cubes $Q,R\subset\mathbb{R}^n$,
\begin{equation*}
\left\|A_QA_R^{-1}\right\|
\le C\max\left\{\left[\frac{l(R)}{l(Q)}\right]^{\frac{d}{p}},
1\right\}
\left[1+\frac{|c_Q-c_R|}{\max\{l(Q),l(R)\}}\right]^{\frac{d}{p}}.
\end{equation*}
\end{corollary}

In some applications, it is unnecessary to require a family of matrices
$\mathbb{A}:=\{A_Q\}_{\mathrm{cube}\,Q}$ to be
the reducing operators of a matrix weight.
Instead, weaker properties, such as those estimates in \eqref{0127},
are sufficient.
In analogy to \cite[Definition 2.1]{fr21},
we give the following definition.

\begin{definition}\label{doubling}
Let $\beta_1,\beta_2,\omega\in\mathbb R$.
A family of matrices $\mathbb{A}:=\{A_Q\}_{\mathrm{cube}\,Q}$ is said to be
\begin{enumerate}[\rm(i)]
\item \emph{strongly doubling} of order $(\beta_1,\beta_2,\omega)$
if there exists a positive constant $C$ such that,
for any cubes $Q,R\subset\mathbb{R}^n$,
\begin{align*}
\left\|A_QA_R^{-1}\right\|
\le C\max\left\{\left[\frac{ l(R)}{ l(Q)}\right]^{\beta_1},
\left[\frac{ l(Q)}{ l(R)}\right]^{\beta_2}\right\}
\left[1+\frac{|c_Q-c_R|}{\max\{l(Q),l(R)\}}\right]^{\omega};
\end{align*}
\item \emph{weakly doubling} of order $\omega$
if there exists a positive constant $C$ such that,
for any cubes $Q,R\subset\mathbb{R}^n$ with $l(Q)=l(R)$,
\begin{align*}
\left\|A_QA_R^{-1}\right\|
\le C\left[1+\frac{|c_Q-c_R|}{l(Q)}\right]^\omega.
\end{align*}
\end{enumerate}
\end{definition}

\begin{remark}\label{r2.18}
Let $p\in(0,\infty)$, $W\in A_{p,\infty}$,
and $\{A_Q\}_{\mathrm{cube}\,Q}$ be a family of
reducing operators of order $p$ for $W$.
By Lemma \ref{sharp estimate} and Definition \ref{doubling}, we conclude that
$\{A_Q\}_{\mathrm{cube}\,Q}$ is strongly doubling of order $(\frac{d_1}p,\frac{d_2}p,\frac{d_1+d_2}p)$
and weakly doubling of order $\frac{d_1+d_2}p$. In this sense,
the requirements that $W\in A_{p,\infty}$ and hence the estimates satisfied by
its reducing operators
are stronger assumptions than those requirements in Definition \ref{doubling}.
\end{remark}

Now, we introduce $\mathbb A$-matrix-weighted Hardy spaces.
In what follows, for any $t\in(0,\infty)$, let
\begin{align}\label{deQt}
\mathscr{Q}_t:=\left\{t([0, 1)^n+k):\
k\in\mathbb Z^n\right\}
\end{align}
and
$\mathscr{Q}:=\bigcup_{t\in(0,\infty)}\mathscr{Q}_t$.
For a family of matrices
$\mathbb A:=\{A_Q\}_{Q\in\mathscr Q}$ and any $t\in(0,\infty)$, let
\begin{align}\label{Atref}
A_t:=\sum_{Q\in\mathscr{Q}_t}A_Q\mathbf{1}_Q.
\end{align}
Note that the dyadic version of \eqref{Atref}
[that is, $t$ in \eqref{Atref} equals to $2^j$ for any $j\in\mathbb Z$]
is already known (see, for instance, \cite[(3.8)]{bhyy1}).
Let $a\in(0,\infty)$ and $N\in\mathbb Z_+$.
Then, for any $\vec f\in(\mathcal{S}')^m$,
the \emph{$\mathbb A$-matrix-weighted grand non-tangential maximal function}
$(M_{a,N}^*)_{\mathbb A}(\vec f)$ of $\vec f$ is defined by setting,
for any $x\in\mathbb R^n$,
\begin{align*}
\left(M_{a,N}^*\right)_{\mathbb A}\left(\vec f\right)(x):=
\sup_{\phi\in\mathcal{S}_N}\sup_{t\in(0,\infty)}\sup_{y\in B(x,at)}
\left|A_{t}(x)\phi_t*\vec f(y)\right|.
\end{align*}

\begin{definition}
Let $p\in(0,\infty)$, $N\in\mathbb Z_+$, and
$\mathbb A:=\{A_Q\}_{Q\in\mathscr Q}$
be a family of positive definite matrices.
The \emph{$\mathbb A$-matrix-weighted Hardy space} $H_{\mathbb A,N}^p$
is defined to be the set of all $\vec f\in(\mathcal{S}')^m$ such that
$(M_{1,N}^*)_{\mathbb A}(\vec f)\in L^p$, equipped with the quasi-norm
$\|\vec f\|_{H_{\mathbb A,N}^p}
:=\|(M_{1,N}^*)_{\mathbb A}(\vec f)\|_{L^p}.$
\end{definition}

Next, we introduce various $\mathbb A$-matrix-weighted maximal functions.
The maximal function in Definition \ref{reducing hardy}(iv)
is new even in the scalar case.
This maximal function actually serves as a bridge connecting the
$\mathbb A$-maximal function and the corresponding
matrix-weighted maximal function, playing a key role in the proof of
the equivalences of various matrix-weighted maximal functions.

\begin{definition}\label{reducing hardy}
Let $\psi\in\mathcal{S}$, $N\in\mathbb N$, and $a,b,l\in(0,\infty)$.
Let $\mathbb A:=\{A_Q\}_{Q\in\mathscr Q}$
be a family of positive definite matrices
and $\vec f\in(\mathcal{S}')^m$.
Then the \emph{$\mathbb A$-matrix-weighted radial maximal function}
$M_{\mathbb A}(\vec f,\psi)$,
the \emph{$\mathbb A$-matrix-weighted grand radial maximal function}
$(M_N)_{\mathbb{A}}(\vec f)$,
the \emph{$\mathbb A$-matrix-weighted non-tangential maximal function}
$(M^*_a)_{\mathbb A}(\vec f,\psi)$,
the \emph{$\mathbb A$-matrix-weighted non-tangential infimum maximal function}
$(\widetilde M^*_{a,b})_{\mathbb A}(\vec f,\psi)$,
the \emph{$\mathbb A$-matrix-weighted maximal function
$(M^{**}_l)_{\mathbb A}(\vec f,\psi)$ of Peetre type},
and the \emph{$\mathbb A$-matrix-weighted grand maximal function
$(M^{**}_{l,N})_{\mathbb A}(\vec f,\psi)$ of Peetre type} of $\vec f$
are defined, respectively, by setting, for any $x\in\mathbb R^n$,
\begin{align*}
M_{\mathbb A}\left(\vec f,\psi\right)(x):=\sup_{t\in(0,\infty)}
\left|A_{t}(x)\psi_t*\vec f(x)\right|,\ \left(M_N\right)_{\mathbb{A}}\left(\vec f\right)(x):=
\sup_{\phi\in\mathcal{S}_N}\sup_{t\in(0,\infty)}
\left|A_{t}(x)\phi_t*\vec f(x)\right|,
\end{align*}
\begin{align*}
\left(M^*_a\right)_{\mathbb A}\left(\vec f,\psi\right)(x):=
\sup_{t\in(0,\infty)}\sup_{y\in B(x,at)}
\left|A_{t}(x)\psi_t*\vec f(y)\right|,
\end{align*}
\begin{align*}
\left(\widetilde M^*_{a,b}\right)_{\mathbb A}\left(\vec f,\psi\right)(x):=
\sup_{t\in(0,\infty)}
\max_{\{Q\in\mathscr{Q}_{bt}:\,Q\cap B(x,at)\neq\emptyset\}}
\inf_{y\in Q}\left|A_{t}(x)\psi_t*\vec f(y)\right|,
\end{align*}
\begin{align*}
\left(M^{**}_l\right)_{\mathbb A}\left(\vec f,\psi\right)(x):=
\sup_{t\in(0,\infty)}\sup_{y\in\mathbb R^n}
\left|A_{t}(x)\psi_t*\vec f(x-y)\right|
\left(1+\frac{|y|}{t}\right)^{-l},
\end{align*}
and
\begin{align*}
\left(M^{**}_{l,N}\right)_{\mathbb A}\left(\vec f\right)(x)
:=\sup_{\phi\in\mathcal{S}_N}\sup_{t\in(0,\infty)}\sup_{y\in\mathbb R^n}
\left|A_{t}(x)\phi_t*\vec f(x-y)\right|
\left(1+\frac{|y|}{t}\right)^{-l}.
\end{align*}
\end{definition}

The following theorem establishes the relations among
various $\mathbb A$-matrix-weighted maximal functions.

\begin{theorem}\label{equivalent A}
Let $\psi\in\mathcal{S}$ satisfy $\int_{\mathbb R^n}\psi(x)\,dx\neq0$.
Assume that $N\in\mathbb N$ and $p,a,l\in(0,\infty)$.
Let $\mathbb A:=\{A_Q\}_{Q\in\mathscr{Q}}$
be a family of matrices.
Then the following statements hold.
\begin{enumerate}[\rm(i)]
\item For any $\vec f\in(\mathcal{S}')^m$ and $x\in\mathbb R^n$,
\begin{align*}
M_{\mathbb A}\left(\vec f,\psi\right)(x)
\le\left(M^*_a\right)_{\mathbb A}\left(\vec f,\psi\right)(x)
\le(1+a)^l\left(M^{**}_l\right)_{\mathbb A}\left(\vec f,\psi\right)(x),
\end{align*}
\begin{align*}
(M_N)_{\mathbb A}\left(\vec f\right)(x)
\le\left(M^*_{a,N}\right)_{\mathbb A}\left(\vec f,\psi\right)(x)
\le(1+a)^l\left(M^{**}_{l,N}\right)_{\mathbb A}\left(\vec f\right)(x),
\end{align*}
\begin{align*}
M_{\mathbb A}\left(\vec f,\psi\right)(x)
\le\|\psi\|_{\mathcal S_N}\left(M_N\right)_{\mathbb A}\left(\vec f\right)(x),
\ \
\left(M^*_a\right)_{\mathbb A}\left(\vec f,\psi\right)(x)
\le\|\psi\|_{\mathcal S_N}\left(M^*_{a,N}\right)_{\mathbb A}\left(\vec f,\psi\right)(x),
\end{align*}
\begin{align*}
\mathrm{and}\ \left(M^{**}_l\right)_{\mathbb A}\left(\vec f,\psi\right)(x)
\le\|\psi\|_{\mathcal S_N}\left(M^{**}_{l,N}\right)_{\mathbb A}\left(\vec f\right)(x).
\end{align*}
For any $\widetilde{N}\in\mathbb N$ with $\widetilde{N}\le N$ and for
any $\vec f\in(\mathcal{S}')^m$ and $x\in\mathbb R^n$,
$(M^{**}_{l,N})_{\mathbb A}(\vec f)(x)
\le(M^{**}_{l,\widetilde{N}})_{\mathbb A}(\vec f)(x).$

\item If $\mathbb A$ is weakly doubling of order $\omega$
for some $\omega\in[0,\infty)$ and if $l\in(\omega+\frac np,\infty)$,
then there exists a positive constant $C$ such that,
for any $\vec f\in(\mathcal{S}')^m$,
\begin{align*}
\left\|\left( M^{**}_l\right)_{\mathbb A}
\left(\vec f,\psi\right)\right\|_{L^p}
\le C\left\|\left(M^*_a\right)_{\mathbb A}
\left(\vec f,\psi\right)\right\|_{L^p}.
\end{align*}

\item If $\mathbb A$ is strongly doubling of order $(\beta_1,\beta_2,\omega)$
for some $\beta_1,\beta_2,\omega\in[0,\infty)$ and if $N\geq l+\beta_2$,
then there exists a positive constant $C$ such that,
for any $\vec f\in(\mathcal{S}')^m$ and $x\in\mathbb R^n$,
\begin{align}\label{aaa}
\left(M^{**}_{l,N}\right)_{\mathbb A}\left(\vec f\right)(x)
\le C\left(M^{**}_l\right)_{\mathbb A}\left(\vec f,\psi\right)(x).
\end{align}

\item
If $\mathbb A$ is strongly doubling of order $(\beta_1,\beta_2,\omega)$
for some $\beta_1,\beta_2,\omega\in[0,\infty)$,
then there exists a positive constant $C$ such that,
for any $\vec f\in(\mathcal{S}')^m$,
$\|(M^*_a)_{\mathbb A}(\vec f,\psi)\|_{L^p}
\le C\|M_{\mathbb A}(\vec f,\psi)\|_{L^p}.$

\item If $\mathbb A$ is strongly doubling of order $(\beta_1,\beta_2,\omega)$
for some $\beta_1,\beta_2,\omega\in[0,\infty)$,
then there exist $\delta\in(0,\infty)$ and a positive constant
$C$ such that, for any $b\in(0,\delta]$
and $\vec f\in(\mathcal{S}')^m$,
$\|(M^*_a)_{\mathbb A}(\vec f,\psi)\|_{L^p}
\le C\|(\widetilde M^*_{a,b})_{\mathbb A}(\vec f,\psi)\|_{L^p}$
and, for any $b\in(0,\infty)$, $\vec f\in(\mathcal S')^m$, and $x\in\mathbb R^n$,
\begin{align}\label{zm}
\left(\widetilde M^*_{a,b}\right)_{\mathbb A}\left(\vec f,\psi\right)(x)
\le\left(M^*_a\right)_{\mathbb A}\left(\vec f,\psi\right)(x).
\end{align}
\end{enumerate}
\end{theorem}

To prove Theorem \ref{equivalent A}, we need the following several technical lemmas.
The following conclusion is a simple application of \cite[Appendix B.3]{g14m}.

\begin{lemma}\label{lucas}
Let $L\in\mathbb Z_+$, $M\in(0,\infty)$,
and $\varphi,\psi\in\mathcal{S}$.
Suppose that, for any $\alpha\in\mathbb Z_+^n$ with $|\alpha|\le L$,
$\int_{\mathbb R^n}x^\alpha\psi(x)\,dx=0$.
Then there exists a positive constant $C$ such that,
for any $s,t\in(0,\infty)$ with $s\le t$ and for any $x\in\mathbb R^n$,
$|\varphi_t*\psi_s(x)|
\le C\frac{s^{L+1}}{t^{n+L+1}(1+t^{-1}|x|)^M}.$
\end{lemma}

\begin{lemma}\label{approximation function 1}
Let $p\in(0,\infty)$ and $\mathbb A:=\{A_Q\}_{Q\in\mathscr{Q}}$
be strongly doubling of order $(\beta_1,\beta_2,\omega)$ for some $\beta_1,\beta_2,\omega\in[0,\infty)$.
Assume that $a\in(0,\infty)$ and $\psi\in\mathcal{S}$.
For any $\varepsilon,K\in[0,\infty)$,
$\vec f\in(\mathcal S')^m$, and $x\in\mathbb R^n$, let
\begin{align}\label{M* epsilon}
(M^*_a)_{\mathbb A}^{\varepsilon,K}\left(\vec f,\psi\right)(x)
:=\sup_{t\in(0,\varepsilon^{-1})}\sup_{y\in B(x,at)}
\left|A_{t}(x)\psi_t*\vec f(y)\right|
\left(\frac t{t+\varepsilon}\right)^K
\left(\frac 1{1+\varepsilon|y|}\right)^K,
\end{align}
where $0^{-1}:=\infty$ and $A_t$ for any
$t\in(0,\varepsilon^{-1})$ is the same as in \eqref{Atref}.
Then there exists $\widetilde K\in(0,\infty)$,
depending only on $p$, $\beta_1$, $\omega$, $\psi$, $\vec f$, and $n$, such that,
for any $\varepsilon\in(0,\infty)$ and $K\in(\widetilde K,\infty)$,
\begin{align*}
\left(M^*_a\right)_{\mathbb A}^{\varepsilon,K}\left(\vec f,\psi\right)
\in L^p\cap L^\infty.
\end{align*}
\end{lemma}

\begin{proof}
Let $\varepsilon\in(0,\infty)$. By Definition \ref{doubling}(i), we find that,
for any $t\in(0,\varepsilon^{-1})$ and $x\in\mathbb R^n$,
$$
\left\|A_{t}(x)\right\|
\le\left\|A_{t}(x)\left[A_{\varepsilon^{-1}}(\mathbf{0})\right]^{-1}\right\|
\left\|A_{\varepsilon^{-1}}(\mathbf{0})\right\|
\lesssim t^{-\beta_1}(1+\varepsilon|x|)^\omega,
$$
where the implicit positive constant depends on $\varepsilon$.
From the arguments used in \cite[pp.\,64--65]{g14m} via replacing
$f$ therein by $\vec f$, we infer that
there exists $L\in\mathbb Z_+$,
depending only on $\psi$ and $\vec f$, such that,
for any $t\in(0,\infty)$ and $y\in\mathbb R^n$,
\begin{align*}
\left|\psi_t*\vec f(y)\right|
\lesssim (1+\varepsilon|y|)^L\left(1+t^L\right)\left(t^{-n}+t^{-n-L}\right).
\end{align*}
These further imply that, for any $t\in(0,\varepsilon^{-1})$,
$K\in[n+L+\beta_1,\infty)$,
$x\in\mathbb R^n$, and $y\in B(x,at)$,
\begin{align}\label{qwert}
&\left|A_{t}(x)\psi_t*\vec f(y)\right|
\left(\frac t{t+\varepsilon}\right)^K
\left(\frac 1{1+\varepsilon|y|}\right)^K\\
&\quad\lesssim \left(1+\varepsilon|x|\right)^\omega
\left(t^{K-n-\beta_1}+t^{K-n-L-\beta_1}\right)
\left(\frac 1{t+\varepsilon}\right)^K
\left(\frac 1{1+\varepsilon|y|}\right)^{K-L}
\lesssim\left(1+\varepsilon|x|\right)^\omega
\left(\frac 1{1+\varepsilon|y|}\right)^{K-L}.\nonumber
\end{align}
Notice that
$1+\varepsilon|x|
\le 1+\varepsilon|y|+\varepsilon|x-y|
< 1+\varepsilon|y|+\varepsilon at
< (1+a)(1+\varepsilon|y|).$
From this, \eqref{qwert}, and the definition of
$(M^*_a)_{\mathbb A}^{\varepsilon,K}(\vec f,\psi)$, we deduce that
\begin{align*}
\left(M^*_a\right)_{\mathbb A}^{\varepsilon,K}\left(\vec f,\psi\right)(x)
\lesssim\left(\frac 1{1+\varepsilon|x|}\right)^{K-L-\omega}.
\end{align*}
Thus, for any $K\in(L+\max\{n+\beta_1,\omega+\frac np\},\infty)$,
$
(M^*_a)_{\mathbb A}^{\varepsilon,K}(\vec f,\psi)
\in L^p\cap L^\infty.
$
This finishes the proof of Lemma \ref{approximation function 1}.
\end{proof}

In the remainder of this article, we use $\mathcal{M}$ to denote
the classical \emph{Hardy--Littlewood maximal operator}, which
is defined by setting, for any $f\in L^1_{\rm loc}$ and $x\in\mathbb{R}^n$,
$\mathcal{M}(f)(x):=\sup_{x\in B}\fint_{B}|f(y)|\,dy,$
where the supremum is taken over all balls $B$ in $\mathbb{R}^n$ containing $x$.

\begin{lemma}\label{approximation function 2}
Let $p\in(0,\infty)$ and $\mathbb A:=\{A_Q\}_{Q\in\mathscr{Q}}$
be weakly doubling of order $\omega$ for some $\omega\in[0,\infty)$.
Assume that $a,l\in(0,\infty)$, $\varepsilon,K\in[0,\infty)$,
and $\psi\in\mathcal{S}$.
For any $\vec f\in(\mathcal S')^m$ and $x\in\mathbb R^n$, let
\begin{align*}
\left(M^{**}_l\right)_{\mathbb A}^{\varepsilon,K}\left(\vec f,\psi\right)(x)
:=\sup_{t\in(0,\varepsilon^{-1})}\sup_{y\in\mathbb R^n}
\left|A_{t}(x)\psi_t*\vec f(x-y)\right|
\left(1+\frac{|y|}{t}\right)^{-l}\left(\frac t{t+\varepsilon}\right)^K
\left(\frac 1{1+\varepsilon|x-y|}\right)^K,\notag
\end{align*}
where $0^{-1}:=\infty$ and $A_t$ for any
$t\in(0,\varepsilon^{-1})$ is the same as in \eqref{Atref}.
If $l\in(\omega+\frac np,\infty)$, then
there exists a positive constant $C$,
independent of $\varepsilon$ and $K$, such that,
for any $\vec f\in(\mathcal S')^m$,
\begin{align*}
\left\|\left(M^{**}_l\right)_{\mathbb A}^{\varepsilon,K}
\left(\vec f,\psi\right)\right\|_{L^p}
\le C\left\|\left(M^*_a\right)_{\mathbb A}^{\varepsilon,K}
\left(\vec f,\psi\right)\right\|_{L^p},
\end{align*}
where $(M^*_a)_{\mathbb A}^{\varepsilon,K}(\vec f,\psi)$
is the same as in \eqref{M* epsilon}.
\end{lemma}

\begin{proof}
Let $l\in(\omega+\frac np,\infty)$ and $\vec f\in(\mathcal S')^m$.
Writing $A_t(x)$ into $A_t(x)[A_t(z)]^{-1}A_t(z)$ and using
the definition of $(M^*_a)_{\mathbb A}^{\varepsilon,K}(\vec f,\psi)$
and Definition \ref{doubling}(ii), we find that,
for any $t\in(0,\infty)$, $x,y\in\mathbb R^n$, and $z\in B(x-y,at)$,
\begin{align*}
I(t,x,y)
:=&\,\left|A_{t}(x)\psi_t*\vec f(x-y)\right|
\left(\frac t{t+\varepsilon}\right)^K
\left(\frac 1{1+\varepsilon|x-y|}\right)^K\\
\le&\,\left\|A_{t}(x)
\left[A_{t}(z)\right]^{-1}\right\|
\left(M^*_a\right)_{\mathbb A}^{\varepsilon,K}\left(\vec f,\psi\right)(z)
\lesssim\left(1+\frac{|x-z|}{t}\right)^\omega
\left(M^*_a\right)_{\mathbb A}^{\varepsilon,K}\left(\vec f,\psi\right)(z),
\end{align*}
which, together with $|x-z|<|y|+at$, further implies that
\begin{align*}
I(t,x,y)
\lesssim\left(1+\frac{|y|+at}t\right)^\omega
\left(M^*_a\right)_{\mathbb A}^{\varepsilon,K}\left(\vec f,\psi\right)(z)
\sim\left(1+\frac{|y|}t\right)^\omega
\left(M^*_a\right)_{\mathbb A}^{\varepsilon,K}\left(\vec f,\psi\right)(z).
\end{align*}
Now, taking the average with respect to $z$
over $B(x-y,at)$ on both sides of the above inequality,
we further obtain, for any $t\in(0,\infty)$ and $x,y\in\mathbb R^n$,
\begin{align*}
[I(t,x,y)]^{\frac n{l-\omega}}
&\lesssim\left(1+\frac{|y|}t\right)^{\frac{n\omega}{l-\omega}}
\fint_{B(x-y,at)}\left[(M^*_a)_{\mathbb A}^{\varepsilon,K}
\left(\vec f,\psi\right)(z)\right]^{\frac n{l-\omega}}\,dz\\
&\le\left(1+\frac{|y|}t\right)^{\frac{n\omega}{l-\omega}}
\left(\frac{|y|+at}{at}\right)^n\fint_{B(x,|y|+at)}
\left[(M^*_a)_{\mathbb A}^{\varepsilon,K}
\left(\vec f,\psi\right)(z)\right]^{\frac n{l-\omega}}\,dz\\
&\lesssim\left(1+\frac{|y|}t\right)^{\frac{nl}{l-\omega}}
\mathcal{M}\left(\left[(M^*_a)_{\mathbb A}^{\varepsilon,K}\left(\vec f,\psi\right)
\right]^{\frac n{l-\omega}}\right)(x).
\end{align*}
Thus, from the definition of $(M^{**}_l)_{\mathbb A}^{\varepsilon,K}$,
we infer that, for any $x\in\mathbb R^n$,
\begin{align*}
\left(M^{**}_l\right)_{\mathbb A}^{\varepsilon,K}\left(\vec f,\psi\right)(x)
\lesssim \left[\mathcal{M}\left(\left[(M^*_a)_{\mathbb A}^{\varepsilon,K}
\left(\vec f,\psi\right)\right]^{\frac n{l-\omega}}\right)(x)\right]^{\frac{l-\omega}n}.
\end{align*}
Taking the $L^p$ norm of both side of the above inequality and using
$l\in(\omega+\frac np,\infty)$ and the boundedness of $\mathcal M$
on $L^{\frac{l-\omega}np}$, we find that
\begin{align*}
\left\|\left(M^{**}_l\right)^{\varepsilon,K}
_{\mathbb A}\left(\vec f,\psi\right)\right\|_{L^p}^p
\lesssim\int_{\mathbb R^n}
\left[\mathcal{M}\left(\left[(M^*_a)_{\mathbb A}^{\varepsilon,K}
\left(\vec f,\psi\right)\right]^{\frac n{l-\omega}}
\right)(x)\right]^{\frac{l-\omega}np}\,dx\lesssim\int_{\mathbb R^n}
\left[(M^*_a)_{\mathbb A}^{\varepsilon,K}
\left(\vec f,\psi\right)(x)\right]^p\,dx.
\end{align*}
This finishes the proof of Lemma \ref{approximation function 2}.
\end{proof}

For any $\vec f:=(f_1,\ldots,f_m)^T\in(C^\infty)^m$, let
\begin{equation*}
\nabla \vec f:=\left[\begin{matrix}
\partial_1 f_1&\cdots&\partial_nf_1\\
\vdots&\ddots&\vdots\\
\partial_1 f_m&\cdots&\partial_nf_m
\end{matrix}\right].
\end{equation*}

\begin{lemma}\label{approximation function 3}
Let $\psi\in\mathcal{S}$ satisfy $\int_{\mathbb R^n}\psi(x)\,dx\neq0$.
Assume that $p\in(0,\infty)$ and $\mathbb A:=\{A_Q\}_{Q\in\mathscr{Q}}$
is strongly doubling of order $(\beta_1,\beta_2,\omega)$ for some $\beta_1,\beta_2,\omega\in[0,\infty)$.
Let $a,l\in(0,\infty)$ and $\varepsilon,K\in[0,\infty)$.
For any $\vec f\in(\mathcal S')^m$ and $x\in\mathbb R^n$, define
\begin{align*}
\left(U^*_a\right)_{\mathbb A}^{\varepsilon,K}\left(\vec f,\psi\right)(x)
:=\sup_{t\in(0,\varepsilon^{-1})}\sup_{y\in B(x,at)}
t\left\|A_{t}(x)
\nabla\left(\psi_t*\vec f\right)(y)\right\|\left(\frac t{t+\varepsilon}\right)^K
\left(\frac 1{1+\varepsilon|y|}\right)^K,\notag
\end{align*}
where $0^{-1}:=\infty$ and $A_t$ for any
$t\in(0,\varepsilon^{-1})$ is the same as in \eqref{Atref}.
Then there exists a positive constant $C$, independent of $\varepsilon$, such that,
for any $\vec f\in(\mathcal S')^m$ and $x\in\mathbb R^n$,
\begin{align*}
\left(U^*_a\right)_{\mathbb A}^{\varepsilon,K}\left(\vec f,\psi\right)(x)
\le C\left(M^{**}_l\right)_{\mathbb A}^{\varepsilon,K}\left(\vec f,\psi\right)(x).
\end{align*}
\end{lemma}

\begin{proof}
Let $\vec f\in(\mathcal S')^m$. By the facts that
$A_{t}(x)\nabla(\psi_t*\vec f)(y)=\nabla [A_{t}(x)(\psi_t*\vec f)](y)$
and that all the norms in a given finite dimensional vector space are equivalent,
we conclude that, for any $t\in(0,\infty)$ and $x,y\in\mathbb R^n$,
\begin{align}\label{delta-psi}
t\left\|A_{t}(x)
\nabla\left(\psi_t*\vec f\right)(y)\right\|
&\sim\sum_{i=1}^m\sum_{j=1}^n
t\left|\partial_j\left[A_{t}(x)
\psi_t*\vec f\right]_i(y)\right|\\
&=\sum_{i=1}^m\sum_{j=1}^n
\left|\left[A_{t}(x)\left(\partial_j\psi\right)_t*\vec f\right]_i(y)\right|
\sim\sum_{j=1}^n
\left|A_{t}(x)\left(\partial_j\psi\right)_t*\vec f(y)\right|.\notag
\end{align}

From \cite[Lemma 2.1.5]{g14m} with $\Psi$, $\Phi$, and $m$
replaced, respectively, by $\partial_j\psi$, $\psi$,
and $L_0:=\lfloor l+1\rfloor+\lceil\beta_2\rceil+K$ and
from $\psi\in\mathcal{S}$,
we deduce that there exist
$\{\Theta^{(s)}\}_{s\in(0,1]}\subset \mathcal{S}$ such that,
for any $j\in\{1,\ldots,n\}$ and $x\in\mathbb R^n$,
\begin{align}\label{lucas1 alter}
\partial_j\psi(x)=\int_0^1\left[\Theta^{(s)}*\psi_s\right](x)\,ds
\end{align}
and, for any $s\in(0,1]$,
\begin{align}\label{lucas1 ineq}
\int_{\mathbb R^n}(1+|z|)^{L_0}\left|\Theta^{(s)}(z)\right|\,dz
\lesssim s^{L_0}.
\end{align}

Using \eqref{lucas1 alter} and the definition of
$(M^{**}_{\lfloor l+1\rfloor})_{\mathbb A}^{\varepsilon,K}
(\vec f,\psi)$ and writing $A_t(x)$ as $A_t(x)[A_{st}(x)]^{-1}
A_{st}(x)$, we find that,
for any $j\in\{1,\ldots,n\}$,
$t\in(0,\varepsilon^{-1})$,
$x\in\mathbb R^n$, and $y\in B(x,at)$,
\begin{align*}
I(t,j,x,y)
:=&\,\left|A_{t}(x)
\left(\partial_j\psi\right)_t*\vec f(y)\right|
\left(\frac t{t+\varepsilon}\right)^K
\left(\frac 1{1+\varepsilon|y|}\right)^K\\
\le&\,\int_0^1\int_{\mathbb R^n}\left|\left[\Theta^{(s)}\right]_t(z)\right|
\left|A_{t}(x)\psi_{st}*\vec f(y-z)\right|\\
&\quad\times\left(\frac{st}{st+\varepsilon}\right)^K s^{-K}
\left(\frac 1{1+\varepsilon|y-z|}\right)^K
\left(\frac{1+\varepsilon|y-z|}{1+\varepsilon|y|}\right)^K\,dz\,ds\\
\le&\,\int_0^1\int_{\mathbb R^n}\left|\left[\Theta^{(s)}\right]_t(z)\right|
\left\|A_{t}(x)\left[A_{st}(x)\right]^{-1}\right\|\\
&\quad\times\left(M^{**}_{\lfloor l+1\rfloor}\right)_{\mathbb A}^{\varepsilon,K}
\left(\vec f,\psi\right)(x)
\left[1+\frac{|x-y+z|}{st}\right]^{\lfloor l+1\rfloor}
s^{-K} (1+\varepsilon|z|)^K \,dz\,ds,
\end{align*}
which, together with Definition \ref{doubling}(i) and
\eqref{lucas1 ineq}, further implies that
\begin{align*}
I(t,j,x,y)
&\lesssim
\left(M^{**}_{\lfloor l+1\rfloor}\right)_{\mathbb A}^{\varepsilon,K}
\left(\vec f,\psi\right)(x)
\int_0^1 s^{-\lfloor l+1\rfloor-\beta_2-K}\\
&\quad\times\int_{\mathbb R^n}\left|\left[\Theta^{(s)}\right]_t(z)\right|
\left(1+\frac{|x-y|}{t}+\frac{|z|}{t}\right)^{\lfloor l+1\rfloor}
\left(1+\varepsilon t\frac{|z|}t\right)^K\,dz\,ds\\
&\le
\left(M^{**}_{\lfloor l+1\rfloor}\right)_{\mathbb A}^{\varepsilon,K}
\left(\vec f,\psi\right)(x)
\int_0^1 s^{-L_0}\int_{\mathbb R^n}\left|\left[\Theta^{(s)}\right]_t(z)\right|
\left(1+a+\frac{|z|}{t}\right)^{\lfloor l+1\rfloor}
\left(1+\frac{|z|}t\right)^K\,dz\,ds\\
&\lesssim
\left(M^{**}_{\lfloor l+1\rfloor}\right)_{\mathbb A}^{\varepsilon,K}
\left(\vec f,\psi\right)(x)
\int_0^1 s^{-L_0}
\int_{\mathbb R^n}\left|\Theta^{(s)}(z)\right|
(1+|z|)^{L_0}\,dz\,ds\lesssim\left(M^{**}_{\lfloor l+1\rfloor}\right)_{\mathbb A}^{\varepsilon,K}
\left(\vec f,\psi\right)(x).
\end{align*}
Taking the supremum with respect to
$y\in\mathbb{R}^n$ on both sides of the above inequality and using
the definition of $(U^*_a)_{\mathbb A}^{\varepsilon,K}(\vec f,\psi)$
and $(M^{**}_{l})_{\mathbb A}^{\varepsilon,K}
(\vec f,\psi)$, together with \eqref{delta-psi},
we conclude that, for any $x\in\mathbb R^n$,
\begin{align*}
\left(U^*_a\right)_{\mathbb A}^{\varepsilon,K}\left(\vec f,\psi\right)(x)
\lesssim\left(M^{**}_{\lfloor l+1\rfloor}\right)_{\mathbb A}^{\varepsilon,K}
\left(\vec f,\psi\right)(x)
\le\left(M^{**}_l\right)_{\mathbb A}^{\varepsilon,K}\left(\vec f,\psi\right)(x).
\end{align*}
This finishes the proof of Lemma \ref{approximation function 3}.
\end{proof}

The following lemma provides the largest ball
contained in the intersection region of two balls,
which is used in the proof of Theorem \ref{equivalent A}(v).

\begin{lemma} \label{ball}
Let $r,\delta\in(0,\infty)$ and $x,y\in\mathbb R^n$ satisfy $|x-y|<(1+\delta)r$.
Then the following statements hold.
\begin{enumerate}[\rm(i)]
\item There exists $z\in\mathbb R^n$ such that
$B(z,r^*)\subset [B(x,r)\cap B(y,\delta r)]$, where
$$
r^*:=\frac{(1+\delta)r-\max\{|x-y|,|1-\delta|r\}}2.
$$

\item If we further assume $y\in B(x,r)$, then
\begin{align}\label{ball subset}
B(z,C_\delta r)\subset \left[B(x,r)\cap B(y,\delta r)\right]
\end{align}
and
\begin{align}\label{ball est}
|B(x,r)\cap B(y,\delta r)|
\geq (C_\delta)^n |B(x,r)|,
\end{align}
where $C_\delta:=\min\{\frac\delta 2,1\}$,
\end{enumerate}
\end{lemma}

\begin{proof}
Statement (ii) follows directly from (i) and hence we only need to show (i).
To this end, we consider the following two cases on $|x-y|$.

\emph{Case 1)} $|x-y|\in[0,|1-\delta|r]$.
In this case, $r^*=\min\{r,\delta r\}$ and
either of the two balls is contained within the other.
Therefore, (i) holds in this case.

\emph{Case 2)} $|x-y|\in(|1-\delta|r,(1+\delta)r)$.
In this case, $2r^*=(1+\delta)r-|x-y|$ and the two balls intersect.
By some geometrical observations, we find that $r^*$ is the radius
of the largest ball contained in the intersection region
of two balls under consideration, which further implies (i) in this case.
This finishes the proof of Lemma \ref{ball}.
\end{proof}

Next, we are able to prove Theorem \ref{equivalent A} by borrowing some ideas from
\cite[Theorem 2.1.4]{g14m}.

\begin{proof}[Proof of Theorem \ref{equivalent A}]
Assertion (i) follows directly from Definition \ref{reducing hardy}
and assertion (ii) is exactly Lemma \ref{approximation function 2}
with $\varepsilon$ and $K$ replaced by $0$.

Now, we show (iii).
Let $\vec f\in(\mathcal S')^m$
and $N\in\mathbb N$ satisfy $N\geq l+\beta_2$.
Applying \cite[Lemma 2.1.5]{g14m} with $\Psi$,  $\Phi$, and $m$
replaced, respectively, by $\phi$, $\psi$, and $N$,
we find that, for any $\phi\in\mathcal S_N$, there exist
$\{\Theta^{(s)}\}_{s\in(0,1]}\subset \mathcal{S}$ such that,
for any $x\in\mathbb R^n$,
\begin{align}\label{lucas1}
\phi(x)=\int_0^1\left[\Theta^{(s)}*\psi_s\right](x)\,ds
\end{align}
and, for any $s\in(0,1]$,
\begin{align*}
I(s):=\int_{\mathbb R^n}(1+|z|)^N\left|\Theta^{(s)}(z)\right|\,dz
\lesssim s^N\int_{\mathbb R^n}(1+|z|)^N
\sum_{\alpha\in\mathbb Z_+^n,\, |\alpha|\le N+1}
|\partial^\alpha\phi(z)|\,dz,
\end{align*}
where the implicit positive constant is independent of $\phi$.
By this and $\phi\in\mathcal S_N$, we obtain, for any $s\in(0,1]$,
\begin{align}\label{lucas2}
I(s)\lesssim s^N\int_{\mathbb R^n}(1+|z|)^{-(n+1)}\|\phi\|_{\mathcal S_N}\,dz
\lesssim s^N.
\end{align}
From \eqref{lucas1} via writing $A_t(x)$ as
$A_t(x)[A_{st}(x)]^{-1}A_{st}(x)$ and from the definition of
$(M^{**}_l)_{\mathbb A}(\vec f,\psi)$
and Definition \ref{doubling}(i), it follows that,
for any $t\in(0,\infty)$ and $x,y\in\mathbb R^n$,
\begin{align*}
\left|A_{t}(x)\phi_t*\vec f(x-y)\right|
&\le\int_0^1\int_{\mathbb R^n}\left|\left[\Theta^{(s)}\right]_t(z)\right|
\left|A_{t}(x)\psi_{st}*\vec f(x-y-z)\right|\,dz\,ds\\
&\le\int_0^1\int_{\mathbb R^n}\left|\left[\Theta^{(s)}\right]_t(z)\right|
\left\|A_{t}(x)\left[A_{st}(x)\right]^{-1}\right\|\left(M^{**}_l\right)_{\mathbb A}\left(\vec f,\psi\right)(x)
\left(1+\frac{|y+z|}{st}\right)^l\,dz\,ds\\
&\lesssim
\left(M^{**}_l\right)_{\mathbb A}\left(\vec f,\psi\right)(x)
\int_0^1 s^{-l-\beta_2}
\int_{\mathbb R^n}\left|\left[\Theta^{(s)}\right]_t(z)\right|
\left(1+\frac{|y|}{t}+\frac{|z|}{t}\right)^l\,dz\,ds\\
&\le
\left(M^{**}_l\right)_{\mathbb A}\left(\vec f,\psi\right)(x)
\left(1+\frac{|y|}{t}\right)^l
\int_0^1 s^{-N}\int_{\mathbb R^n}\left|\Theta^{(s)}(z)\right|
(1+|z|)^N\,dz\,ds,
\end{align*}
where $A_t$ is the same as in \eqref{Atref}.
Using this and \eqref{lucas2}, we conclude that,
for any $t\in(0,\infty)$ and $x,y\in\mathbb R^n$,
\begin{align*}
\left|A_{t}(x)\phi_t*\vec f(x-y)\right|\left(1+\frac{|y|}{t}\right)^{-l}
\lesssim\left(M^{**}_l\right)_{\mathbb A}\left(\vec f,\psi\right)(x).
\end{align*}
Taking the supremum with respect to $y\in\mathbb{R}^n$ in left-hand side of the above inequality
and using the definition of
$(M^{**}_{l,N})_{\mathbb A}(\vec f)$,
we conclude that \eqref{aaa} holds.
This finishes the proof of (iii).

Next, we prove (iv) and (v) together because both proofs are similar.
Inequality \eqref{zm} follows immediately from the definitions of
$(M^*_a)_{\mathbb A}(\vec f,\psi)$ and
$(\widetilde M^*_{a,b})_{\mathbb A}(\vec f,\psi)$
and hence we focus on the others.
If $M_{\mathbb A}(\vec f,\psi)\notin L^p$, then (iv) obviously holds
and, if $(\widetilde M^*_{a,b})_{\mathbb A}(\vec f,\psi)\notin L^p$,
then (v) obviously holds.
Now, we consider the case where
\begin{align}\label{biaohao}
M_{\mathbb A}(\vec f,\psi)\in L^p
\ \ \text{or}\ \ (\widetilde M^*_{a,b})_{\mathbb A}(\vec f,\psi)\in L^p.
\end{align}
Let $\varepsilon,K\in[0,\infty)$ and
$\vec f\in(\mathcal S')^m$.
By Lemmas \ref{approximation function 2} and \ref{approximation function 3},
we find that there exists a positive constant $C_1$,
independent of $\varepsilon$ and $\vec f$, such that
\begin{align}\label{U<M}
\left\|\left(U^*_a\right)_{\mathbb A}^{\varepsilon,K}
\left(\vec f,\psi\right)\right\|_{L^p}
\le C_1\left\|\left(M^*_a\right)_{\mathbb A}^{\varepsilon,K}
\left(\vec f,\psi\right)\right\|_{L^p}.
\end{align}
Let
\begin{align*}
E_\varepsilon
:=\left\{x\in\mathbb R^n:\
(U^*_a)_{\mathbb A}^{\varepsilon,K}\left(\vec f,\psi\right)(x)
\le 2^{\frac1p}C_1(M^*_a)_{\mathbb A}^{\varepsilon,K}\left(\vec f,\psi\right)(x)\right\}.
\end{align*}
From \eqref{U<M}, we infer that
\begin{align}\label{E epsilon C}
\int_{(E_\varepsilon)^\complement}
\left[(M^*_a)_{\mathbb A}^{\varepsilon,K}\left(\vec f,\psi\right)(x)\right]^p\,dx
&\le \frac12C_1^{-p}\int_{(E_\varepsilon)^\complement}
\left[(U^*_a)_{\mathbb A}^{\varepsilon,K}\left(\vec f,\psi\right)(x)\right]^p\,dx\\
&\le \frac12\int_{(E_\varepsilon)^\complement}
\left[(M^*_a)_{\mathbb A}^{\varepsilon,K}\left(\vec f,\psi\right)(x)\right]^p\,dx.\notag
\end{align}

Next, we show that
\begin{align}\label{C3add}
&\int_{E_\varepsilon}\left[(M^*_a)_{\mathbb A}^{\varepsilon,K}\left(\vec f,\psi\right)(x)\right]^p\,dx\\
&\quad\lesssim\min\left\{
\int_{\mathbb R^n}\left[M_{\mathbb A}\left(\vec f,\psi\right)(x)
\right]^p\,dx,
\int_{E_\varepsilon}\left[\left(\widetilde M^*_{a,b}\right)_{\mathbb A}\left(\vec f,\psi\right)(x)
\right]^p\,dx
\right\}=:\text{M}_{\min}.\notag
\end{align}
To this end, we first consider the part related to $M_{\mathbb A}(\vec f,\psi)$.
Let $x\in E_\varepsilon$.
Indeed, using the definition of $(M^*_a)_{\mathbb A}^{\varepsilon,K}$,
we conclude that there exist $t_x\in(0,\varepsilon^{-1})$
and $y_x\in B(x,at_x)$ such that
\begin{align}\label{2023.11.22}
\frac12 \left(M^*_a\right)_{\mathbb A}^{\varepsilon,K}\left(\vec f,\psi\right)(x)
\le\left|A_{t_x}(x)\psi_{t_x}*\vec f(y_x)\right|
\left(\frac{t_x}{t_x+\varepsilon}\right)^K
\left(\frac 1{1+\varepsilon|y_x|}\right)^K.
\end{align}
For any $\xi\in\mathbb R^n$, let
$J(\xi):=t_x\|A_{t_x}(x)\nabla(\psi_{t_x}*\vec f)(\xi)\|$.
From the definitions of $(U^*_a)_{\mathbb A}^{\varepsilon,K}$ and $E_\varepsilon$,
we deduce that, for any $\xi\in B(x,at_x)$,
\begin{align*}
J(\xi)
\left(\frac{t_x}{t_x+\varepsilon}\right)^K
\left(\frac 1{1+\varepsilon|\xi|}\right)^K
\le (U^*_a)_{\mathbb A}^{\varepsilon,K}\left(\vec f,\psi\right)(x)
\le 2^{\frac1p}C_1(M^*_a)_{\mathbb A}^{\varepsilon,K}\left(\vec f,\psi\right)(x),
\end{align*}
which, together with \eqref{2023.11.22}, further implies that
\begin{align*}
J(\xi)
\le 2^{1+\frac1p}C_1\left|A_{t_x}(x)\psi_{t_x}*\vec f(y_x)\right|
\left(\frac{1+\varepsilon|\xi|}{1+\varepsilon|y_x|}\right)^K.
\end{align*}
By this and
$
(1+\varepsilon|\xi|)/(1+\varepsilon|y_x|)
\le 1+\varepsilon|\xi-y_x|
< 1+2\varepsilon at_x
< 1+2a,
$
we find that
\begin{align*}
J(\xi)
\le 2^{1+\frac1p}(1+2a)^KC_1
\left|A_{t_x}(x)\psi_{t_x}*\vec f(y_x)\right|.
\end{align*}
Using this, the mean value theorem, and the Cauchy--Schwartz inequality,
we obtain, for any $y\in B(x,at_x)$,
\begin{align}\label{imp}
&\left|\,\left|A_{t_x}(x)\psi_{t_x}*\vec f(y)\right|
-\left|A_{t_x}(x)\psi_{t_x}*\vec f(y_x)\right|\,\right|\\
&\quad\le\left|A_{t_x}(x)\psi_{t_x}*\vec f(y)
-A_{t_x}(x)\psi_{t_x}*\vec f(y_x)\right|\nonumber\\
&\quad\sim\sum_{i=1}^m\left|\left[A_{t_x}(x)\psi_{t_x}*\vec f(y)
-A_{t_x}(x)\psi_{t_x}*\vec f(y_x)\right]_i\right|
\le\sum_{i=1}^m\left|\nabla\left(\left[A_{t_x}(x)
\psi_{t_x}*\vec f\right]_i\right)(\xi_i)\right||y-y_x|\notag\\
&\quad\le\sum_{i=1}^m\left\|A_{t_x}(x)
\nabla\left(\psi_{t_x}*\vec f\right)(\xi_i)\right\||y-y_x|
\lesssim
\left|A_{t_x}(x)\psi_{t_x}*\vec f(y_x)\right|\frac{|y-y_x|}{t_x},\notag
\end{align}
where the implicit positive constants are independent of $\varepsilon$ and $\vec f$ and,
for any $i\in\{1,\ldots,m\}$, $\xi_i=\theta_iy+(1-\theta_i)y_x$ for some $\theta_i\in[0,1]$.
Let $C_2$ be the implicit positive constant in \eqref{imp}.
Then, for any $y\in B(x,at_x)\cap B(y_x,(2C_2)^{-1}t_x)$,
\begin{align*}
\left|\,\left|A_{t_x}(x)\psi_{t_x}*\vec f(y)\right|
-\left|A_{t_x}(x)\psi_{t_x}*\vec f(y_x)\right|\,\right|
\le\frac12\left|A_{t_x}(x)\psi_{t_x}*\vec f(y_x)\right|
\end{align*}
and hence
$|A_{t_x}(x)\psi_{t_x}*\vec f(y)|
\geq\frac12|A_{t_x}(x)\psi_{t_x}*\vec f(y_x)|$,
which, together with \eqref{2023.11.22}, further implies that
\begin{align}\label{chao}
\left|A_{t_x}(x)\psi_{t_x}*\vec f(y)\right|
\geq\frac14(M^*_a)_{\mathbb A}^{\varepsilon,K}\left(\vec f,\psi\right)(x).
\end{align}
From this, the definition of $M_{\mathbb A}(\vec f,\psi)$, and
\eqref{ball est}, it follows that, for any $q\in(0,\infty)$,
\begin{align*}
\mathcal M\left(\left[M_{\mathbb A}\left(\vec f,\psi\right)
\right]^q\right)(x)
&\geq\frac{1}{|B(x,at_x)|}
\int_{B(x,at_x)\cap B(y_x,(2C_2)^{-1}t_x)}\left[M_{\mathbb A}
\left(\vec f,\psi\right)(y)\right]^q\,dy\\
&\geq\frac{1}{|B(x,at_x)|}
\int_{B(x,at_x)\cap B(y_x,(2C_2)^{-1}t_x)}\frac{1}{4^q}\left[(M^*_a)_{\mathbb A}
^{\varepsilon,K}\left(\vec f,\psi\right)(x)\right]^q\,dy\nonumber\\
&\gtrsim\left[(M^*_a)_{\mathbb A}^{\varepsilon,K}
\left(\vec f,\psi\right)(x)\right]^q.\nonumber
\end{align*}
Let $q\in(0,p)$. Taking the $L^\frac pq$ norm on both
sides of the above inequality and using
the boundedness of $\mathcal{M}$ on $L^{\frac pq}$, we obtain
\begin{align*}
\int_{E_\varepsilon}\left[(M^*_a)_{\mathbb A}^{\varepsilon,K}\left(\vec f,\psi\right)(x)\right]^p\,dx
&\lesssim
\int_{\mathbb R^n}\left[\mathcal M\left(\left[M_{\mathbb A}\left(\vec f,\psi\right)
\right]^q\right)(x)\right]^{\frac pq}\,dx
\lesssim
\int_{\mathbb R^n}\left[M_{\mathbb A}\left(\vec f,\psi\right)(x)
\right]^p\,dx.
\end{align*}
So far we have proved one part of \eqref{C3add}, next we focus on the other part.
Still let $t_x$ and $y_x$ be the same as in \eqref{2023.11.22},
$C_2$ as in \eqref{imp}, and $x\in E_\varepsilon$.
From \eqref{ball subset}, we infer that there exists $z\in\mathbb R^n$ such that
$$
B\left(z,\min\left\{(4C_2)^{-1},a\right\}t_x\right)
\subset \left[B(x,at_x)\cap B\left(y_x,(2C_2)^{-1} t_x\right)\right].
$$
Let $\delta:=n^{-\frac12}\min\{(4C_2)^{-1},a\}$.
Then it is easy to show that, for any $b\in(0,\delta]$, there exists $Q\in\mathscr{Q}_{bt_x}$
such that $z\in Q$ and hence
$$
Q
\subset B\left(z,\min\left\{(4C_2)^{-1},a\right\}t_x\right)
\subset \left[B(x,at_x)\cap B\left(y_x,(2C_2)^{-1} t_x\right)\right].
$$
This, together with the definition of $(\widetilde M^*_{a,b})_{\mathbb A}(\vec f,\psi)$
and \eqref{chao}, further implies that
$(\widetilde M^*_{a,b})_{\mathbb A}(\vec f,\psi)(x)
\geq\frac14(M^*_a)_{\mathbb A}^{\varepsilon,K}(\vec f,\psi)(x),$
and hence
\begin{align*}
\int_{E_\varepsilon}\left[(M^*_a)_{\mathbb A}^{\varepsilon,K}\left(\vec f,\psi\right)(x)\right]^p\,dx
\leq4^p
\int_{E_\varepsilon}\left[\left(\widetilde M^*_{a,b}\right)_{\mathbb A}\left(\vec f,\psi\right)(x)
\right]^p\,dx.
\end{align*}
This finishes the proof of the other part of \eqref{C3add} and hence \eqref{C3add} holds.

Let $C_3$ be the implicit positive constant in \eqref{C3add} and notice
that $C_3$ is independent of $\varepsilon$ and $\vec f$.
Combining \eqref{E epsilon C} and \eqref{C3add}, we conclude that
\begin{align}\label{important}
\left\|(M^*_a)_{\mathbb A}^{\varepsilon,K}\left(\vec f,\psi\right)\right\|_{L^p}^p
\le C_3\text{M}_{\min}
+\frac12\left\|(M^*_a)_{\mathbb A}^{\varepsilon,K}\left(\vec f,\psi\right)\right\|_{L^p}^p.
\end{align}
By this and Lemma \ref{approximation function 1}, we find that
there exists $\widetilde K\in(0,\infty)$, depending on $\vec f$, such that,
for any $\varepsilon\in(0,\infty)$ and $K\in(\widetilde K,\infty)$,
$(M^*_a)_{\mathbb A}^{\varepsilon,K}(\vec f,\psi)\in L^p$ and hence
\begin{align}\label{MaAMA}
\left\|(M^*_a)_{\mathbb A}^{\varepsilon,K}\left(\vec f,\psi\right)\right\|_{L^p}
\le 2^{\frac1p}C_3^{\frac1p}\text{M}_{\min}.
\end{align}
Notice that, for any $t\in(0,\varepsilon^{-1})$ and $y\in B(x,at)$,
$\frac {1+\varepsilon|x|}{1+\varepsilon|y|}
\geq\frac 1{1+\varepsilon|x-y|}
>\frac 1{1+\varepsilon at}
>\frac 1{1+a}$
and, consequently, from the definition of
$(M^*_a)_{\mathbb A}^{\varepsilon,K}(\vec f,\psi)$, we deduce that,
for any $x\in\mathbb{R}^n$,
\begin{align*}
\left(M^*_a\right)_{\mathbb A}^{\varepsilon,K}\left(\vec f,\psi\right)(x)
\geq\sup_{t\in(0,\varepsilon^{-1})}\sup_{y\in B(x,at)}
\left|A_{t}(x)\psi_t*\vec f(y)\right|
\left(\frac t{t+\varepsilon}\right)^K
\frac {(1+a)^{-K}}{(1+\varepsilon|x|)^K}.\notag
\end{align*}
Using this and the definition of $(M^*_a)_{\mathbb A}^{\varepsilon,K}(\vec f,\psi)$
and letting $\varepsilon\to 0^+$, we obtain,
for any $x\in\mathbb{R}^n$,
\begin{align*}
\liminf_{\varepsilon\to 0^+}
(M^*_a)_{\mathbb A}^{\varepsilon,K}\left(\vec f,\psi\right)(x)
\geq(1+a)^{-K}(M^*_a)_{\mathbb A}\left(\vec f,\psi\right)(x).
\end{align*}
Taking the $L^p$ norm on both sides of the above inequality
and applying the Fatou lemma and \eqref{MaAMA},
we find that, for any $K\in(\widetilde K,\infty)$,
\begin{align}\label{2N+1p}
\left\|(M^*_a)_{\mathbb A}\left(\vec f,\psi\right)\right\|_{L^p}
\le (1+a)^{K}2^{\frac1p}C_3^{\frac1p}\text{M}_{\min}.
\end{align}
This is close to what we need, but $(1+a)^K$ depends on $K$ and hence on $\vec f$.
However, by \eqref{biaohao} and \eqref{2N+1p}, we obtain an important property that
$(M^*_a)_{\mathbb A}(\vec f,\psi)\in L^p$.
Applying \eqref{important} with $\varepsilon=0$, we conclude that
\begin{align*}
\left\|(M^*_a)_{\mathbb A}\left(\vec f,\psi\right)\right\|_{L^p}^p
\le C_3\text{M}_{\min}
+\frac12\left\|(M^*_a)_{\mathbb A}\left(\vec f,\psi\right)\right\|_{L^p}^p,
\end{align*}
which, together with the just proved important property
[namely $(M^*_a)_{\mathbb A}(\vec f,\psi)\in L^p$],
further implies that
$\|(M^*_a)_{\mathbb A}(\vec f,\psi)\|_{L^p}
\le 2^{\frac1p}C_3^{\frac1p}M_{\min}.$
This finishes the proofs of (iv) and (v) and
hence Theorem \ref{equivalent A}.
\end{proof}

Now, we are able to provide the equivalences of various
$\mathbb A$-matrix-weighted maximal functions.

\begin{theorem}\label{if and only if A}
Let $p\in(0,\infty)$ and the family of matrices $\mathbb A:=\{A_Q\}_{Q\in\mathscr{Q}}$
be strongly doubling of order $(\beta_1,\beta_2,\omega)$
for some $\beta_1,\beta_2,\omega\in[0,\infty)$.
Assume that $\psi\in\mathcal{S}$ satisfies $\int_{\mathbb R^n}\psi(x)\,dx\neq0$.
Let $l\in(\frac{n}p+\omega,\infty)$
and $N\in\mathbb N$ satisfy $N\geq l+\beta_2$.
Suppose that $a\in(0,\infty)$ and $b\in(0,\delta]$, where $\delta$ is the same as in
Theorem \ref{equivalent A}(v).
Then, for any $\vec f\in(\mathcal{S}')^m$,
\begin{align*}
\left\|\vec f\right\|_{H^p_{\mathbb A,N}}
&\sim\left\|M_{\mathbb A}\left(\vec f,\psi\right)\right\|_{L^p}
\sim\left\|\left(M^*_a\right)_{\mathbb A}\left(\vec f,\psi\right)\right\|_{L^p}
\sim\left\|\left(\widetilde M^*_{a,b}\right)_{\mathbb A}\left(\vec f,\psi\right)\right\|_{L^p}
\sim\left\|\left(M_N\right)_{\mathbb{A}}\left(\vec f\right)\right\|_{L^p}\\
&\sim\left\|\left(M^*_{a,N}\right)_{\mathbb A}\left(\vec f\right)\right\|_{L^p}
\sim\left\|\left(M^{**}_l\right)_{\mathbb A}\left(\vec f,\psi\right)\right\|_{L^p}
\sim\left\|\left(M^{**}_{l,N}\right)_{\mathbb A}\left(\vec f\right)\right\|_{L^p},
\end{align*}
where the positive equivalence constants are independent of $\vec f$.
\end{theorem}

\begin{proof}
Using (i) and (v) of Theorem \ref{equivalent A}, we conclude that
$(\widetilde M^*_{a,b})_{\mathbb A}(\vec f,\psi)$ and $(M^{**}_{l,N})_{\mathbb A}(\vec f)$
are, respectively, the minimum and the maximum of these maximal functions.
On the other hand, from (ii) through (iv) of Theorem \ref{equivalent A}, we infer that
\begin{align*}
\left\|\left(M^{**}_{l,N}\right)_{\mathbb A}\left(\vec f\right)\right\|_{L^p}
\lesssim\left\|\left(M^{**}_l\right)_{\mathbb A}\left(\vec f,\psi\right)\right\|_{L^p}
\lesssim\left\|\left(M^{*}_a\right)_{\mathbb A}\left(\vec f,\psi\right)\right\|_{L^p}\lesssim\left\|M_{\mathbb A}\left(\vec f,\psi\right)\right\|_{L^p}
\lesssim\left\|\left(\widetilde M^*_{a,b}\right)_{\mathbb A}\left(\vec f,\psi\right)\right\|_{L^p},
\end{align*}
which further implies the desired conclusions and hence completes the proof of
Theorem \ref{if and only if A}.
\end{proof}

Based on Theorem \ref{if and only if A}, in what follows we denote
$H_{\mathbb A,N(\mathbb A)}^p$ simply by
$H_{\mathbb{A}}^p$, where the family of matrices $\mathbb A:=\{A_Q\}_{Q\in\mathscr{Q}}$
is strongly doubling of order $(\beta_1,\beta_2,\omega)$
for some $\beta_1,\beta_2,\omega\in[0,\infty)$ and
$$
N(\mathbb A):=\left\lfloor\frac np+\omega+\beta_2\right\rfloor+1.
$$

\subsection{Relations Between Matrix-Weighted Hardy Spaces
and $\mathbb A$-Matrix-Weighted Hardy Spaces} \label{weighted}

The following theorem is the main result of this subsection,
which establishes the relations between matrix-weighted Hardy spaces
and $\mathbb A$-matrix-weighted Hardy spaces.

\begin{theorem}\label{weight and reducing}
Let $p\in(0,\infty)$, $W\in A_{p,\infty}$, and
$\mathbb A:=\{A_Q\}_{Q\in\mathscr{Q}}$
be a family of reducing operators of order $p$ for $W$.
Assume that $\psi\in\mathcal{S}$ satisfies
$\int_{\mathbb R^n}\psi(x)\,dx\neq 0$.
Let $a\in(0,\infty)$,
$l\in(n/p+[d_{p,\infty}^{\mathrm{lower}}(W)+d_{p,\infty}^{\mathrm{upper}}(W)]/p,\infty)$,
and $N\in\mathbb N$ satisfy $N> l+ d_{p,\infty}^{\mathrm{upper}}(W)/p$.
Then, for any $\vec f\in(\mathcal S')^m$,
\begin{align*}
\left\|\vec f\right\|_{H^p_{\mathbb A}}
&\sim\left\|\vec f\right\|_{H^p_W}\sim\left\|M^p_W\left(\vec f,\psi\right)\right\|_{L^p}
\sim\left\|\left(M^*_a\right)^p_W\left(\vec f,\psi\right)\right\|_{L^p}
\sim\left\|\left(M_N\right)^p_W\right\|_{L^p}\\
&\sim\left\|\left(M^*_{a,N}\right)^p_W\left(\vec f\right)\right\|_{L^p}
\sim\left\|\left(M^{**}_l\right)^p_W\left(\vec f,\psi\right)\right\|_{L^p}
\sim\left\|\left(M^{**}_{l,N}\right)^p_W\left(\vec f\right)\right\|_{L^p},
\end{align*}
where the positive equivalence constants are independent of $\vec f$.
\end{theorem}

\begin{remark}\label{2.29}
It is worth mentioning that Theorems \ref{if and only if A} and
\ref{weight and reducing} provide various $\mathbb A$-matrix-weighted
maximal function characterizations of matrix-weighted Hardy spaces,
which are new even in the scalar case.
Let $m=1$, $p\in(0,\infty)$, $W\in A_\infty(\mathbb{R}^n)$, $\{A_Q\}_{\mathrm{cube}\,Q}$
be a family of
reducing operators of order $p$ for $W$, and $\psi\in\mathcal{S}$ satisfy
$\int_{\mathbb R^n}\psi(x)\,dx\neq 0$.
Here, and thereafter, for any cube $Q$, $W(Q):=\int_QW(x)\,dx$.
Observe that, in this case, by Definition
\ref{reduce}, one obviously has, for any cube
$Q$, $|A_Q|\sim [\frac {W(Q)}{|Q|}]^{\frac 1p}$, where
the positive equivalence constants
depend only on $p$, and hence, with $a,b$ as in Theorem \ref{if and only if A},
for any $f\in \mathcal{S}'$,
\begin{align}\label{e11.29}
\|f\|_{H^p_W}
\sim\left\|\sup_{t\in(0,\infty)}\left\{\sum_{Q\in\mathscr{Q}_t}
\left[\frac{ W(Q)}{|Q|}\right]^{\frac{1}{p}}
\mathbf{1}_Q(\cdot)\right\}\max_{\{Q\in\mathscr{Q}_{bt}:\,Q\cap B(\cdot,at)\neq\emptyset\}}
\inf_{y\in Q}\left|\psi_t*f(y)\right|
\right\|_{L^p},
\end{align}
where the positive equivalence constants are independent of $f$.
Obviously, the right-hand side of \eqref{e11.29} is a discretization
of the weighted $L^p$ norm of the non-tangential infimum
maximal function, and, by Theorem \ref{weight and reducing},
it can be replaced by any of all the other
quasi-norms related to reducing operators appearing
in Theorem \ref{if and only if A} with $m=1$.
Thus, in this sense, the $\mathbb A$-matrix-weighted maximal functions
in Definition \ref{reducing hardy} can be regarded as
the discrete variant of the corresponding
matrix-weighted maximal functions in Definition \ref{HW}.
\end{remark}

To show Theorem \ref{weight and reducing}, we need the following two lemmas.
The following lemma is exactly \cite[Corollary 3.9]{bhyyNew}.

\begin{lemma}\label{le-Apinfty}
Let $p\in(0,\infty)$, $W\in A_{p,\infty}$,
and $\{A_Q\}_{\mathrm{cube}\,Q}$ be a family of
reducing operators of order $p$ for $W$.
Then there exists a positive constant $\widetilde C$,
depending only on $m$ and $p$, such that,
for any cube $Q\subset\mathbb R^n$ and any $M\in(0,\infty)$,
\begin{align}\label{reduceback}
\left|\left\{y\in Q:\ \left\|A_Q W^{-\frac1p}(y)\right\|^p\geq e^M\right\}\right|
\le \frac{\log(\widetilde C[W]_{A_{p,\infty}})}{M} |Q|.
\end{align}
\end{lemma}

\begin{remark}\label{re-AQW}
If $m=1$, $W\in A_\infty(\mathbb{R}^n)$, and
$\{A_Q\}_{\mathrm{cube}\,Q}$ is a family of
reducing operators of order $p$ for $W$, then
\eqref{reduceback} becomes
$|\{y\in Q:\ W(y)\le e^{-M}\frac{W(Q)}{|Q|}\}|
\lesssim \frac{|Q|}{M},$
where the implicit positive constant is independent of $Q$ and $M$.
This coincides with the formula in \cite[Theorem 7.3.3(a)]{g14c}.
In the matrix case, Volberg \cite[p.\,454,\ Remark]{v97} pointed out that
the set where the $A_{p,\infty}$ matrix weight
is much smaller than its average is small,
which is exactly what is expressed in Lemma \ref{le-Apinfty}.
\end{remark}

For any $t\in(0,\infty)$ and any nonnegative measurable function $f$
on $\mathbb R^n$ (or any $f\in L^1_{\mathrm{loc}}$), let
\begin{equation*}
E_t (f):=\sum_{Q\in\mathscr{Q}_t}\fint_Q f(x)\,dx\mathbf{1}_Q,
\end{equation*}
where $\mathscr{Q}_t$ is the same as in \eqref{deQt}.
The following lemma is exactly \cite[Corollary 5.8]{bhyyNew},
which is a suitable replacement of the Fefferman--Stein vector-valued maximal
inequality in the case of matrix weights and plays a key role
in the proof of Theorem \ref{weight and reducing}.

\begin{lemma}\label{Nazarov 2}
Let $p\in(0,\infty)$, $q\in(0,\infty]$, $W\in A_{p,\infty}$,
and $\{A_Q\}_{Q\in\mathscr{Q}}$ be a family of
reducing operators of order $p$ for $W$.
For any $t\in(0,\infty)$, define
\begin{align}\label{2.28x}
\gamma_t:=\left\|W^{\frac{1}{p}}A_t^{-1}\right\|,
\end{align}
where $A_t$ is the same as in \eqref{Atref}.
Then there exists a positive constant $C$ such that,
for any sequence $\{f_j\}_{j\in\mathbb{Z}}$ of
nonnegative measurable functions on $\mathbb R^n$
or $\{f_j\}_{j\in\mathbb{Z}}\subset L^1_{\mathrm{loc}}$,
$$
\left\|\left\{\gamma_{2^{-j}}E_{2^{-j}}\left(f_j\right)\right\}_{j\in\mathbb Z}\right\|_{L^pl^q}
\le C\left\|\left\{E_{2^{-j}}\left(f_j\right)\right\}_{j\in\mathbb Z}\right\|_{L^pl^q},
$$
where, for any $\{f_j\}_{j\in{\mathbb Z}}$,
$\|\{f_j\}_{j\in{\mathbb Z}}\|_{L^pl^q}:=\|(\sum_{j\in{\mathbb Z}}|f_j|^q)^{\frac1q}\|_{L^p}$.
\end{lemma}

Next, we prove Theorem \ref{weight and reducing}.

\begin{proof}[Proof of Theorem \ref{weight and reducing}]
Let $d_1\in[\![d_{p,\infty}^{\mathrm{lower}}(W),n)$ and
$d_2\in[\![d_{p,\infty}^{\mathrm{upper}}(W),\infty)$.
Then, applying Remark \ref{r2.18}, we conclude that
$\mathbb A$ is strongly doubling of order
$(\frac{d_1}{p},\frac{d_2}{p},\frac{d_1+d_2}{p})$.
By Definition \ref{HW}, we find that
$M^p_W(\vec f,\psi)$ and $(M^{**}_{l,N})^p_W(\vec f)$
are, respectively, the minimum and the maximum of these maximal functions.
Therefore, to show the present theorem, it is enough to prove
\begin{align}\label{W<A}
\left\|\left(M^{**}_{l,N}\right)^p_W\left(\vec f\right)\right\|_{L^p}
\lesssim\left\|\vec f\right\|_{H^p_{\mathbb A}}
\end{align}
and
\begin{align}\label{A<W}
\left\|\vec f\right\|_{H^p_{\mathbb A}}
\lesssim \left\|M^p_W\left(\vec f,\psi\right)\right\|_{L^p}.
\end{align}

We first show \eqref{W<A}.
From the definition of $\gamma_t$, it follows that,
for any $j\in\mathbb Z$, $t\in(2^{-j-1},2^{-j}]$, and $x\in\mathbb R^n$,
\begin{align}\label{1223}
&\sup_{\phi\in\mathcal S_N}\sup_{y\in{\mathbb{R}^n}}
\left|W^{\frac{1}{p}}(x)\phi_t*\vec f(x-y)\right|
\left(1+\frac{|y|}{t}\right)^{-l}\\
&\quad\le\gamma_{2^{-j}}(x)
\sup_{\phi\in\mathcal S_N}\sup_{y\in{\mathbb{R}^n}}
\left|A_{2^{-j}}(x)\phi_t*\vec f(x-y)\right|
\left(1+\frac{|y|}{t}\right)^{-l}
\le\gamma_{2^{-j}}(x)J_j(x),\nonumber
\end{align}
where $A_{2^{-j}}$ is the same as in
\eqref{Atref} with $t:=2^{-j}$,
$\gamma_{2^{-j}}$ is the same as \eqref{2.28x}, and
\begin{align*}
J_j(x):=\sum_{Q\in\mathscr{Q}_{2^{-j}}}\sup_{\phi\in\mathcal S_N}
\sup_{t\in(2^{-j-1},2^{-j}]}\sup_{z\in Q}\sup_{y\in\mathbb R^n}
\left|A_Q\phi_t*\vec f(z-y)\right|
\left(1+\frac{|y|}{t}\right)^{-l}\mathbf{1}_Q(x).
\end{align*}
Using \eqref{1223} and the definition of $(M^{**}_{l,N})^p_{W}(\vec f)$,
we conclude that
$
\|(M^{**}_{l,N})^p_{W}(\vec f)\|_{L^p}
\le\|\{\gamma_{2^{-j}}J_j\}_{j\in\mathbb Z}\|_{L^pl^\infty},
$
which, together with Lemma \ref{Nazarov 2} and the observation that $J_j$
is constant on any $Q\in\mathscr{Q}_{2^{-j}}$, further implies that
\begin{align}\label{sim4 equ}
\left\|\left(M^{**}_{l,N}\right)^p_{W}\left(\vec f\right)\right\|_{L^p}
\le\left\|\left\{\gamma_{2^{-j}}E_{2^{-j}}\left(J_j\right)\right\}_{j\in\mathbb Z}\right\|_{L^pl^\infty}
\lesssim\left\|\left\{E_{2^{-j}}\left(J_j\right)\right\}_{j\in\mathbb Z}\right\|_{L^pl^\infty}
=\left\|\left\{J_j\right\}_{j\in\mathbb Z}\right\|_{L^pl^\infty}.
\end{align}
By Definition \ref{doubling}(i) and the definition of $(M^{**}_{l,N})_{\mathbb A}(\vec f)$,
we find that, for any $j\in\mathbb Z$,
$Q\in\mathscr{Q}_{2^{-j}}$, and $x\in Q$,
$\|A_{2^{-j}}(x)A_t^{-1}(x)\|\lesssim 1$ with $t\in(2^{-j-1},2^{-j}]$
and hence
\begin{align*}
J_j(x)
&\le\sup_{\phi\in\mathcal S_N}\sup_{t\in(2^{-j-1},2^{-j}]}
\sup_{y\in\mathbb R^n}\left\|A_{2^{-j}}(x)A_t^{-1}(x)\right\|
\left|A_t(x)\phi_t*\vec f(x-y)\right|
\sup_{z\in Q}\left(1+\frac{|x-z|}{t}+\frac{|y|}{t}\right)^{-l}\\
&\lesssim\sup_{\phi\in\mathcal S_N}\sup_{t\in(2^{-j-1},2^{-j}]}
\sup_{y\in\mathbb R^n}
\left|A_t(x)\phi_t*\vec f(x-y)\right|
\left(1+\frac{|y|}{t}\right)^{-l}
\lesssim\left(M^{**}_{l,N}\right)_{\mathbb A}\left(\vec f\right)(x).
\end{align*}
From this, \eqref{sim4 equ}, and Theorem \ref{if and only if A}, we deduce that
\begin{align*}
\left\|\left(M^{**}_{l,N}\right)^p_{W}\left(\vec f\right)\right\|_{L^p}
\lesssim\left\|\left(M^{**}_{l,N}\right)_{\mathbb A}\left(\vec f\right)
\right\|_{L^p}
\sim \left\|\vec f\right\|_{H^p_{\mathbb A}},
\end{align*}
which completes the proof of \eqref{W<A}.

Now, we prove \eqref{A<W}.
Let $\nu\in(0,p)$ and $b$ be the same
as in Theorem \ref{if and only if A}.
Let $L:=2\log(\widetilde C[W]_{A_{p,\infty}})$,
where $\widetilde C$
is the same as in Lemma \ref{le-Apinfty},
and, for any $Q\in\mathscr{Q}$,
$E_{Q}:=\{y\in Q:\ \|A_QW^{-\frac 1p}(y)\|^p<e^L\}$, where,
by Lemma \ref{le-Apinfty},
we conclude that, for any $Q\in\mathscr{Q}$,
$|E_Q|\geq\frac12|Q|$.
From this and Definition \ref{doubling}(i), we infer that,
for any $t\in(0,\infty)$, $Q\in\mathscr{Q}_t$, $R\in\mathscr{Q}_{bt}$, $x\in Q$,
and $R\cap B(x,at)\neq\emptyset$,
$E_R\subset R\subset (1+2a+2b)Q$ and $|E_R|\sim|(1+2a+2b)Q|$.
By these and the definition of $M^p_{W}(\vec f,\psi)$,
we find that, for any $t\in(0,\infty)$, $Q\in\mathscr{Q}_t$,
$R\in\mathscr{Q}_{bt}$, $x\in Q$, and $R\cap B(x,at)\neq\emptyset$,
\begin{align*}
\inf_{y\in R}\left|A_t(x)\psi_t*\vec f(y)\right|^\nu&\le
\inf_{y\in R}\left\|A_t(x)A_R^{-1}\right\|^\nu\left|A_R\psi_t*\vec f(y)\right|^\nu
\lesssim
\inf_{y\in R}\left|A_R\psi_t*\vec f(y)\right|^\nu\\
&\leq\inf_{y\in E_R}\left\|A_RW^{-\frac1p}(y)\right\|^\nu
\left|W^{\frac1p}(y)\psi_t*\vec f(y)\right|^\nu\\
&\lesssim
\inf_{y\in E_R}\left|W^{\frac1p}(y)\psi_t*\vec f(y)\right|^\nu
\le\fint_{E_R}\left|W^{\frac1p}(y)\psi_t*\vec f(y)\right|^\nu\,dy\\
&\lesssim\fint_{(1+2a+2b)Q}\left[M^p_{W}\left(\vec f,\psi\right)(y)\right]^\nu\,dy
\lesssim\mathcal{M}\left(\left[M^p_{W}\left(\vec f,\psi\right)\right]^\nu\right)(x),
\end{align*}
which, together with the definition of
$(\widetilde M^{*}_{a,b})_{\mathbb A}(\vec f,\psi)$,
further implies that, for any $x\in\mathbb R^n$,
$$
\left[\left(\widetilde M^{*}_{a,b}\right)_{\mathbb A}\left(\vec f,\psi\right)(x)\right]^\nu
\lesssim \mathcal{M}\left(\left[M^p_{W}\left(\vec f,\psi\right)\right]^\nu\right)(x).
$$
Combining this, Theorem \ref{if and only if A},
and the boundedness of $\mathcal{M}$ on $L^{\frac{p}{\nu}}$,
we conclude that
\begin{align*}
\left\|\vec f\right\|_{H^p_{\mathbb A}}
&\sim\left\|\left(\widetilde M^{*}_{a,b}\right)_{\mathbb A}
\left(\vec f,\psi\right)\right\|_{L^p}
=\left\|\left[\left(\widetilde M^{*}_{a,b}\right)_{\mathbb A}
\left(\vec f,\psi\right)\right]^{\nu}\right\|_{L^{\frac p\nu}}^\frac1\nu\\
&\lesssim\left\|\mathcal{M}\left(\left[M^p_{W}\left(\vec f,\psi
\right)\right]^\nu\right)\right\|_{L^{\frac p\nu}}^\frac1\nu
\lesssim\left\|\left[M^p_{W}\left(\vec f,\psi\right)\right]^\nu
\right\|_{L^{\frac p\nu}}^\frac1\nu
=\left\|M^p_{W}\left(\vec f,\psi\right)\right\|_{L^p}.
\end{align*}
This finishes the proof of \eqref{A<W} and hence Theorem \ref{weight and reducing}.
\end{proof}

Obviously, Theorem \ref{if and only if W}
follows directly from Theorem \ref{weight and reducing};
we omit the details. Observe that the $\mathbb A$-matrix-weighted
non-tangential infimum maximal function
plays a key role in the above proof of Theorem \ref{weight and reducing}.

Finally, we present an embedding proposition for matrix-weighted Hardy spaces.

\begin{proposition}\label{embedding}
Let $p\in(0,\infty)$ and $W\in A_{p,\infty}$.
Then $H_{W}^p\subset(\mathcal S')^m$.
Moreover, there exists a positive constant $C$,
depending only on $n$, $p$, and $W$,
such that, for any $\vec f\in(\mathcal S')^m$
and $\phi\in\mathcal{S}$,
\begin{align*}
\left|\left\langle\vec f,\phi\right\rangle\right|
\le C\|\phi\|_{S_{N(W)}}\left\|\vec f\right\|_{H_{W}^p},
\end{align*}
where $N(W)$ is the same as in \eqref{N(W)}.
\end{proposition}

\begin{proof}
By the definition of $(M^*_{1,N(W)})^p_{W}$,
we find that, for almost every $x\in B(\mathbf{0},1)$,
\begin{align*}
\left|\left\langle\vec f,\phi\right\rangle\right|
&=\left|\widetilde{\phi}*\vec f(\mathbf{0})\right|
\le\left\|W^{-\frac1p}(x)\right\|
\left|W^{\frac1p}(x)\widetilde{\phi}*\vec f(\mathbf{0})\right|\\
&\le\left\|W^{-\frac1p}(x)\right\|\|\phi\|_{S_{N(W)}}
\left(M^*_{1,N(W)}\right)^p_{W}\left(\vec f\right)(x),
\end{align*}
where $\widetilde{\phi}:=\phi(-\cdot)$, and hence
\begin{align*}
\log\left(\left|\left\langle\vec f,\phi\right\rangle\right|^p\right)
\le \log\left(\left\|W^{-\frac1p}(x)\right\|^p\right)
+\log\left(\|\phi\|_{S_{N(W)}}^p\right)
+p\log\left(\left(M^*_{1,N(W)}\right)^p_{W}\left(\vec f\right)(x)\right).
\end{align*}
Taking the average with respect to $x$ over
$B(\mathbf0,1)$ on both sides of the above inequality
and using \cite[(3.1)]{bhyyNew} and the Jensen inequality,
we conclude that
\begin{align*}
\log\left(\left|\left\langle\vec f,\phi\right\rangle\right|^p\right)
&\le\fint_{B(\mathbf{0},1)}\log\left(\left\|W^{-\frac1p}(x)\right\|^p\right)\,dx
+\log\left(\|\phi\|_{S_{N(W)}}^p\right)\\
&\quad+\fint_{B(\mathbf{0},1)}\log\left(\left[\left(M^*_{1,N(W)}
\right)^p_{W}\left(\vec f\right)(x)\right]^p\right)\,dx\\
&\le C_0+\log\left(\|\phi\|_{S_{N(W)}}^p\right)+\log\left(\fint_{B(\mathbf{0},1)}\left[\left(M^*_{1,N(W)}
\right)^p_{W}\left(\vec f\right)(x)\right]^p\,dx\right)\\
&=\log\left(e^{C_0}\|\phi\|_{S_{N(W)}}^p\fint_{B(\mathbf{0},1)}\left[\left(M^*_{1,N(W)}
\right)^p_{W}\left(\vec f\right)(x)\right]^p\,dx\right).
\end{align*}
From this and the definition of $H^p_W$, we deduce that
\begin{align*}
\left|\left\langle\vec f,\phi\right\rangle\right|^p
\lesssim\|\phi\|_{S_{N(W)}}^p
\fint_{B(\mathbf{0},1)}\left[\left(M^*_{1,N(W)}\right)^p_{W}\left(\vec f\right)(x)\right]^p\,dx
\lesssim\|\phi\|_{S_{N(W)}}^p
\left\|\vec f\right\|_{H_{W}^p}^p.
\end{align*}
This finishes the proof of Proposition \ref{embedding}.
\end{proof}

\begin{proposition}\label{jwdj}
Let $p\in(0,\infty)$ and $W\in A_{p,\infty}$.
Then $H^p_W$ is complete.
\end{proposition}

\begin{proof}
Suppose that $\{\vec f_k\}_{k\in\mathbb{N}}$ is a Cauchy sequence in $H^p_W$.
Using this and
Proposition \ref{embedding},
we find that $\{\vec f_k\}_{k\in\mathbb{N}}$
is also a Cauchy sequence in $(\mathcal{S}')^m$, which, together with
the completeness of $(\mathcal{S}')^m$, further implies that there exists
$\vec f\in(\mathcal{S}')^m$ such that
$\vec f=\lim_{k\to\infty}\vec f_k$ in $(\mathcal{S}')^m$.
Let $\psi\in\mathcal S$ satisfy
$\int_{\mathbb R^n}\psi(x)\,dx\neq 0$.
From the definition of $M^p_W(\vec f,\psi)$, we infer that,
for any $k,l\in\mathbb{N}$, $t\in(0,\infty)$, and $x\in\mathbb{R}^n$,
$$\left|W^{\frac1p}(x)\left\langle\vec f_{k+l}-\vec f_k,\psi_t(\cdot-x)
\right\rangle\right|\le M^p_W\left(\vec f_{k+l}-\vec f_k,\psi\right)(x).$$
Letting $l\to\infty$, we obtain,
for any $k\in\mathbb{N}$, $t\in(0,\infty)$, and $x\in\mathbb{R}^n$,
$$\left|W^{\frac1p}(x)\left\langle\vec f-\vec f_k,\psi_t(\cdot-x)
\right\rangle\right|\le\liminf_{l\to\infty}
M^p_W\left(\vec f_{k+l}-\vec f_k,
\psi\right)(x),
$$
which, together with the definition of $M^p_W(\vec f,\psi)$, further implies that
$$M^p_W\left(\vec f-\vec f_k,
\psi\right)(x)\le\liminf_{l\to\infty}
M^p_W\left(\vec f_{k+l}-\vec f_k,
\psi\right)(x).
$$
Taking the $L^p$ norm on both sides of the above inequality
and using the Fatou lemma, we conclude that
$$\left\|M^p_W\left(\vec f-\vec f_k,
\psi\right)\right\|_{L^p}\le\liminf_{l\to\infty}
\left\|M^p_W\left(\vec f_{k+l}-\vec f_k,
\psi\right)\right\|_{L^p}\to0
$$
as $k\to\infty$,
which, together with Theorem \ref{weight and reducing} and the
sublinearity of the quasi-norm of $H^p_W$, further implies that
$\vec f\in H^p_W$ and
$\lim_{k\to\infty}\|\vec f-\vec f_k\|_{H^p_W}=0$.
This finishes the proof of Proposition \ref{jwdj}.
\end{proof}

\section{Atomic Characterization\label{ad}}

In this section, we establish an atomic characterization of $H^p_W$
via beginning with the following concept of atoms.
\begin{definition}\label{F-atom}
Let $p\in(0,\infty)$, $q\in[1,\infty]$, and $s\in{\mathbb Z}_+$.
\begin{itemize}
\item[\rm (I)] Let $W$ be a matrix weight.
A function $\vec a$ is called a \emph{$(p,q,s)_{W}$-atom}
supported in a cube $Q$ if
\begin{enumerate}
\item[\rm (i)] $\operatorname{supp}\vec a\subset Q$,
\item[\rm (ii)]
$\{\int_Q[\int_{Q}|W^{\frac1p}(y)\vec a(x)
|^p\,dy]^{\frac qp}\,dx\}^{\frac1q}\leq|Q|^{\frac1q},$
\item[\rm (iii)]
for any $\gamma:=(\gamma_1,\ldots,\gamma_n)\in\mathbb{Z}_+^n$
with $|\gamma|:=\gamma_1+\cdots+\gamma_n\le s$,
$\int_{\mathbb{R}^n}x^\gamma \vec a(x)\,dx=\vec{0}$,
where $x^\gamma:=x_1^{\gamma_1}\cdots x_n^{\gamma_n}$
for any $x:=(x_1,\ldots,x_n)\in\mathbb{R}^n$.
\end{enumerate}
\item[\rm (II)] Let $\mathbb A:=\{A_Q\}_{\mathrm{cube}\,Q}$
be a family of positive definite matrices.
A function $\vec a$ is called a \emph{$(p,q,s)_{\mathbb A}$-atom}
supported in a cube $Q$ if $\vec a$ satisfies (i) and (iii) of (I) and
\begin{enumerate}
\item[\rm (iv)]
$[\int_Q|A_Q\vec a(x)|^q\,dx]^\frac1q\le|Q|^{\frac1q-\frac1p}.$
\end{enumerate}
\end{itemize}
\end{definition}
\begin{remark}\label{kfc} Let all the symbols be
the same as in Definition \ref{F-atom}.
\begin{enumerate}[{\rm(i)}]
\item From Definition \ref{reduce}, we easily deduce that,
if $W$ is a matrix weight and
$\mathbb A:=\{A_Q\}_{\mathrm{cube}\,Q}$ is a family of
reducing operators of order $p$ for $W$, then a $(p,q,s)_{\mathbb A}$-atom
is a harmless positive constant multiple of a $(p,q,s)_W$-atom.

\item If $m=1$ and $W \equiv 1$, then Definition \ref{F-atom}(ii) reduces back to
$\|\vec{a}\|_{L^q}\le |Q|^{\frac1q-\frac1p}$.
In this case, the $(p,q,s)_{W}$-atom coincides with
the classical atom of Hardy spaces originally introduced by
Coifman \cite[p.\,269]{c74} and Latter \cite[p.\,95]{l78}.

\item
If $m=1$ and $W$ is a scalar weight, then Definition \ref{F-atom}(ii)
becomes $\|\vec a\|_{L^q}\le \frac{|Q|^{\frac{1}{q}}}{[W(Q)]^\frac1p}$.
In this case, the $(p,q,s)_{W}$-atom coincides with
the atom of weighted Hardy spaces;
see, for instance, \cite[Definition 3.5]{shyy}
with $X:=L^p_W$ and see also \cite[p.\,230]{r20}.
\end{enumerate}
\end{remark}
To establish the atomic characterization of $H^p_W$,
we also need the reverse H\"older inequality associated with
matrix weights, which is
exactly \cite[Corollary 5.7(i)]{bhyyNew}.

\begin{lemma}\label{86}
Let $p\in(0,\infty)$, $W\in A_{p,\infty}$,
and $\mathbb A:=\{A_Q\}_{Q\in\mathscr{Q}}$ be a family of
reducing operators of order $p$ for $W$.
Then there exist a positive constant $r\in(1,\infty)$, depending only on $n$, $m$, $p$,
and $[W]_{A_{p,\infty}}$, and a positive constant $C$,
depending only on $m$ and $p$, such that
\begin{equation}\label{kappa}
\sup_{\mathrm{cube}\,Q}
\left[\fint_Q\left\|W^{\frac{1}{p}}(x)A_Q^{-1}\right\|^{rp}\,dx\right]^{\frac{1}{rp}}
\le C.
\end{equation}
\end{lemma}

Let
\begin{align}\label{r_w}
r_W:=\sup\left\{r:\ r\text{ satisfies \eqref{kappa}}\right\}.
\end{align}

\begin{remark}
Let $m=1$ and all the other symbols be the same as in Lemma \ref{86}.
Then \eqref{kappa} reduces back to that, for any cube $Q$,
$\{\fint_Q[W(x)]^{r}\,dx\}^{\frac{1}{r}}
\le C\fint_QW(x)\,dx$
and $r_W$ coincides with a special case of the critical index
for the reverse H\"older condition in \cite[(1.14)]{ylk}.
\end{remark}

In what follows, $l^p$ is defined to be the space of all sequences
$\{\lambda_k\}_{k\in\mathbb{Z}}$ of $\mathbb{C}$ such that
$$\|\{\lambda_k\}_{k\in\mathbb{Z}}\|_{l^p}:=
\left(\sum_{k\in\mathbb Z}|\lambda_k|^p\right)^{\frac 1p}<\infty.$$

\begin{theorem}\label{W-F-atom}
Let $p\in(0,1]$ and $W$ be a matrix weight.
Then the following two statements hold.
\begin{enumerate}[{\rm(i)}]
\item Let $W\in A_{p,\infty}$,
$q\in(\max\{1,\frac{r_Wp}{r_W-1}\},\infty]$
with $r_W$ the same as in \eqref{r_w},
and $s\in(\lfloor n(\frac{1}{p}-1)+\frac{d^{\mathrm
{upper}}_{p,\infty}(W)}{p}\rfloor,\infty)\cap\mathbb Z_+$.
For any $\{\lambda_k\}_{k\in{\mathbb Z}}\in l^p$ and
any $(p,q,s)_W$-atoms $\{\vec a_k\}_{k\in\mathbb Z}$,
there exists $\vec f\in H^p_W$ such that
$\vec f=\sum_{k\in\mathbb Z}\lambda_k\vec a_k$
in both $H^p_W$ and $(\mathcal{S}')^m$.
Moreover, there exists a positive constant $C$, independent of both
$\{\lambda_k\}_{k\in\mathbb Z}$ and $\{\vec a_k\}_{k\in\mathbb Z}$, such that
$\|\vec f\|_{H^p_W}\le
C\|\{\lambda_k\}_{k\in\mathbb Z}\|_{l^p}.$

\item Let $W\in A_p$ and
$s\in(\lfloor n(\frac{1}{p}-1)\rfloor,\infty)\cap\mathbb Z_+$.
For any $\vec f\in H^p_W$, there exist
a sequence $\{\lambda_k\}_{k\in\mathbb Z}\in l^p$
and $(p,\infty,s)_W$-atoms $\{\vec a_k\}_{k\in\mathbb Z}$ such that
$\vec f=\sum_{k\in\mathbb Z}\lambda_k\vec a_k$ in both $H^p_W$ and $(\mathcal{S}')^m$.
Moreover, there exists a positive constant $C$, independent of $\vec f$,
such that
$\|\{\lambda_k\}_{k\in\mathbb Z}\|_{l^p}
\le C\|\vec f\|_{H^p_W}.$
\end{enumerate}
\end{theorem}

\begin{remark}\label{wsw} Let all the symbols be the same as
in Theorem \ref{W-F-atom} and, in addition, let $\mathbb A:=\{A_Q\}_{\mathrm{cube}\,Q}$ be a family of
reducing operators of order $p$ for $W$.
\begin{enumerate}[{\rm(i)}]
\item Using Remark \ref{kfc}(i), we find that Theorem \ref{W-F-atom}
still holds with the $(p,q,s)_W$-atoms therein replaced by
the $(p,q,s)_{\mathbb A}$-atoms.

\item Let $1\le q_1<q_2\le\infty$
and $s\in{\mathbb Z}_+$.
If $\vec a$ is a $(p,q_2,s)_{\mathbb A}$-atom
supported in a cube $Q$, then, by Definition \ref{F-atom}(iv), we find that
$$\left[\int_Q\left|A_Q\vec a(x)\right|^{q_1}\,dx\right]^\frac1{q_1}
\le|Q|^{\frac1{q_1}-\frac1{q_2}}
\left[\int_Q\left|A_Q\vec a(x)\right|^{q_2}\,dx\right]^\frac1{q_2}
\le|Q|^{\frac1{q_1}-\frac1p},$$
which, together with both (i) and (iii) of Definition \ref{F-atom},
further implies that $\vec a$ is also a $(p,q_1,s)_{\mathbb A}$-atom.
This and Remark \ref{kfc}(i) further show that, for any $q\in[1,\infty)$,
Theorem \ref{W-F-atom}(ii) still holds if $(p,\infty,s)_W$-atoms therein
are replaced, respectively, by $(p,q,s)_W$-atoms or $(p,q,s)_{\mathbb A}$-atoms.

\item Let $m=1$, $p=1$, $q=\infty$, $s=0$, and $W\in A_1({\mathbb{R}^n})$.
Then in this case
Theorem \ref{W-F-atom} reduces back to \cite[Theorem 5.1]{B81}.
\end{enumerate}
\end{remark}

In what follows, for any $\vec f\in(\mathscr{M})^m$ and $g\in\mathscr{M}$, let
$$
\left\langle\vec f, g\right\rangle
:=\int_{\mathbb R^n}\vec f(x)g(x)\,dx
$$
whenever the integral in the right-hand side of the above equality exists.
The proof of Theorem \ref{W-F-atom} strongly depends on
the following variant of the Calder\'on--Zygmund
decomposition associated with matrix weights,
which is of independent interest.

\begin{theorem}\label{l-CZ}
Let $p\in(0,1]$, $W\in A_{p,\infty}$, and
$\mathbb{A}:=\{A_Q\}_{Q\in\mathscr{Q}}$
be a family of reducing operators of order $p$ for $W$.
Assume that $N\in[N(W),\infty)\cap\mathbb N$ with $N(W)$ the same as in
\eqref{N(W)}, $s\in(\lfloor n(\frac{1}{p}-1)+
d_{p,\infty}^{\mathrm{upper}}(W)\rfloor,\infty)\cap{\mathbb Z}_+$,
and $\psi\in\mathcal{S}$ satisfies
$\operatorname{supp}\psi\subset B(\mathbf{0},1)$
and $\int_{\mathbb R^n}\psi(x)\,dx\neq0$.
Then there exists a positive constant $C$ such that, for any
$\vec f\in H_W^p$ and $\alpha\in(0,\infty)$, the following statements hold.
\begin{enumerate}[{\rm(i)}]
\item There exists an open set $O\subsetneqq\mathbb R^n$ such that, for any $x\in O$,
\begin{equation}\label{>alpha}
(M_N)_{\mathbb{A}}\left(\vec f\right)(x)>\alpha
\end{equation}
and, for any $x\in O^\complement$,
\begin{equation}\label{<alpha}
(M_N)_{\mathbb{A}}\left(\vec f\right)(x)\le C\alpha.
\end{equation}

\item If $O\neq\emptyset$,
then there exist $\vec g\in(\mathcal S')^m$,
$\{\vec b_k\}_{k\in\mathbb N}\subset(\mathcal{S}')^m$, and a sequence
$\{Q_k^*\}_{k\in\mathbb N}$ of closed cubes such that
$O=\bigcup_{k\in\mathbb N}Q_k^*$ and
\begin{align}\label{except1}
\vec f=\vec g+\vec b
:=\vec g+ \sum_{k\in\mathbb N}\vec b_k
\end{align}
in $(\mathcal S')^m$.
For any $k\in{\mathbb N}$, $\operatorname{supp}\vec b_k\subset Q_k^*$,
\begin{align}\label{3.4x}
\left\langle\vec b_k, \cdot^\gamma\right\rangle=\vec 0
\end{align}
for any $\gamma\in\mathbb Z_+^n$ with $|\gamma|\le s$, and
\begin{align}\label{estimate-f}
\int_{\mathbb R^n}\left[M_W\left(\vec b_k,\psi\right)(x)\right]^p\,dx
\le C\int_{Q^*_k}\left[(M_N)_{\mathbb{A}}\left(\vec f\right)(x)\right]^p\,dx.
\end{align}

\item If $\vec f\in(L^1_{\rm{loc}})^m$,
then $\vec g\in(L^1_{\rm{loc}})^m$, $\{\vec b_k\}_{k\in\mathbb N}\subset(L^1_{\rm{loc}})^m$,
and there exist $\{\vec g_k\}_{k\in\mathbb N}\subset(L^1_{\rm{loc}})^m$ and
$\vec g_0\in(L^1_{\rm{loc}})^m$ such that, for any $x\in O$,
\begin{align}\label{except}
\vec f(x)
=\vec g(x)+\vec b(x)
=\sum_{k\in\mathbb N}\vec g_k(x) + \sum_{k\in\mathbb N}\vec b_k(x)
\end{align}
pointwisely and, for any $x\in O^\complement$,
\begin{align}\label{g0}
\vec g_0(x)=\vec g(x)=\vec f(x).
\end{align}
Moreover, for any $k\in{\mathbb N}$,
$\operatorname{supp}\vec g_k\subset Q_k^*$ and
\begin{align}\label{estimate-g}
\sup_{x\in Q_k^*}\left|A_{Q_k^*}\vec g_k(x)\right|\le C\alpha.
\end{align}

\end{enumerate}
\end{theorem}

To prove Theorem \ref{l-CZ}, we need several lemmas.
The following lemma is extensively used
in the proofs of Theorems \ref{W-F-atom} and \ref{l-CZ} and
Lemma \ref{dense}.

\begin{lemma}\label{le<}
Let $N\in\mathbb Z_+$ and $\mathbb{A}:=\{A_Q\}_{Q\in\mathscr{Q}}$
be a family of positive definite matrices. If $\varphi\in\mathcal{S}$
satisfies that $\operatorname{supp}\varphi\subset B(x_0,t)$
with $x_0\in\mathbb R^n$ and $t\in(0,\infty)$ and if,
for any $\alpha\in\mathbb Z^n_+$ with $|\alpha|\le N+1$,
\begin{align}\label{le-w-<}
\sup_{x\in\mathbb R^n}\left|\partial^\alpha\varphi(x)\right|
\le t^{-(n+|\alpha|)},
\end{align}
then, for any $\vec f\in(\mathcal{S}')^m$ and $x\in\mathbb R^n$,
\begin{equation}\label{<MN}
\left|A_t(x)\left\langle \vec f,\varphi\right\rangle\right|
\le \left(2+t^{-1}|x-x_0|\right)^{N+n+1}
(M_N)_{\mathbb{A}}\left(\vec f\right)(x),
\end{equation}
where $A_t$ is the same as in \eqref{Atref}.
\end{lemma}

\begin{proof}
Let $C_{(x)}:=(2+\frac{|x-x_0|}t)^{N+n+1}$ and
$\phi^{(x)}(\cdot):=\frac{t^n}{C_{(x)}}\varphi(x-t\cdot)$.
Then, obviously, we have, for any $x\in\mathbb R^n$,
$|A_t(x)\langle \vec f,\varphi\rangle|
=C_{(x)}|A_t(x)\phi^{(x)}_t*\vec f(x)|$.
Therefore, by the definition of $(M_N)_{\mathbb{A}}(\vec f)$,
we conclude that, to show \eqref{<MN},
it suffices to prove $\phi^{(x)}\in\mathcal S_N$.
Using $\operatorname{supp}\varphi\subset B(x_0,t)$ and \eqref{le-w-<},
we find that, for any $\alpha\in\mathbb Z^n_+$ with $|\alpha|\le N+1$,
\begin{align*}
\sup_{y\in\mathbb R^n}(1+|y|)^{N+n+1}\left|\partial^\alpha\phi^{(x)}(y)\right|
&=\frac{t^{n+|\alpha|}}{C_{(x)}}\sup_{y\in B(x_0,t)}
\left(1+\frac{|x-y|}t\right)^{N+n+1}\left|\partial^\alpha\varphi(y)\right|\\
&\le\left(2+\frac{|x-x_0|}t\right)^{-(N+n+1)}
\sup_{y\in B(x_0,t)}\left(1+\frac{|x-x_0|+|x_0-y|}t\right)^{N+n+1}
\le1
\end{align*}
and hence $\phi^{(x)}\in\mathcal{S}_N$.
This finishes the proof of Lemma \ref{le<}.
\end{proof}

The following Whitney decomposition is exactly \cite[p.\,609, Proposition]{g14c}.
Let $\mathscr Q^{\mathrm{dy}}:=\bigcup_{j\in\mathbb Z}\mathscr Q_{2^{-j}}$
be the set of all dyadic cubes, where $\mathscr Q_{2^{-j}}$ is the same as in
\eqref{deQt} with $t:=2^{-j}$ for any $j\in\mathbb Z$.

\begin{lemma}\label{wd}
Let $\Omega$ be an open nonempty proper subset of $\mathbb R^n$.
Then there exists a sequence of closed dyadic cubes $\{Q_k\}_{k\in{\mathbb N}}$
such that the following statements hold.
\begin{enumerate}[{\rm(i)}]
\item
$\Omega=\bigcup_{k\in{\mathbb N}}Q_k$ and $\{Q_k\}_{k\in{\mathbb N}}$ have disjoint interiors.
\item
For any $k\in{\mathbb N}$, $\sqrt{n}l(Q_k)\le\mathrm{dist}(Q_k,\Omega^\complement)\le4\sqrt{n}l(Q_k)$
and hence $10\sqrt{n}Q_k\cap\Omega^\complement\neq\emptyset$.
\item
Let $k,j\in{\mathbb N}$.
If $Q_k\cap Q_j\neq\emptyset$, then
$\frac{1}{4}\le\frac{l(Q_k)}{l(Q_j)}\le4.$
\item
For any $j\in\mathbb N$,
$\#\{k\in\mathbb N:\ Q_k\cap Q_j\neq\emptyset\}\le 12^n-4^n$.
\item
For any $\lambda\in(1,\frac{5}{4})$,
$\sum_{k\in{\mathbb N}}\mathbf{1}_{\lambda Q_k}
\le\left(12^n-4^n+1\right)\mathbf{1}_{\Omega}.$
\end{enumerate}
\end{lemma}

For any $P\in\mathscr Q^{\mathrm{dy}}$, let its \emph{level}
$j_P$ be the unique integer such that
$P\in\mathscr{Q}_{2^{-j_P}}$.
Strongly depending on \cite[Theorem 3.6]{fr21},
we obtain the following sharp conclusion,
which is of independent interests.

\begin{proposition}\label{nazarov}
Let $p\in(0,\infty)$.
Then, for any sequence $\omega:=\{\omega_j\}_{j\in\mathbb Z}$
of non-negative measurable functions on ${\mathbb R}^n$ satisfying
\begin{equation}\label{gammaj}
\left\|\omega\right\|_{\mathcal K}:=
\sup_{Q\in\mathscr Q^{\mathrm{dy}}}\left\{\fint_Q
\sup_{j\in\mathbb Z,\,j\geq j_Q}[\omega_j(x)]^p\,dx\right\}^\frac1p\leq1,
\end{equation}
any $k\in\mathbb Z$, any $P\in\mathscr Q^{\mathrm{dy}}$,
and any sequence $\{f_Q\}_{Q\in\mathscr{Q}^{\mathrm{dy}}}$ of complex numbers,
\begin{align}\label{1202}
\left\|\sup_{j\in\mathbb Z,\,j\geq j_P+k}\omega_j|f_j|\right\|_{L^p(P)}
\le 2^\frac np\max\left\{1,2^{-\frac{kn}{p}}\right\}
\left\|\sup_{j\in\mathbb Z,\,j\geq j_P+k}|f_j|\right\|_{L^p(P)},
\end{align}
where, for any $j\in\mathbb Z$,
$f_j:=\sum_{Q\in\mathscr{Q}_{2^{-j}}}f_Q\mathbf{1}_Q$.
In addition, the constant in \eqref{1202} is sharp in the following sense:
For any $k\in\mathbb Z$,
if let $C_k:=\sup\|\sup_{j\in\mathbb Z,\,j\geq j_P+k}\omega_j|f_j|\,\|_{L^p(P)}
\|\sup_{j\in\mathbb Z,\,j\geq j_P+k}|f_j|\,\|_{L^p(P)}^{-1}$
with the first supremum taken over all $\omega$,
$\{f_j\}_{j\in\mathbb Z}$, and $P\in\mathscr Q^{\mathrm{dy}}$,
then $C_k\geq
\max\{1,2^{-\frac{kn}{p}}\}$.
\end{proposition}

Slightly different from \cite[Theorem 3.7(ii)]{fr21} on $\mathbb{R}^n$,
Proposition \ref{nazarov} is its dyadic variant on dyadic cubes,
which can be regarded as a substitute of the Fefferman--Stein vector-valued
maximal inequality in the case of matrix weights.

\begin{proof}[Proof of Proposition \ref{nazarov}]
If the present proposition holds in the case $k=0$, then,
when $k\in\mathbb N$, for any sequence $\omega:=\{\omega_j\}_{j\in\mathbb Z}$
of non-negative measurable functions on ${\mathbb R}^n$ satisfying \eqref{gammaj},
any sequence $\{f_Q\}_{Q\in\mathscr{Q}^{\mathrm{dy}}}$ of
complex numbers, and any $P\in\mathscr Q^{\mathrm{dy}}$,
\begin{align*}
\left\|\sup_{j\in\mathbb Z,\,j\geq j_P+k}\omega_j|f_j|\right\|_{L^p(P)}
&=\left\|\sup_{j\in\mathbb Z,\,j\geq j_P}\omega_j
|f_j|\mathbf{1}_{j\geq j_P+k}\right\|_{L^p(P)}\\
&\leq2^\frac np\left\|\sup_{j\in\mathbb Z,\,j\geq j_P}
|f_j|\mathbf{1}_{j\geq j_P+k}\right\|_{L^p(P)}
=2^\frac np\left\|\sup_{j\in\mathbb Z,\,j\geq j_P+k}|f_j|\right\|_{L^p(P)}\nonumber
\end{align*}
and hence the present proposition when $k\in\mathbb N$ also holds.
Thus, without loss of generality, in the remainder
of the present proof we may assume $k\le 0$.
Let $\omega:=\{\omega_j\}_{j\in\mathbb Z}$ be any sequence
of non-negative measurable functions on ${\mathbb R}^n$ satisfying \eqref{gammaj},
$\{f_Q\}_{Q\in\mathscr{Q}^{\mathrm{dy}}}$
be any sequence of complex numbers, and $P\in\mathscr Q^{\mathrm{dy}}$.
For any $Q\in\mathscr Q^{\mathrm{dy}}$, let
\begin{align}\label{gQ}
g_Q:=\begin{cases}
f_Q&\text{if }Q\cap P\neq\emptyset\text{ and }j_Q\geq j_P+k,\\
0&\text{otherwise}.
\end{cases}
\end{align}
For any $j\in\mathbb Z$, let
$g_j:=\sum_{Q\in\mathscr{Q}_{2^{-j}}}g_Q\mathbf{1}_Q$.
Then
\begin{equation}\label{first est}
\left\|\sup_{j\in\mathbb Z,\,j\geq j_P+k}\omega_j|f_j|\right\|_{L^p(P)}
\le\left\|\sup_{j\in\mathbb Z,\,j\geq j_P+k}\omega_j|g_j|\right\|_{L^p}.
\end{equation}
By \cite[Lemma 3.6]{fr21}, the fact that, for any
$j\in\mathbb Z$, $|g_j|^p\mathbf{1}_{j\geq j_P+k}$
is constant on each $Q\in\mathscr{Q}_{2^{-j}}$, and \eqref{gammaj},
we conclude that
\begin{align}\label{nazarov equ}
\left\|\sup_{j\in\mathbb Z,\,j\geq j_P+k}\omega_j|g_j|\right\|_{L^p}^p
&=\left\|\sup_{j\in\mathbb Z}\omega_j^p|g_j|^p
\mathbf{1}_{j\geq j_P+k}\right\|_{L^1}\\
&\le\|\omega\|_{\mathcal K}^p\left\|\sup_{j\in\mathbb Z}|g_j|^p
\mathbf{1}_{j\geq j_P+k}\right\|_{L^1}
\leq\left\|\sup_{j\in\mathbb Z,\,j\geq j_P+k}|g_j|
\right\|_{L^p}^p.\notag
\end{align}
In addition, from the definition of $\{g_j\}_{j\in\mathbb Z}$,
we infer that
\begin{align*}
\left\|\sup_{j\in\mathbb Z,\,j\geq j_P+k}|g_j|\right\|_{L^p}^p
&\leq\int_{P}\sup_{j\geq j_P}|g_j(x)|^p\,dx+\sum_{j=j_P+k}^{j_P-1}
\int_{\mathbb R^n}|g_j(x)|^p\,dx\\
&\leq\sum_{i=0}^{-k}2^{in}\int_{P}\sup_{j\geq j_P+k}|f_j|^p\,dx
\le2^{(-k+1)n}\left\|\sup_{j\in\mathbb Z,\,j\geq j_P+k}|f_j|\right\|_{L^p(P)}^p.
\end{align*}
Combining this, \eqref{first est}, and \eqref{nazarov equ},
we find \eqref{1202} holds.

Now, we show that the constant in \eqref{1202} is sharp.
If, for any $j\in\mathbb Z$, $\omega_j\equiv 1$,
then it is easy to prove that $\omega:=\{\omega_j\}_{j\in\mathbb Z}$
satisfies \eqref{gammaj} and, for any $k\in\mathbb Z$, $C_k\geq1$.
Let $Q_0$ be a fixed dyadic cube with edge length $1$.
If, for any $j\in\mathbb Z$, $\omega_j:=2^{-\frac{jn}{p}}\mathbf 1_{Q_0}$
and $f_j\equiv 1$, then $\omega:=\{\omega_j\}_{j\in\mathbb Z}$
satisfies \eqref{gammaj} and, for any $k\in\mathbb Z$,
\begin{align*}
\left\|\sup_{j\in\mathbb Z,\,j\geq k}\omega_j|f_j|\right\|_{L^p(Q_0)}
\left\|\sup_{j\in\mathbb Z,\,j\geq k}|f_j|\right\|_{L^p(Q_0)}^{-1}
=\left(\int_{Q_0}2^{-kn}\,dx\right)^\frac1p
\left(\int_{Q_0}1\,dx\right)^{-\frac1p}=2^{-\frac{kn}{p}}
\end{align*}
and hence $C_k\geq2^{-\frac{kn}{p}}$.
This finishes the proof of the sharpness of \eqref{1202} and hence Proposition \ref{nazarov}.
\end{proof}

As a simple application of Proposition \ref{nazarov},
we obtain the following conclusion.

\begin{corollary}\label{nazarov cor}
Let $p\in(0,\infty)$ and $W\in A_{p,\infty}$.
Then there exists a positive constant $C$ such that,
for any $k\in\mathbb Z$, any $P\in\mathscr Q^{\mathrm{dy}}$,
and any sequence $\{f_Q\}_{Q\in\mathscr{Q}^{\mathrm{dy}}}$ of complex numbers,
\begin{align*}
\left\|\sup_{j\in\mathbb Z,\,j\geq j_P+k}\gamma_{2^{-j}}|f_j|\right\|_{L^p(P)}
\le C\max\left\{1,2^{-\frac{kn}{p}}\right\}
\left\|\sup_{j\in\mathbb Z,\,j\geq j_P+k}|f_j|\right\|_{L^p(P)},
\end{align*}
where, for any $j\in\mathbb Z$, $\gamma_{2^{-j}}$
is the same as in \eqref{2.28x}
and $f_j:=\sum_{Q\in\mathscr{Q}_{2^{-j}}}f_Q\mathbf{1}_Q$.
\end{corollary}

\begin{proof}
By \cite[Corollary 5.7(ii)]{bhyyNew}, we find that
$\{\omega_j:=\gamma_{2^{-j}}\}_{j\in\mathbb Z}$ satisfies \eqref{gammaj}
and hence the present corollary is immediately
deduced from Proposition \ref{nazarov}.
This finishes the proof of Corollary \ref{nazarov cor}.
\end{proof}

We also need the following
approximation of the identity.

\begin{lemma}\label{<sup}
Let $\psi\in\mathcal{S}$ satisfy
$\operatorname{supp}\psi \subset B(\mathbf{0},1)$ and
$\int_{{\mathbb{R}^n}}\psi(x)\,dx=1$. Assume that $\vec f\in(L^1_{\mathrm{loc}})^m$
and $H:\ \mathbb{R}^n\to M_m(\mathbb{C})$ is a matrix-valued function
satisfying $\operatorname{ess\,sup}_{x\in{\mathbb{R}^n}}\|H(x)\|<\infty$.
Then, for almost every $x\in{\mathbb{R}^n}$,
$$\lim_{t\to0^+}H(x)\psi_t*\vec f(x)=
H(x)\vec f(x).
$$
\end{lemma}

\begin{proof}
Using the Lebesgue differentiation theorem, we conclude that,
for any  $t\in(0,\infty)$
and almost every $x\in\mathbb R^n$,
\begin{align*}
\left|H(x)\psi_t*\vec f(x)-H(x)\vec f(x)\right|
&\le \int_{\mathbb R^n}\left|\frac1{t^n}\psi\left(\frac yt\right) \right|
\left|H(x)\vec f(x-y)-H(x)\vec f(x)\right|\,dy\\
&\lesssim\fint_{B(\mathbf{0},t)}\left|\vec f(x-y)-\vec f(x)\right|\,dy
\to 0
\end{align*}
as $t\to 0^+$.
This finishes the proof of Lemma \ref{<sup}.
\end{proof}

Next, using the above several lemmas,
we show Theorem \ref{l-CZ}.

\begin{proof}[Proof of Theorem \ref{l-CZ}]
We first prove (i).
Let $\vec f\in H^p_W$ and $O:=\overline{\{x\in{\mathbb{R}^n}:
\ (M_N)_{\mathbb{A}}(\vec f)(x)\le\alpha\}}^\complement.$
Then $O$ is open and \eqref{>alpha} holds.
By \eqref{>alpha} and Theorems \ref{if and only if A}
and \ref{weight and reducing}, we obtain
\begin{align}\label{pro-O}
|O|
\le\frac1{\alpha^p}\int_O\left|(M_N)_{\mathbb{A}}
\left(\vec f\right)(x)\right|^p\,dx
\le\frac1{\alpha^p}\left\|(M_N)_{\mathbb{A}}
\left(\vec f\right)\right\|_{L^p}^p
\sim\frac1{\alpha^p}\left\|\vec f\right\|_{H^p_W}^p
<\infty,
\end{align}
and hence $O$ is a proper subset of $\mathbb R^n$.
Now, we verify \eqref{<alpha}.
Let
$x\in O^\complement
=\overline{\{x\in\mathbb R^n:\
(M_N)_{\mathbb{A}}(\vec f)(x)\le\alpha\}}$.
Then there exists a sequence $\{x_i\}_{i\in{\mathbb N}}\subset{\mathbb{R}^n}$
such that, for any $i\in{\mathbb N}$, $(M_N)_{\mathbb{A}}(\vec f)(x_i)\le\alpha$
and $\lim_{i\to\infty} x_i=x$.
From this, we infer that,
for any $t\in(0,\infty)$ and $\phi\in\mathcal{S}_N$,
there exists $K\in\mathbb N$ such that,
for any $i\geq K$, $|x-x_i|<t$,
which, together with Lemma \ref{sharp estimate} and
the definition of $(M_N)_{\mathbb{A}}(\vec f)$, further implies that
\begin{align}\label{3.18x}
\left|A_{t}(x)\phi_t*\vec f(x)\right|
&\le\left|A_{t}(x)\phi_t*\vec f(x_i)\right|+
\left|A_{t}(x)\left[\phi_t*\vec f(x)-\phi_t*\vec f(x_i)\right]\right|\\
&\le\left\|A_{t}(x)A^{-1}_{t}(x_i)\right\|
\left|A_{t}(x_i)\phi_t*\vec f(x_i)\right|+
\left|A_{t}(x)\left[\phi_t*\vec f(x)-\phi_t*\vec f(x_i)\right]\right|\nonumber\\
&\lesssim\alpha+\left|A_{t}(x)\left[\phi_t*\vec f(x)-\phi_t*\vec f(x_i)\right]\right|,\nonumber
\end{align}
where $A_t$ is the same as in \eqref{Atref}. In addition,
by $\phi_t*\vec f\in (C^\infty)^m$ which can be deduced from
\cite[Theorem 2.3.20]{g14c}, we find that
$$
\lim_{i\to\infty} \left|A_{t}(x)\left[\phi_t*\vec f(x)-\phi_t*\vec f(x_i)\right]\right|=0,
$$
which, together with \eqref{3.18x}, further implies that
\eqref{<alpha} holds.
This finishes the proof of (i).

Next, we show (ii).
Using the above proven fact that $O$ is an open nonempty proper subset of $\mathbb R^n$ and Lemma \ref{wd},
we conclude that there exists a sequence
$\{Q_k\}_{k\in\mathbb N}$ of closed dyadic cubes
such that
\begin{enumerate}[{\rm(a)}]
\item
$O=\bigcup_{k\in{\mathbb N}}Q_k$,
\item
$10\sqrt{n}Q_k\cap O^\complement\neq\emptyset$, and
\item
for any $\lambda\in(1,\frac{5}{4})$,
$\sum_{k\in{\mathbb N}}\mathbf{1}_{\lambda Q_k}
\le(12^n-4^n+1)\mathbf{1}_{O}$.
\end{enumerate}
For any $k\in\mathbb N$, let $c_k$ be the center of $Q_k$
and $l_k$ the edge length of $Q_k$.
Let $1 <\widetilde a < a^*:=\frac98<\frac{5}{4}$. Then
$Q_k \subset \widetilde{Q}_k:= \widetilde{a}Q_k \subset Q^*_k:= a^*Q_k$.
From (a) and (c), we infer that $O=\bigcup_{k\in{\mathbb N}}Q_k^*$
and
\begin{align}\label{jrfzwsw}
\sum_{k\in{\mathbb N}}\mathbf{1}_{Q^*_k}
\le (12^n-4^n+1)\mathbf{1}_O.
\end{align}
Using the estimate in \cite[p.\,102]{S93}, we conclude that there exists a sequence of functions
\begin{align}\label{de-eta}
\left\{\eta_k:\ \mathbb R^n\to [0,1]\right\}_{k\in{\mathbb N}}
\subset C_{\mathrm c}^\infty,
\end{align}
where $C_{\mathrm c}^\infty$
denotes the set of all infinitely differentiable functions
on $\mathbb R^n$ with compact support,
such that
\begin{align}\label{dlsl}
\mathbf{1}_{O}(x)=\sum_{k\in{\mathbb N}}\eta_k(x)\ \text{for any}\ x\in\mathbb R^n,
\end{align}
\begin{align}\label{supp-eta}
\operatorname{supp}\eta_k\subset \widetilde{Q}_k\ \text{for any }k\in\mathbb N,
\end{align}
\begin{align}\label{betaeta}
\sup_{x\in{\mathbb{R}^n}}\left|\partial^\alpha\eta_k(x)\right|\lesssim
\left[ l(Q_k)\right]^{-|\alpha|}\ \text{for any } k\in\mathbb N\text{ and }\alpha\in\mathbb Z^n_+,
\end{align}
where the implicit positive constant depends only on $\alpha$,
and
\begin{align}\label{betaeta2}
\int_{\mathbb R^n}\eta_k(x)\,dx\sim|Q_k|\ \text{for any }k\in{\mathbb N},
\end{align}
where the positive equivalence constants depend only on $n$.
For any $k\in\mathbb N$, let
\begin{align}\label{w-eta}
\widetilde{\eta}_k:=\frac{\eta_k}{\int_{{\mathbb{R}^n}}\eta_k(x)\,dx}.
\end{align}
Fix $k\in\mathbb N$. For any $s\in\mathbb Z_+$,
let $\mathcal{P}_s$ be the set of all polynomials
of total degree not greater than $s$ on ${\mathbb{R}^n}$.
Let $\mathcal{H}_s$ be $\mathcal{P}_s$ considered as a subspace of
the Hilbert space $L^2(Q^*_k,\widetilde{\eta}_k\,dx)$.
Let $M\in{\mathbb N}$ and $\{e_i\}_{i=1}^M$ be polynomials forming an orthonormal basis of $\mathcal{H}_s$,
that is, for any $i,j\in\{1,\ldots,M\}$,
\begin{align}\label{defi-t}
\left\langle e_i,e_j\widetilde{\eta}_k\right\rangle
= \begin{cases}
1&\text{if } i=j,\\
0&\text{otherwise}.
\end{cases}
\end{align}
Using \cite[p.\,104,\ (28)]{S93}, we find that,
for any $j\in\{1,\ldots,M\}$ and $\alpha\in\mathbb Z^n_+$,
\begin{align}\label{qqq}
\sup_{x\in Q_k^*}\left|\partial^\alpha e_j(x)\right|
\lesssim [ l(Q_k)]^{-|\alpha|},
\end{align}
where the implicit positive constant depends only on $\alpha$.
Let $\vec b_k:=(\vec f-P_k\vec f)\eta_k,$
where $P_k$ is the projection operator defined by setting
\begin{align}\label{de-Pk}
P_k\vec f:=\sum_{j=1}^{M}\left\langle\vec f,
e_j\widetilde{\eta}_k\right\rangle e_j.
\end{align}
From \eqref{supp-eta}, we deduce that
$\operatorname{supp}\vec b_k
\subset\operatorname{supp}\eta_k
\subset\widetilde{Q}_k$.
By \eqref{defi-t}, we find that,
for any $q:=\sum_{j=1}^{M}c_je_j\in\mathcal{P}_s$
and $\vec h\in(\mathcal{S}')^m$,
\begin{align}\label{qg}
\left\langle \left(P_k\vec h-\vec h\right)\eta_k, q\right\rangle
&=\sum_{j=1}^{M}c_j\left\langle P_k\vec h, e_j\eta_k\right\rangle-
\sum_{j=1}^{M}c_j\left\langle\vec h, e_j\eta_k\right\rangle\\
&=\sum_{j=1}^{M}c_j\left\langle\vec h, e_j\widetilde \eta_k\right\rangle
\left\langle e_j, e_j\eta_k\right\rangle-
\sum_{j=1}^{M}c_j\left\langle\vec h, e_j\eta_k\right\rangle
=\vec 0\notag
\end{align}
and hence
\begin{align}\label{vanish}
\left\langle\vec b_k, q\right\rangle
=\left\langle \left(P_k\vec f-\vec f\right)\eta_k, q\right\rangle
=\vec 0,
\end{align}
which further implies that, for any $\gamma\in\mathbb Z_+^n$ with $|\gamma|\le s$,
$\langle\vec b_k, \cdot^\gamma\rangle=\vec 0$.

Next, we estimate $\vec b_k$.
From the definition of $\vec b_k$, we infer that
\begin{align*}
\int_{\mathbb R^n}\left[M_W\left(\vec b_k,\psi\right)(x)\right]^p\,dx
&\le\int_{Q_k^*}\left[M_{W}\left(\left[
P_k\vec f\right]\eta_k,\psi\right)(x)\right]^p\,dx
+\int_{Q_k^*}\left[M_{W}\left(\vec f\eta_k,\psi\right)(x)\right]^p\,dx\\
&\quad
+\int_{(Q_k^*)^\complement}
\left[M_{W}\left(\vec b_k,\psi\right)(x)\right]^p\,dx\\
&=:\mathrm{I}+\mathrm{II}+\mathrm{III}.
\end{align*}

We first estimate $\mathrm{I}$. For any $t\in(0,\infty)$ and $x\in Q_k^*$,
by the definition of $P_k$, we obtain
\begin{align*}
\left|W^{\frac1p}(x)\psi_t*\left[\left(P_k\vec f\right)\eta_k\right](x)\right|
\le\left\|W^{\frac1p}(x)\left[A_{l_k}(c_k)\right]^{-1}\right\|
\sum_{j=1}^{M}\left|A_{l_k}(c_k)\left\langle\vec f,e_j\widetilde{\eta}_k\right\rangle\right|
\left|\psi_t*\left(e_j\eta_k\right)(x)\right|
\end{align*}
and, from \eqref{betaeta} and \eqref{qqq}, we deduce that
\begin{align}\label{phit*e}
\left|\psi_t*\left(e_j\eta_k\right)(x)\right|
\le\int_{\mathbb R^n}|\psi_t(y)|\left|e_j(x-y)\right||\eta_k(x-y)|\,dx
\lesssim\|\psi\|_{L^1}.
\end{align}
These, together with Proposition \ref{reduceM},
Lemma \ref{sharp estimate},
and \eqref{control}, further imply that
\begin{align*}
\mathrm{I}
&\lesssim\int_{Q^*_k}\left\|W^{\frac1p}(x)\left[A_{l_k}(c_k)\right]^{-1}\right\|^p\,dx
\sum_{j=1}^{M}\left|A_{l_k}(c_k)\left\langle\vec f,e_j\widetilde{\eta}_k\right\rangle\right|^p\\
&\lesssim\left\|A_{Q^*_k}\left[A_{l_k}(c_k)\right]^{-1}\right\|^p
\sum_{j=1}^{M}\int_{Q^*_k}\left\|A_{l_k}(c_k)\left[A_{r_k}(y)\right]^{-1}\right\|^p
\left|A_{r_k}(y)\left\langle\vec f,e_j\widetilde{\eta}_k\right\rangle\right|^p\,dy\nonumber\\
&\lesssim\int_{Q^*_k}\left(2+r_k^{-1}|c_k-y|\right)^{N+n+1}
\left[(M_N)_{\mathbb{A}}\left(\vec f\right)(y)\right]^p\,dy
\lesssim\int_{Q^*_k}\left[(M_N)_{\mathbb{A}}\left(\vec f\right)(y)\right]^p\,dy,\nonumber
\end{align*}
where $r_k:=\sqrt{n}\widetilde{a}l_k$.
This finishes the estimation of $\mathrm{I}$.

Now, we estimate $\mathrm{II}$. Notice that
\begin{align*}
\mathrm{II}
\le\int_{Q_k^*}\sup_{t\in(0,l_k]}\left|W^{\frac1p}(x)
\psi_t*\left(\vec f\eta_k\right)(x)\right|^p\,dx
+\int_{Q_k^*}\sup_{t\in(l_k,\infty)}\cdots\,dx
=:\mathrm{II}_1+\mathrm{II}_2.
\end{align*}

We first estimate $\mathrm{II}_1$.
For any $k\in{\mathbb N}$, let $j_k:=-\log_2 l_k$.
Using $a^*=\frac98$, we obtain
$Q_k^*=\frac98Q_k=\bigcup_{R\in\mathscr Q_{l_k/16},\, R\subset Q_k^*}R$.
By Corollary \ref{nazarov cor} and Lemma \ref{sharp estimate},
we conclude that
\begin{align}\label{II1}
\mathrm{II}_1
&=\sum_{R\in\mathscr Q_{l_k/16},\, R\subset Q_k^*}\int_R
\sup_{\{j\in\mathbb Z:\,j\geq j_R-4\}}\sup_{t\in(2^{-(j+1)},2^{-j}]}
\left|W^{\frac1p}(x)
\psi_t*\left(\vec f\eta_k\right)(x)\right|^p\,dx\\
&\le\sum_{R\in\mathscr Q_{l_k/16},\, R\subset Q_k^*}\int_R
\sup_{\{j\in\mathbb Z:\,j\geq j_R-4\}}
\left\|W^{\frac1p}(x)\left[A_{2^{-j}}(x)\right]^{-1}\right\|^p\notag
\sup_{t\in(2^{-(j+1)},2^{-j}]}\left\|A_{2^{-j}}(x)\left[A_t(x)\right]^{-1}\right\|\\
&\quad\times
\sum_{Q\in\mathscr Q_{2^{-j}}}\sup_{y\in Q}\left|A_t(x)\psi_t*\left(\vec f\eta_k\right)(y)\right|^p\mathbf{1}_Q(x)\,dx\notag\\
&\lesssim\int_{Q_k^*} \sup_{\{j\in\mathbb Z:\,j\geq j_k\}}\sup_{t\in(2^{-(j+1)},2^{-j}]}
\sum_{Q\in\mathscr Q_{2^{-j}}}\sup_{y\in Q}\
\left|A_t(x)\left\langle\vec f,\zeta_{(t,y)}\right\rangle\right|^p\mathbf{1}_Q(x)\,dx,\notag
\end{align}
where $\zeta_{(t,y)}(\cdot):=\psi_t(y-\cdot)\eta_k(\cdot)$.
Let $j\in\mathbb Z$ satisfy $j\geq j_k$,
$t\in(2^{-(j+1)},2^{-j}]$, $Q\in \mathscr Q_{2^{-j}}$, and $x,y\in Q$.
From
\eqref{de-eta},
$\psi\in\mathcal{S}$,
$\operatorname{supp}\psi\subset B(\mathbf{0},1)$,
\eqref{betaeta},
and $t\in(0,l_k]$, we infer that
$\zeta_{(t,y)}\in\mathcal{S}$ satisfies that
$\operatorname{supp}\zeta_{(t,y)}
\subset\operatorname{supp}[\psi_t(y-\cdot)]
\subset B(y,t)$
and, for any $\alpha:=(\alpha_1,\ldots,\alpha_n)\in\mathbb Z^n_+$ with $|\alpha|\le N+1$,
\begin{align*}
\sup_{z\in\mathbb R^n}\left|\partial^\alpha\zeta_{(t,y)}(z)\right|
&\sim\sup_{z\in\mathbb R^n}\left|\sum_{\beta\in\mathbb Z^n_+,\,\beta\le\alpha}
t^{-(n+|\beta|)}\partial^\beta\psi\left(\frac{y-z}t\right)
\partial^{\alpha-\beta}\eta_k(z)\right|\\
&\lesssim \sum_{\beta\in\mathbb Z^n_+,\,\beta\le\alpha}
t^{-(n+|\beta|)} l_k^{-|\alpha-\beta|}
\lesssim t^{-(n+|\alpha|)},
\end{align*}
where, for any $\beta:=(\beta_1,\ldots,\beta_n)\in\mathbb Z^n_+$,
$\beta\le\alpha$ means that, for any $i\in\{1,\ldots,n\}$, $\beta_i\le\alpha_i$.
These, together with Lemma \ref{le<}, further imply that
\begin{align*}
\left|A_{t}(x)\left\langle\vec f,\zeta_{(t,y)}\right\rangle\right|
\lesssim \left(2+\frac{|x-y|}t\right)^{N+n+1}
(M_N)_{\mathbb{A}}\left(\vec f\right)(x)
\sim (M_N)_{\mathbb{A}}\left(\vec f\right)(x).
\end{align*}
By this and \eqref{II1}, we find that
\begin{align*}
\mathrm{II}_1
\lesssim\int_{Q_k^*}\left[(M_N)_{\mathbb{A}}\left(\vec f\right)(x)\right]^p\,dx.
\end{align*}
This finishes the estimation of $\mathrm{II}_1$.

Next, we estimate $\mathrm{II}_2$.
Let $x\in Q_k^*$ and $t\in(l_k,\infty)$, and we then have
\begin{align}\label{II2}
\left|W^{\frac1p}(x)\psi_t*\left(\vec f\eta_k\right)(x)\right|
\le\left\|W^{\frac1p}(x)\left[A_{l_k}(c_k)\right]^{-1}\right\|
\left|A_{l_k}(c_k)\left\langle\vec f,\zeta_{(t,x)}\right\rangle\right|.
\end{align}
From
\eqref{de-eta},
$\psi\in\mathcal{S}$,
\eqref{supp-eta},
\eqref{betaeta},
and $t\in(l_k,\infty)$, we deduce that
$\zeta_{(t,x)}\in\mathcal{S}$ satisfies that
$$
\operatorname{supp}\zeta_{(t,x)}
\subset\operatorname{supp}\eta_k
\subset \widetilde{Q}_k
\subset B(c_k,r_k),
$$
where $r_k:=\sqrt{n}\widetilde{a}l_k$, and,
for any $\alpha\in\mathbb Z^n_+$ with $|\alpha|\le N+1$,
\begin{align*}
\sup_{y\in\mathbb R^n}\left|\partial^\alpha\zeta_{(t,x)}(y)\right|
&\sim\sup_{y\in\mathbb R^n}\left|\sum_{\beta\in\mathbb Z^n_+,\,\beta\le\alpha}
t^{-(n+|\beta|)}\partial^\beta\psi\left(\frac{x-y}t\right)
\partial^{\alpha-\beta}\eta_k(y)\right|\\
&\lesssim \sum_{\beta\in\mathbb Z^n_+,\,\beta\le\alpha}
t^{-(n+|\beta|)} l_k^{-|\alpha-\beta|}
\lesssim l_k^{-(n+|\alpha|)}.
\end{align*}
These, together with Lemmas \ref{sharp estimate} and \ref{le<},
further imply that, for any
$y\in 10\sqrt{n}Q_k=\frac{80}{9}\sqrt{n}Q^*_k$,
\begin{align}\label{1005}
\left|A_{l_k}(c_k)\left\langle\vec f,\zeta_{(t,x)}\right\rangle\right|
\le\left\|A_{l_k}(c_k)\left[A_{r_k}(y)\right]^{-1}\right\|
\left|A_{r_k}(y)\left\langle\vec f,\zeta_{(t,x)}\right\rangle\right|
\lesssim(M_N)_{\mathbb{A}}\left(\vec f\right)(y).
\end{align}
Taking the average with respect to $y$ over $Q^*_k$
on both sides of the above inequality, we obtain
$$
\left|A_{l_k}(c_k)\left\langle\vec f,\zeta_{(t,x)}\right\rangle\right|^p
\lesssim\fint_{Q_k^*}\left[(M_N)_{\mathbb{A}}\left(\vec f\right)(y)\right]^p\,dy.
$$
By this, \eqref{II2}, Proposition \ref{reduceM},
and Lemma \ref{sharp estimate}, we find that
\begin{align*}
\mathrm{II}_2
&\lesssim\int_{Q_k^*}\left\|W^{\frac1p}(x)\left[A_{l_k}(c_k)\right]^{-1}\right\|^p\,dx
\fint_{Q_k^*}\left[(M_N)_{\mathbb{A}}\left(\vec f\right)(y)\right]^p\,dy\\
&\sim
\left\|A_{Q_k^*}\left[A_{l_k}(c_k)\right]^{-1}\right\|^p
\int_{Q_k^*}\left[(M_N)_{\mathbb{A}}\left(\vec f\right)(y)\right]^p\,dy
\lesssim\int_{Q_k^*}\left[(M_N)_{\mathbb{A}}\left(\vec f\right)(y)\right]^p\,dy.
\end{align*}
This finishes the estimation of $\mathrm{II}_2$.
Combining the estimates of $\mathrm{II}_1$ and $\mathrm{II}_2$, we obtain
\begin{align*}
\mathrm{II}
\le\mathrm{II}_1+\mathrm{II}_2
\lesssim\int_{Q_k^*}\left[(M_N)_{\mathbb{A}}\left(\vec f\right)(x)\right]^p\,dx,
\end{align*}
which completes the estimation of $\mathrm{II}$.

Now, we estimate $\mathrm{III}$.
Let $q_{(t,x)}$ be the $s$th-degree Taylor polynomial of $\psi_t(x-\cdot)$
centered at $c_k$.
From \eqref{vanish} and the definition of $\vec b_k$, we infer that,
for any $x\in(Q^*_k)^\complement$,
\begin{align}\label{3sbtyssr}
\sup_{t\in(0,\infty)}
\left|A_{l_k}(c_k)\psi_t*\vec b_k(x)\right|&=\sup_{t\in(0,\infty)}
\left|A_{l_k}(c_k)\left\langle\vec b_k,
\psi_t(x-\cdot)-q_{(t,x)}(\cdot)\right\rangle
\right|\\
&\le\sup_{t\in(0,\infty)}
\left|A_{l_k}(c_k)\left\langle\vec f,
\eta_k(\cdot)\left[\psi_t(x-\cdot)-q_{(t,x)}(\cdot)\right]\right\rangle
\right|\nonumber\\
&\quad+\sup_{t\in(0,\infty)}
\left|A_{l_k}(c_k)\left\langle P_k\vec f,
\eta_k(\cdot)\left[\psi_t(x-\cdot)-q_{(t,x)}(\cdot)\right]\right\rangle
\right|\nonumber\\
&=\sup_{t\in(0,\infty)}
\left|A_{l_k}(c_k)\left\langle\vec f,\phi_{(t,x)}\right\rangle\right|
+\sup_{t\in(0,\infty)}
\left|A_{l_k}(c_k)\left\langle P_k\vec f,\phi_{(t,x)}\right\rangle\right|\nonumber\\
&=:III_1(x)+III_2(x),\nonumber
\end{align}
where $\phi_{(t,x)}(\cdot):=\eta_k(\cdot)[\psi_t(x-\cdot)-q_{(t,x)}(\cdot)]$.

We first consider $III_1(x)$.
By \eqref{supp-eta}, we conclude that,
for any $t\in(0,\infty)$ and $x\in(Q^*_k)^\complement$,
\begin{equation}\label{supp phi}
\operatorname{supp}\phi_{(t,x)}
\subset\operatorname{supp}\eta_k
\subset \widetilde Q_k
\subset B(c_k,r_k),
\end{equation}
where $r_k:=\sqrt{n}\widetilde al_k$.
Using the formula in \cite[p.\,105]{S93}, we find that,
for any $\alpha\in{\mathbb Z}^n_+$, $t\in(0,\infty)$, $x\in(Q_k^*)^\complement$,
and $y\in{\mathbb{R}^n}$,
\begin{align}\label{pbeta}
\sup_{y\in{\mathbb{R}^n}}\left|\partial^\alpha\phi_{(t,x)}(y)\right|
\lesssim\frac{l_k^{n+s+1}}{|x-c_k|^{n+s+1}}l_k^{-(n+|\alpha|)}
\sim\frac{l_k^{n+s+1}}{|x-c_k|^{n+s+1}}r_k^{-(n+|\alpha|)},
\end{align}
where the implicit positive constants depend on $\alpha$
but are independent of $t$, $x$, and $l_k$.
These, together with Lemmas \ref{sharp estimate} and
\ref{le<}, further imply that,
for any $y\in\mathbb R^n$, $t\in(0,\infty)$, and $x\in(Q_k^*)^\complement$,
\begin{align}\label{<alpha2}
\left|A_{l_k}(y)\left\langle\vec f,\phi_{(t,x)}\right\rangle\right|
&\le\left\|A_{l_k}(y)\left[A_{r_k}(y)\right]^{-1}\right\|
\left|A_{r_k}(y)\left\langle\vec f,\phi_{(t,x)}\right\rangle\right|\\
&\lesssim\frac{l_k^{n+s+1}}{|x-c_k|^{n+s+1}}
\left(2+\frac{|y-c_k|}{l_k}\right)^{N+n+1}
(M_N)_{\mathbb{A}}\left(\vec f\right)(y).\notag
\end{align}
Due to (b), we are able to choose $y_k\in 10\sqrt{n}Q_k\cap O^\complement$.
From Lemma \ref{sharp estimate}, \eqref{<alpha2}, and \eqref{<alpha}, we deduce that,
for any $x\in(Q_k^*)^\complement$,
\begin{align}\label{estimate of III}
III_1(x)&\le
\left\|A_{l_k}(c_k)\left[A_{l_k}(y_k)\right]^{-1}\right\|
\sup_{t\in(0,\infty)}\left|A_{l_k}(y_k)\left\langle\vec f,\phi_{(t,x)}\right\rangle\right|\\
&\lesssim
\frac{l_k^{n+s+1}}{|x-c_k|^{n+s+1}}
(M_N)_{\mathbb{A}}\left(\vec f\right)(y_k)\lesssim
\frac{l_k^{n+s+1}}{|x-c_k|^{n+s+1}}\alpha.\notag
\end{align}
Next, we estimate $III_2(x)$.
From \eqref{de-eta}, the fact that $\{e_i\}_{i=1}^M$ are polynomials,
and \eqref{supp-eta},
we infer that $\varphi_{(j,k)}:=e_j\widetilde{\eta}_k
\in C_{\mathrm c}^\infty\subset\mathcal S$
satisfies that
$\operatorname{supp}\varphi_{(j,k)}
\subset\operatorname{supp}\eta_k
\subset \widetilde{Q}_k
\subset B(c_k,r_k)$,
where $r_k:=\sqrt{n}\widetilde{a}l_k$.
Using \eqref{betaeta}, \eqref{betaeta2}, \eqref{w-eta}, and \eqref{qqq},
we find that,
for any $\alpha\in\mathbb Z^n_+$ with $|\alpha|\le N+1$,
$$
\sup_{x\in\mathbb R^n}\left|\partial^\alpha\varphi_{(j,k)}(x)\right|
\sim\sup_{x\in\mathbb R^n}\left|\sum_{\beta\le\alpha}\partial^\alpha
e_j(x)\partial^{\alpha-\beta}\widetilde{\eta}_k(x)\right|
\lesssim\sum_{\beta\le\alpha}[ l(Q_k)]^{-n-|\beta|-|\alpha-\beta|}
\sim r_k^{-(n+|\alpha|)}.
$$
These, together with Lemma \ref{le<}, further imply that,
for any $x\in\mathbb R^n$,
\begin{equation}\label{control}
\left|A_{r_k}(x)\left\langle\vec f,\varphi_{(j,k)}\right\rangle\right|
\lesssim\left(2+r_k^{-1}|x-c_k|\right)^{N+n+1}
(M_N)_{\mathbb{A}}\left(\vec f\right)(x).
\end{equation}
By (b), we are able to choose $y_k\in 10\sqrt{n}Q_k\cap O^\complement$.
From this, Lemma \ref{sharp estimate}, \eqref{control}, and \eqref{<alpha},
we deduce that, for any $x\in Q_k^*$,
\begin{align}\label{hnt}
\left|A_{l_k}(x)\left\langle\vec f,\varphi_{(j,k)}\right\rangle\right|
\le\left\|A_{l_k}(x)\left[A_{r_k}(y_k)\right]^{-1}\right\|
\left|A_{r_k}(y_k)\left\langle\vec f,\varphi_{(j,k)}\right\rangle\right|
\lesssim(M_N)_{\mathbb{A}}\left(\vec f\right)(y_k)
\lesssim\alpha.
\end{align}
Applying this and the definition of $P_k$, we obtain, for any $x\in Q_k^*$,
\begin{align}\label{estimate of gk}
\left|A_{l_k}(x)P_k\vec f(x)\right|
\lesssim\sum_{j=1}^{M}\left|A_{l_k}(x)
\left\langle\vec f,e_j\widetilde{\eta}_k\right\rangle\right|
\lesssim\alpha.
\end{align}
By this, \eqref{supp phi}, Lemma \ref{sharp estimate}, and \eqref{pbeta},
we find that, for any $x\in(Q_k^*)^\complement$,
\begin{align}\label{Pk<alpha}
III_2(x)
&\le
\sup_{t\in(0,\infty)}\int_{Q^*_k}\left\|A_{l_k}(c_k)A_{l_k}(y)^{-1}\right\|
\left|A_{l_k}(y)P_k\vec f(y)\right|\left|\phi_{(t,x)}(y)\right|\,dy\\
&\lesssim
\int_{Q^*_k}\alpha\frac{l_k^{s+1}}{|x-c_k|^{n+s+1}}\,dy\sim
\frac{l_k^{n+s+1}}{|x-c_k|^{n+s+1}}\alpha.\nonumber
\end{align}
Combining \eqref{3sbtyssr}, \eqref{estimate of III}, and \eqref{Pk<alpha},
we conclude that, for any $x\in(Q_k^*)^\complement$,
\begin{align}\label{exin}
\sup_{t\in(0,\infty)}
\left|A_{l_k}(c_k)\psi_t*\vec b_k(x)\right|
\lesssim\frac{l_k^{n+s+1}}{|x-c_k|^{n+s+1}}\alpha,
\end{align}
which, together with the definition of $M_{W}(\vec b_k,\psi)$, further implies
that
\begin{align}\label{ex}
M_{W}\left(\vec b_k,\psi\right)(x)
\lesssim\left\|W^{\frac{1}{p}}(x)\left[A_{l_k}(c_k)\right]^{-1}\right\|
\frac{l_k^{n+s+1}}{|x-c_k|^{n+s+1}}\alpha.
\end{align}
Using $s\in(\lfloor n(\frac{1}{p}-1)+
d_{p,\infty}^{\mathrm{upper}}(W)\rfloor,\infty)$
and the definition of $d_{p,\infty}^{\mathrm{upper}}(W)$,
we find that there exists $d\in(d_{p,\infty}^{\mathrm{upper}}(W),\infty)$
such that $s\in(\lfloor n(\frac{1}{p}-1)+d\rfloor,\infty)$
and $d$ is an $A_{p,\infty}$-upper dimension of $W$.
These, together with \eqref{ex}, Proposition \ref{reduceM}, Lemma
\ref{sharp estimate}, and \eqref{>alpha}, further imply that
\begin{align*}
\text{III}
&\lesssim\sum_{i=1}^{\infty}
\left[\frac{l_k^{n+s+1}}{(2^il_k)^{n+s+1}}\alpha\right]^p
\int_{2^{i}Q^*_k\setminus2^{i-1}Q^*_k}
\left\|W^{\frac{1}{p}}(x)\left[A_{l_k}(c_k)\right]^{-1}\right\|^p\,dx\\
&\lesssim\sum_{i=1}^{\infty}
\left[\frac{\alpha}{2^{i(n+s+1)}}\right]^p\left(2^il_k\right)^n
\left\|A_{2^{i}Q^*_k}\left[A_{l_k}(c_k)\right]^{-1}\right\|^p\notag\\
&\lesssim\alpha^p\left|Q_k^*\right|\sum_{i=1}^{\infty}2^{-i(n+s+1)p}2^{in}2^{idp}
\sim\alpha^p\left|Q_k^*\right|
<\int_{Q^*_k}\left[(M_N)_{\mathbb{A}}\left(\vec f\right)(x)\right]^p\,dx.\nonumber
\end{align*}
This finishes the estimation of $\mathrm{III}$.

Combining the estimates of I, II, and III, we obtain
\begin{align*}
\int_{\mathbb R^n}\left[M_W\left(\vec b_k,\psi\right)(x)\right]^p\,dx
\lesssim\int_{Q^*_k}\left[(M_N)_{\mathbb{A}}\left(\vec f\right)(x)\right]^p\,dx.
\end{align*}
This finishes the estimation of $\vec b_k$.
Using this, Theorem \ref{if and only if W}, $p\in(0,1]$,
\eqref{jrfzwsw}, and \eqref{>alpha}, we conclude that
\begin{align}\label{bkcon}
\left\|\sum_{k\in\mathbb{N}}\vec b_k\right\|^p_{H^p_W}&\sim
\int_{{\mathbb{R}^n}}\left[M_{W}\left(\sum_{k\in\mathbb{N}}\vec b_k,\psi\right)(x)\right]^p\,dx
\le\sum_{k\in{\mathbb N}}\int_{{\mathbb{R}^n}}\left[M_{W}\left(\vec b_k,\psi\right)(x)\right]^p\,dx\\
&\lesssim\int_{\bigcup_{k\in{\mathbb N}}Q^{*}_k}\left[(M_N)_{\mathbb{A}}\left(\vec f\right)(x)\right]^p\,dx
\le\int_{\{(M_N)_{\mathbb{A}}(\vec f)>\alpha\}}
\left[(M_N)_{\mathbb{A}}\left(\vec f\right)(x)\right]^p\,dx,\nonumber
\end{align}
which further implies that
$\|\sum_{k\in\mathbb{N}}\vec b_k\|_{H^p_W}\lesssim
\|\vec f\|_{H^p_W}$
and hence
\begin{align}\label{0125}
\vec b=\sum_{k\in\mathbb{N}}\vec b_k
\end{align}
in $H^p_W$. From this and Proposition \ref{embedding},
we infer that the right-hand side of \eqref{0125} converges in $(\mathcal{S}')^m$.
Let $\vec g:=\vec f-\sum_{k\in\mathbb N}\vec b_k$. Then \eqref{except1} holds.
This finishes the proof of (ii).

In the end, we prove (iii).
For any $k\in\mathbb{N}$, let $\vec g_k:=(P_k\vec f)\eta_k$.
By \eqref{supp-eta} and (c), we conclude that, for any $x\in\mathbb R^n$,
\begin{align*}
\#\left\{k\in\mathbb N:\ \eta_k(x)\neq 0\right\}
\le\sum_{k\in\mathbb N}\mathbf{1}_{\widetilde a Q_k}(x)
\le (12^n-4^n+1)\mathbf{1}_Q(x),
\end{align*}
which further implies that
$\vec g(x):=\sum_{k\in\mathbb N} \vec g_k(x)$
and
$\vec b(x):=\sum_{k\in\mathbb N} \vec b_k(x)$
are finite summation and hence they are well defined.
Using this and \eqref{dlsl},
we find that, for any $x\in O$,
$\vec g(x)+\vec b(x)
=\vec f(x)\sum_{k\in{\mathbb N}}\eta_k(x)
=\vec f(x).$
For any $x\in\mathbb{R}^n$, let
$\vec g_0(x):=\vec f(x)-\sum_{k\in\mathbb{N}}\vec g_k(x)-
\sum_{k\in\mathbb{N}}\vec b_k(x)=\vec f(x)\mathbf{1}_{O^\complement}(x)$.
Then \eqref{except} holds.
From the definition of $\vec g_k$ and \eqref{supp-eta}, we deduce that
$\operatorname{supp}\vec g_k
\subset\operatorname{supp}\eta_k
\subset\widetilde{Q}_k
\subset Q_k^*$.
Combining the definitions of $\vec g_k$ and $\eta_k$,
\eqref{estimate of gk}, and
Lemma \ref{sharp estimate}, we obtain
\begin{align*}
\left|A_{Q_k^*}(x)\vec g_k(x)\right|
\le\left\|A_{Q_k^*}(x)A_{l_k}(x)^{-1}\right\|
\left|A_{l_k}(x)P_k\vec f(x)\right|
\lesssim\alpha.
\end{align*}
This finishes the proof of (iii) and hence Theorem \ref{l-CZ}.
\end{proof}

To show Theorem \ref{W-F-atom}(ii),
we need the following properties of $A_p$-matrix weights, which are
superior to those of $A_{p,\infty}$ weights.
The following lemma is essentially contained in \cite[Lemmas 3.2 and 3.3]{fr21}.

\begin{lemma}\label{8}
Let $p\in(0,\infty)$, $W\in A_p$,
and $\{A_Q\}_{\mathrm{cube}\,Q}$ be a family of
reducing operators of order $p$ for $W$.
\begin{enumerate}
\item[{\rm(i)}] If $p\in(0,1]$, then
$$
\sup_{\emph{cube}\,Q}\mathop{\mathrm{\,ess\,sup\,}}_{x\in Q}
\left\|A_QW^{-\frac{1}{p}}(x)\right\|
\sim[W]_{A_p}^{\frac{1}{p}},
$$
where the positive equivalence constants depend only on $m$ and $p$.
\item[{\rm(ii)}] If $p\in(1,\infty)$,
then there exist a positive constant $\varepsilon$,
depending only on $n$, $m$, $p$, and $[W]_{A_p}$,
and a positive constant $C$,
depending only on $m$ and $p$,
such that, for any $r\in[0,p'+\varepsilon]$,
\begin{equation*}
\sup_{Q\in\mathscr{Q}}
\left[\fint_Q\left\|A_QW^{-\frac{1}{p}}(x)\right\|^r\,dx\right]^{\frac{1}{r}}
\le C[W]_{A_p}^{\frac{1}{p}}.
\end{equation*}
\end{enumerate}
\end{lemma}

\begin{proposition}\label{dense}
Let $p\in(0,1]$ and $W\in A_{p}$.
Then $H^p_W\cap (L^1_{\mathrm{loc}})^m$ is dense in $H^p_W$.
\end{proposition}

\begin{proof}
Let $\alpha\in(0,\infty)$, $N\in{\mathbb Z}_+$,
$\psi\in\mathcal{S}$ satisfy
$\int_{{\mathbb{R}^n}}\psi(x)\,dx\neq0$ and
$\operatorname{supp}\psi\subset B(\mathbf{0},1)$,
and $\vec f\in H_W^p$.
Let $s\in(\frac{d_{p,\infty}^{\mathrm{lower}}(W)}{p}-1,\infty)
\cap(\lfloor n(\frac{1}{p}-1)\rfloor,\infty)\cap\mathbb N$. Then
there exists $d\in[\![d_{p,\infty}^{\mathrm{lower}}(W),\infty)$ such that
$s\in(\frac{d}{p}-1,\infty)$.
For any $k\in{\mathbb N}$, let $c_k$, $l_k$, $Q_k^*$, $\eta_k$, $P_k$,
and $\vec b_k$
be the same as in the proof of Theorem \ref{l-CZ}.
By \eqref{bkcon}, we conclude that
\begin{align}\label{bjto0}
\left\|\vec b\right\|^p_{H^p_W}\lesssim\int_{\{(M_N)_{\mathbb{A}}(\vec f)>\alpha\}}
\left[(M_N)_{\mathbb{A}}\left(\vec f\right)(x)\right]^p\,dx\to0
\end{align}
as $\alpha\to\infty$.
Therefore, to prove the present proposition, it suffices to show that
$\vec g\in (L^1_{\mathrm{loc}})^m$.

We claim that, for any $x\in \mathbb{R}^n$,
\begin{align}\label{claim}
\sup_{t\in(0,1)}\left|A_1(x)\psi_t*\vec g(x)\right|
\lesssim
(M_N)_{\mathbb A}\left(\vec f\right)(x)\mathbf{1}_{O^\complement}(x)
+\sum_{k\in\mathbb N}
\frac{l_k^{n+s+1-\frac dp}}{(l_k+|x-c_k|)^{n+s+1-\frac dp}},
\end{align}
where $A_{1}$ is the same as in \eqref{Atref} with $t:=1$
and the implicit positive constant depends on $\alpha$.

To prove \eqref{claim}, using \eqref{pro-O},
we conclude that, for any $k\in\mathbb{N}$,
\begin{align}\label{le>lk}
l_k<|O|^{\frac1n}\lesssim 1.
\end{align}
Indeed, from this, Corollary \ref{p01}, the definition of
$(M_N)_{\mathbb A}(\vec f)$,
\eqref{exin}, and \eqref{<alpha},
we infer that,
for any $x\in O^\complement$,
\begin{align}\label{PO}
\sup_{t\in(0,1)}\left|A_1(x)\psi_t*\vec g(x)\right|
&\le \sup_{t\in(0,1)}\left\|A_1(x)[A_t(x)]^{-1}\right\|\left|A_t(x)\psi_t*\vec f(x)\right|\\
&\quad+\sum_{k\in\mathbb N}\sup_{t\in(0,1)}
\left\|A_1(x)\left[A_{l_k}(c_k)\right]^{-1}\right\|
\left|A_{l_k}(c_k)\psi_t*\vec b_k(x)\right|\nonumber\\
&\lesssim (M_N)_{\mathbb A}\left(\vec f\right)(x)
+\sum_{k\in\mathbb N}
(1+|x-c_k|)^\frac{d}{p}
\frac{l_k^{n+s+1}}{(l_k+|x-c_k|)^{n+s+1}}\nonumber\\
&\lesssim (M_N)_{\mathbb A}\left(\vec f\right)(x)
+\sum_{k\in\mathbb N}
\frac{l_k^{n+s+1-\frac dp}}{(l_k+|x-c_k|)^{n+s+1-\frac dp}},\nonumber
\end{align}
which is exactly \eqref{claim} in the case of $x\in O^\complement$.

Next, let $x\in O$.
By (a) appearing in the proof of Theorem \ref{l-CZ},
we conclude that there exists
$k_x\in{\mathbb Z}$ such that $x\in Q_{k_x}$.
Let $N_{\mathrm{near}}^x:=\{k\in{\mathbb N}:\ Q_k^*\cap Q_{k_x}^*\neq\emptyset\}$
and $N_{\mathrm{far}}^x:=\{k\in{\mathbb N}:\ Q_k^*\cap Q_{k_x}^*=\emptyset\}$.
Using both (iii) and (v) of Lemma \ref{wd}, we find that
$\sup_{x\in O}\#N_{\mathrm{near}}^x<\infty$.
Consequently, we have
\begin{align*}
\vec g(x)
&=\vec f(x)-\sum_{k\in N_{\mathrm{near}}^x} \vec b_k(x)
-\sum_{k\in N_{\mathrm{far}}^x} \vec b_k(x)
=\vec f(x)\left[1-\sum_{k\in N_{\mathrm{near}}^x}\eta_k(x)\right]
+\sum_{k\in N_{\mathrm{near}}^x} \left(P_k\vec f\right)(x)\eta_k(x)
-\sum_{k\in N_{\mathrm{far}}^x} \vec b_k(x)\\
&=:\vec h(x)
+\sum_{k\in N_{\mathrm{near}}^x} \left(P_k\vec f\right)(x)\eta_k(x)
-\sum_{k\in N_{\mathrm{far}}^x} \vec b_k(x)
\end{align*}
and hence
\begin{align*}
&\sup_{t\in(0,1)}\left|A_1(x)\psi_t*\vec g(x)\right|\\
&\quad\le \sup_{t\in(0,1)}\left|A_1(x)\psi_t*\vec h(x)\right|
+\sum_{k\in N_{\mathrm{near}}^x}\sup_{t\in(0,1)}
\left|A_1(x)\psi_t*\left[\left(P_k\vec f\right)\eta_k\right](x)\right|
+\sum_{k\in N_{\mathrm{far}}^x}
\sup_{t\in(0,1)}\left|A_1(x)\psi_t*\vec b_k(x)\right|\\
&\quad=:I(x)+II(x)+III(x).
\end{align*}

We first estimate $I(x)$.
From \eqref{dlsl} and the definition of $N_{\mathrm{near}}^x$, we deduce that,
for any $z\in Q_{k_x}^*$, $1-\sum_{k\in N_{\mathrm{near}}^x}\eta_k(z)=0$.
By this and
$\operatorname{supp}{\psi_t}\subset B(\mathbf{0},t)$,
we conclude that,
for any $t\in(0,2^{-3}l_{k_x})\subset(0,\operatorname{dist}
(x,(Q_{k_x}^{*})^\complement))$,
$\psi_t*\vec h(x)=\vec0$.
Using this, we immediately find that
\begin{align}\label{est-I(x)}
I(x)&\le
\sup_{t\in(2^{-3}l_{k_x},1)}\left|A_1(x)
\psi_t*\vec f(x)\right|
+\sum_{k\in N_{\mathrm{near}}^x}\sup_{t\in(2^{-3}l_{k_x},1)}\left|A_1(x)
\psi_t*\left(\eta_k\vec f\right)(x)\right|\\
&=\sup_{t\in(2^{-3}l_{k_x},1)}\left|A_1(x)
\left\langle\vec f,\psi_t(x-\cdot)\right\rangle\right|
+\sum_{k\in N_{\mathrm{near}}^x}
\sup_{t\in(2^{-3}l_{k_x},1)}\left|A_1(x)
\left\langle\vec f,\zeta^{(k)}_{(t,x)}\right\rangle\right|\notag\\
&=:I_1(x)+I_2(x),\notag
\end{align}
where $\zeta^{(k)}_{(t,x)}(\cdot):=\psi_t(x-\cdot)\eta_k(\cdot)$.

Firstly, we estimate $I_1(x)$.
By $\psi\in \mathcal{S}$ and $\operatorname{supp}\psi\subset B(\mathbf{0},1)$,
we conclude that, for any $t\in(0,\infty)$,
$$\operatorname{supp}[\psi_t(x-\cdot)]\subset B(x,t)$$
and, for any $\beta\in{\mathbb Z}^n_+$,
$|\partial^\beta[\psi_t(x-\cdot)]|\lesssim t^{-n-|\beta|}$.
These, together with Lemma \ref{le<},
further imply that, for any $y\in{\mathbb{R}^n}$ and $t\in(0,\infty)$,
\begin{align}\label{bxgl}
\left|A_t(y)
\left\langle\vec f,\psi_t(x-\cdot)\right\rangle\right|
\lesssim \left(2+\frac{|y-x|}t\right)^{N+n+1}
(M_N)_{\mathbb{A}}\left(\vec f\right)(y).
\end{align}
Due to (b) appearing in the proof of Theorem \ref{l-CZ},
we are able to choose $y_x\in 10\sqrt{n}Q_{k_x}\cap O^\complement$.
Then, from \eqref{le>lk}, Corollary \ref{p01}, \eqref{bxgl},
and \eqref{<alpha}, we infer that,
for any $t\in(2^{-3}l_{k_x},1)$,
\begin{align}\label{V1}
\left|A_1(x)
\left\langle\vec f,\psi_t(x-\cdot)\right\rangle\right|
&\le\left\|A_1(x)A_t(y_x)^{-1}\right\|\left|A_t(y_x)
\left\langle\vec f,\psi_t(x-\cdot)\right\rangle\right|\\
&\lesssim\left|A_t(y_x)
\left\langle\vec f,\psi_t(x-\cdot)\right\rangle\right|
\lesssim(M_N)_{\mathbb{A}}\left(\vec f\right)(y_x)
\lesssim1.\nonumber
\end{align}

Secondly, we consider $I_2(x)$.
Using \eqref{le>lk}, Corollary \ref{p01}, and
\eqref{1005},
we conclude that, for any $k\in N_{\mathrm{near}}^x$,
\begin{align*}
\left|A_1(x)\left\langle\vec f,\zeta^{(k)}_{(t,x)}\right\rangle\right|
&\le\left\|A_1(x)\left[A_{l_k}(c_k)\right]^{-1}\right\|
\left|A_{l_k}(c_k)\left\langle\vec f,\zeta^{(k)}_{(t,x)}\right\rangle\right|
\lesssim(M_N)_{\mathbb{A}}\left(\vec f\right)(y_x)\lesssim1.\nonumber
\end{align*}
From this, \eqref{est-I(x)}, and \eqref{V1}, we deduce that
\begin{align}\label{esI}
I(x)
\lesssim1
\lesssim\sum_{k\in\mathbb N}
\frac{l_k^{n+s+1-\frac dp}}{(l_k+|x-c_k|)^{n+s+1-\frac dp}}.
\end{align}

Now, we estimate $II(x)$.
By the definition of $P_k$, \eqref{le>lk}, Corollary \ref{p01}, \eqref{hnt}, and \eqref{phit*e},
we obtain
\begin{align}\label{esII}
II(x)
&\le\sum_{k\in N^x_{\mathrm{near}}}\left\|A_1(x)\left[A_{l_k}(c_k)\right]^{-1}\right\|
\sum_{j=1}^{M}\left|A_{l_k}(c_k)\left
\langle\vec f,e^{(k)}_j\widetilde{\eta}_k\right\rangle\right|
\sup_{t\in(0,1)}\left|\psi_t*\left[e^{(k)}_j\eta_k\right](x)\right|
\\&\lesssim1\lesssim\sum_{k\in\mathbb N}
\frac{l_k^{n+s+1-\frac dp}}{(l_k+|x-c_k|)^{n+s+1-\frac dp}},\nonumber
\end{align}
where, for any $k\in\mathbb N$,
$\{e^{(k)}_j\}_{j=1}^M$ is the same as in the proof of Theorem \ref{l-CZ}
associated with $Q^*_k$.

Finally, we estimate $III(x)$.
From \eqref{le>lk}, Corollary \ref{p01}, and \eqref{exin},
we infer that
\begin{align*}
III(x)
&\le \sum_{k\in N_{\mathrm{far}}^x}
\sup_{t\in(0,1)}\left\|A_1(x)\left[A_{l_k}(c_k)\right]^{-1}\right\|
\left|A_{l_k}(c_k)\psi_t*\vec b_k(x)\right|\\
&\lesssim\sum_{k\in N^x_{\mathrm{far}}}
(1+|x-c_k|)^{\frac dp}
\frac{l_{k}^{n+s+1}}{|x-c_{k}|^{n+s+1}}
\lesssim\sum_{k\in \mathbb{N}}
\frac{l_{k}^{n+s+1-\frac dp}}{(l_{k}+|x-c_{k}|)^{n+s+1-\frac dp}}.
\end{align*}
Combining this, \eqref{PO}, \eqref{esI}, and \eqref{esII},
we find that the above claim \eqref{claim} holds.
From this, \eqref{<alpha}, and Theorem \ref{weight and reducing},
we deduce that
\begin{align*}
&\int_{\mathbb{R}^n}\sup_{t\in(0,1)}
\left|A_1(x)\psi_t*\vec g(x)\right|\,dx\\
&\quad\lesssim\int_{O^\complement}(M_N)_{\mathbb A}
\left(\vec f\right)(x)\,dx
+\sum_{k\in\mathbb N}\int_{\mathbb{R}^n}
\frac{l_k^{n+s+1-\frac dp}}{(l_k+|x-c_k|)^{n+s+1-\frac dp}}\,dx\\
&\quad\lesssim\alpha^{1-p}\int_{O^\complement}
\left[(M_N)_{\mathbb A}\left(\vec f\right)(x)\right]^p\,dx
+\sum_{k\in{\mathbb N}}l_k^n\lesssim
\alpha^{1-p}\left\|\vec f\right\|^p_{H^p_W}+|O|<\infty
\end{align*}
and hence
\begin{align}\label{L1}
\sup_{t\in(0,1)}\left|A_1\psi_t*\vec g\right|\in L^1,
\end{align}
which, together with the definition of $A_1$, further implies that
$\sup_{t\in(0,1)}|\psi_t*\vec g|\in L^1_{\mathrm{loc}}$.
From \cite[Lemma 7]{art} with $\Omega:=\mathbb{R}^n$ and
$B_x(\Omega):=\{\psi_t:\ t\in(0,1)\}$, we infer that
there exists $\vec g_0\in(L^1_{\rm{loc}})^m$ such that
$\vec g=\vec g_0$ in $[(C^\infty_{\mathrm c})']^m$.
By this and Lemma \ref{<sup}, we find that,
for almost every $x\in\mathbb{R}^n$,
$$\left|A_1(x)\vec g_{0}(x)\right|\leq\sup_{t\in(0,1)}
\left|A_1(x)\psi_t*\vec g_0(x)\right|=\sup_{t\in(0,1)}
\left|A_1(x)\psi_t*\vec g(x)\right|.$$
From this, Corollary \ref{p01}, and \eqref{L1}, we deduce that,
for any $\varphi\in\mathcal S$,
\begin{align*}
\int_{\mathbb{R}^n}\left|\vec g_{0}(x)\varphi(x)\right|\,dx&
\leq\int_{\mathbb{R}^n}\left\|A_1(\mathbf{0})\right\|
\left\|A_1(\mathbf{0})A_1^{-1}(x)\right\|
\sup_{t\in(0,1)}\left|A_1(x)\psi_t*\vec g(x)\right||\varphi(x)|\,dx\\
&\lesssim\int_{\mathbb{R}^n}
(1+|x|)^{\frac dp}\sup_{t\in(0,1)}
\left|A_1(x)\psi_t*\vec g(x)\right||\varphi(x)|\,dx<\infty\nonumber
\end{align*}
and hence $\vec g_0\in(\mathcal{S}')^m$.
Using this and the fact that $C^\infty_{\mathrm c}$ is dense in $\mathcal S$,
we conclude that $\vec g=\vec g_{0}$ in $(\mathcal S')^m$,
which completes the proof of Proposition \ref{dense}.
\end{proof}

Next, we show Theorem \ref{W-F-atom}.

\begin{proof}[Proof of Theorem \ref{W-F-atom}]
Let $\mathbb{A}:=\{A_Q\}_{Q\in\mathscr{Q}}$
be a family of reducing operators of order $p$ for $W$. For simplicity
of presentation of the proof, by Remark \ref{wsw}(i),
we are able to replace all $(p,q,s)_W$-atoms by $(p,q,s)_{\mathbb A}$-atoms.

To prove (i), let $\{\lambda_k\}_{k\in{\mathbb Z}}\in l^p$
and $\{\vec a_k\}_{k\in{\mathbb Z}}$ be a sequence of $(p,q,s)_{\mathbb A}$-atoms
supported, respectively, in $\{Q_k\}_{k\in{\mathbb Z}}$.
Let $\vec a$ be a $(p,q,s)_{\mathbb A}$-atom supported in the cube $Q$.
Let $\psi\in\mathcal{S}$
satisfy $\operatorname{supp}\psi\subset B(\mathbf{0},1)$ and
$\int_{\mathbb R^n}\psi(x)\,dx\neq0$.
By the H\"older inequality, Lemma \ref{sharp estimate},
\cite[Corollary 2.1.12]{g14c}, Lemma \ref{86},
the boundedness of $\mathcal{M}$ on $L^q$, and both (i) and (iv) of
Definition \ref{F-atom},
we conclude that
\begin{align}\label{inball}
\int_{2Q}\left[M^p_W\left(\vec a,\psi\right)(x)\right]^p\,dx
&=\int_{2Q}\sup_{t\in(0,\infty)}
\left|W^{\frac{1}{p}}(x)\psi_t*\vec a(x)\right|^p\,dx\\
&\le\int_{2Q}\left\|W^{\frac{1}{p}}(x)A_{2Q}^{-1}\right\|^p
\left\|A_{2Q}A_Q^{-1}\right\|^p
\sup_{t\in(0,\infty)}
\left|A_Q\psi_t*\vec a(x)\right|^p\,dx\nonumber\\
&\le\left[\int_{2Q}\left\|W^{\frac{1}{p}}(x)A_{2Q}^{-1}
\right\|^{\frac{pq}{q-p}}\,dx\right]^{\frac{q-p}{q}}
\left\{\int_{2Q}\left[\mathcal{M}\left(
\left|A_Q\vec a(x)\right|\right)\right]^q\,dx\right\}^{\frac pq}\nonumber\\
&\lesssim|Q|^{\frac{q-p}{q}}\left[\int_{\mathbb{R}^n}
\left|A_Q\vec a(x)\right|^q\,dx\right]^{\frac pq}
\le1.\nonumber
\end{align}
On the other hand, for any $t\in(0,\infty)$ and $x\in\mathbb{R}^n$,
let $q_{(t,x)}$ be the $s$th-degree Taylor polynomial of $y\mapsto\psi_t(x-y)$,
centered at $y=c_Q$; from this, we immediately infer that,
for any $x\in(2Q)^\complement$ and $y\in Q$,
\begin{align}\label{fhfx}
\left|\psi_t(x-y)-q_{(t,x)}(y)\right|\lesssim\frac{[l(Q)]^{s+1}}{t^{n+s+1}}.
\end{align}
Let $d\in[\![d_{p,\infty}^{\mathrm{upper}}(W),\infty)$
satisfy $s\in(n(\frac1p-1)+\frac{d}{p}-1,\infty)$.
By Definition \ref{F-atom}(i) and
$\operatorname{supp} \psi_t\subset B(\mathbf{0},t)$, we find that,
for any $x\in (2Q)^\complement$ and $t\in(0,\frac12|x-c_Q|]$,
$|\psi_t*\vec a(x)|=0$.
From this, both (iii) and (iv) of Definition \ref{F-atom}, \eqref{fhfx},
the H\"older inequality, Lemma \ref{sharp estimate},
Proposition \ref{reduceM}, and $s\in(n(\frac1p-1)+\frac{d}{p}-1,\infty)$,
we deduce that
\begin{align}\label{outball}
&\int_{(2Q)^\complement}\left[M^p_W(\vec a,\psi)(x)\right]^p\,dx
\\&\quad=\int_{(2Q)^\complement}\sup_{t\in(\frac{1}{2}|x-c_Q|,\infty)}
\left|W^{\frac{1}{p}}(x)\int_{Q}\vec a(y)\left[\psi_t(x-y)-q_{(t,x)}(y)\right]\,dy\right|^p\,dx\nonumber\\
&\quad\lesssim\int_{(2Q)^\complement}\left\|W^{\frac{1}{p}}(x)A_Q^{-1}\right\|^p
\sup_{t\in(\frac{1}{2}|x-c_Q|,\infty)}
\left(\frac{[l(Q)]^{s+1}}{t^{n+s+1}}\right)^p\left[\int_{Q}\left|A_Q\vec a(y)\right|
\,dy\right]^p\,dx\nonumber\\
&\quad\lesssim\int_{(2Q)^\complement}\left\|W^{\frac{1}{p}}(x)A_Q^{-1}\right\|^p
\left(\frac{[l(Q)]^{s+1}}{|x-c_Q|^{n+s+1}}\right)^p\,dx
\left\{|Q|^{\frac{1}{q'}}\left[\int_{Q}\left|A_Q\vec a(y)\right|^q
\,dy\right]^\frac1q\right\}^p\nonumber\\
&\quad\le\sum_{j=1}^{\infty}\int_{2^{j+1}Q\setminus2^jQ}
\left\|W^{\frac{1}{p}}(x)A_{2^{j+1}Q}^{-1}\right\|^p
\left\|A_{2^{j+1}Q}A_Q^{-1}\right\|^p
\left(\frac{[l(Q)]^{s+1+n-\frac{n}{p}}}{|x-c_Q|^{n+s+1}}\right)^p\,dx\nonumber\\
&\quad\lesssim\sum_{j=1}^{\infty}2^{jd}\left\{\frac{[l(Q)]^{s+1+n-\frac{n}
{p}}}{[2^jl(Q)]^{s+1+n-\frac{n}{p}}}\right\}^p
\fint_{2^{j+1}Q}
\left\|W^{\frac{1}{p}}(x)A_{2^{j+1}Q}^{-1}\right\|^p\,dx
\sim\sum_{j=1}^{\infty}2^{j(\frac {n+d}p-n-s-1)p}
\sim 1\nonumber.
\end{align}
Combining Theorem \ref{if and only if W}, $p\in(0,1]$, \eqref{inball}, and \eqref{outball},
we conclude that, for any $k_1,k_2\in{\mathbb Z}$,
\begin{align}\label{11}
\left\|\sum_{k=k_1}^{k_2}\lambda_k\vec a_k\right\|_{H^p_W}^p
&\le\left\|\sum_{k=k_1}^{k_2}\lambda_k M_W^p\left(\vec a_k,\psi\right)\right\|_{L^p}^p\\
&\le\left\|\sum_{k=k_1}^{k_2}\lambda_kM^p_W(\vec a_k,\psi)\mathbf{1}_{2Q_k}\right\|_{L^p}^p
+\left\|\sum_{k=k_1}^{k_2}\lambda_kM^p_W(\vec a_k,\psi)\mathbf{1}_{(2Q_k)^\complement}\right\|_{L^p}^p\nonumber\\
&\lesssim\sum_{k=k_1}^{k_2}|\lambda_k|^p\int_{2Q_k}\left[M^p_W(\vec a_k,\psi)(x)\right]^p\,dx
+\sum_{k=k_1}^{k_2}|\lambda_k|^p\int_{(2Q_k)^\complement}\left[M^p_W(\vec a_k,\psi)(x)\right]^p\,dx\nonumber\\
&\lesssim\sum_{k=k_1}^{k_2}|\lambda_k|^p
\le\sum_{k\in{\mathbb Z}}|\lambda_k|^p.\nonumber
\end{align}
This, together with Proposition \ref{jwdj}, further implies that
$\{\sum_{k=k_1}^{k_2}\lambda_k\vec a_k\}_{k_1,k_2\in\mathbb{\mathbb{Z}}}$
converges in $H^p_W$.
By this and Proposition \ref{embedding}, we find that
$\{\sum_{k=k_1}^{k_2}\lambda_k\vec a_k\}_{k_1,k_2\in\mathbb{\mathbb{Z}}}$ converges
in $(\mathcal{S}')^m$.
Using \eqref{11}, we conclude that
$\|\vec f\|_{H^p_W}\lesssim(\sum_{k\in\mathbb{Z}}|\lambda_k|^p)^{\frac1p}$.
This finishes the proof of (i).

Next, we show (ii).
We begin with the special case where
$\vec f\in (L^1_{\mathrm{loc}})^m\cap H_W^p$.
Let $d\in[\![d_{p,\infty}^{\mathrm{lower}}(W),n)$.
For any $j\in{\mathbb Z}$, using Theorem \ref{l-CZ} with $\alpha:=2^j$, we obtain
$\vec f=\vec g_j+\vec b_j$ and there exist
$O_j$, $\{Q_{j,k}\}_{k\in{\mathbb N}}$, $\{Q^*_{j,k}\}_{k\in{\mathbb N}}$, and $\{\vec g_{j,k}\}_{k\in{\mathbb N}}$
satisfying Theorem \ref{l-CZ},
where, for any $k\in\mathbb N$, $Q_{j,k}$ centers at $c_{j,k}:=(c_{j,k}^{(1)},\ldots,c_{j,k}^{(n)})$
with edge length $l_{j,k}$.
For any $j\in{\mathbb Z}$ and $k\in{\mathbb N}$, let $P_{j,k}$,
$\eta_{j,k}$, and $\widetilde\eta_{j,k}$ be
the same as, respectively, in \eqref{de-Pk}, \eqref{de-eta}, and \eqref{w-eta}
with $\alpha:=2^j$.
We first prove that
\begin{align}\label{shoulian}
\vec f=\sum^{\infty}_{j=-\infty}\left(\vec g_{j+1}-\vec g_j\right)
\end{align}
in $(\mathcal{S}')^m$.
By \eqref{bjto0} with $\alpha:=2^j$, we conclude that
$\lim_{j\to\infty}\|\vec b_j\|_{H^p_W}=0$.
This, together with Proposition \ref{embedding}, further implies that
\begin{align}\label{bto0}
\vec b_j\to\vec0\ \text{in }(\mathcal{S}')^m\text{ as }j\to \infty.
\end{align}
Now, we claim that
\begin{align}\label{claimjd}
\vec g_j\to\vec0\ \text{in }(\mathcal{S}')^m\text{ as }j\to -\infty.
\end{align}

Let $\psi\in\mathcal{S}$ satisfy
$\operatorname{supp}\psi \subset B(\mathbf{0},1)$ and
$\int_{{\mathbb{R}^n}}\psi(x)\,dx=1$.
From Lemma \ref{<sup}, Corollary \ref{p01}, and \eqref{<alpha},
we infer that, for any $j\in\mathbb Z$
and any $x\in (O_j)^{\complement}$,
\begin{align*}
\left|A_1(x)\vec g_{j}(x)\right|
&=\left|A_1(x)\vec f(x)\right|
\le \sup_{t\in(0,1]}\left|A_1(x)\psi_t*\vec f(x)\right|\\
&\le\sup_{t\in(0,1]}\left\|A_1(x)[A_t(x)]^{-1}\right\|
\left|A_t(x)\psi_t*\vec f(x)\right|
\lesssim (M_N)_{\mathbb{A}}\left(\vec f\right)(x)\lesssim2^j,
\end{align*}
where $A_t$ for any $t\in(0,\infty)$ is the same as in \eqref{Atref}.
By this and Corollary \ref{p01}, we conclude that,
for any $j\in\mathbb Z$ and $h\in\mathcal S$,
\begin{align}\label{jc4}
\int_{(O_j)^{\complement}}\left|\vec g_{j}(x)h(x)\right|\,dx&\lesssim
\int_{(O_j)^{\complement}}
\left\|[A_1(\mathbf{0})]^{-1}\right\|\left\|A_1(\mathbf{0})[A_1(x)]^{-1}\right\|
\left|A_1(x)\vec g_{j}(x)\right| |h(x)|\,dx\\
&\lesssim2^j\int_{(O_j)^{\complement}}\left(1+|x|\right)^{\frac{d}{p}}\left|h(x)\right|\,dx
\to 0\nonumber
\end{align}
as $j\to-\infty$.

On the other hand, for any $j\in\mathbb Z$, we consider $\{\vec g_{j,k}\}_{k\in{\mathbb N}}$.
For any $\kappa\in[1,\infty)$, let
$$
S^{\kappa}_1:=\left\{Q\in\mathscr Q:\ \mathbf{0}\in 3Q
\text{ and }\left|Q\right|\geq \kappa\right\},
$$
$$S^{\kappa}_2:=\left\{Q\in\mathscr Q:\ \mathbf{0}\in 3Q
\text{ and }\left|Q\right|<\kappa\right\},\ \
\mathrm{and}\ \
S_3:=\left\{Q\in\mathscr Q:\ \mathbf{0}\notin 3Q\right\}.$$
It is easy to show that $\mathscr Q=S^{\kappa}_1\cup S^{\kappa}_2\cup S_3$.

We first consider the case where
$k\in{\mathbb N}$ and $j\in{\mathbb Z}$ with
$Q_{j,k}\in S^{\kappa}_1$.
Let $\widetilde{a}<\frac98$ be the same as in the proof of Theorem \ref{l-CZ}
and, for any $x\in\mathbb R^n$,
$$\vec g_{j,k}(x):=P_{j,k}\vec f\eta_{j,k}(x)
:=\left(\sum_{|\beta|\le s}a^{j,k}_{1,\beta}x^\beta,
\ldots,\sum_{|\beta|\le s}a^{j,k}_{m,\beta}x^\beta\right)^T\eta_{j,k}(x).$$
From both (i) and (iii) of Lemma \ref{wd}, we deduce that,
for any $j\in{\mathbb Z}$, $k,i\in{\mathbb N}$ with $k\neq i$ and
$Q_{j,k}\cap Q_{j,i}\neq\emptyset$, and
$x:=(x_1,\ldots,x_n)\in \frac12Q_{j,k}$,
$$\max_{\tau\in\{1,\ldots,n\}}\left|x_\tau-c_{j,i}^{(\tau)}\right|\geq
\max_{\tau\in\{1,\ldots,n\}}\left|c_{j,k}^{(\tau)}-c_{j,i}^{(\tau)}\right|
-\max_{\tau\in\{1,\ldots,n\}}\left|c_{j,k}^{(\tau)}-x_\tau\right|
\geq\frac{l_{j,k}+l_{j,i}}{2}-\frac{l_{j,k}}{4}
\geq\frac{9}{16}l_{j,i}>\frac{\widetilde{a}}{2}l_{j,i}$$
and hence $x\notin \bigcup_{i\in{\mathbb N},\,i\neq k}\widetilde{a}Q_{j,i}$.
By the proof of the Whitney decomposition in \cite[p.\,611]{g14c},
we conclude that, for any $j\in{\mathbb Z}$ and $k,i\in{\mathbb N}$ with $k\neq i$ and
$Q_{j,k}\cap Q_{j,i}=\emptyset$, $Q_{j,k}\cap\widetilde{a}Q_{j,i}=\emptyset$.
These, together with \eqref{dlsl}, further imply that,
for any $x\in\frac12Q_{j,k}$, $\eta_{j,k}(x)=1$.
By this, we immediately obtain, for any $x\in{\mathbb{R}^n}$,
\begin{align}\label{nywmgn}
\vec g_j(x)\mathbf{1}_{\frac12Q_{j,k}}(x)=
\vec g_{j,k}(x)\mathbf{1}_{\frac12Q_{j,k}}(x)
=\left(\sum_{|\beta|\le s}a^{j,k}_{1,\beta}x^\beta,
\ldots,\sum_{|\beta|\le s}a^{j,k}_{m,\beta}x^\beta\right)^T
\mathbf{1}_{\frac12Q_{j,k}}(x).
\end{align}
From the fact that all the norms of a given
finite dimensional Banach space are equivalent,
we infer that, for any cube $Q\in\mathscr{Q}$ and any
$\{a_{l,\beta}\}_{l\in\{1,\ldots,m\},\,|\beta|\le s}$,
\begin{align}\label{equiv}
\int_{Q}\left|\left(\sum_{|\beta|\le s}a_{1,\beta}^\frac1px^\beta,
\ldots,\sum_{|\beta|\le s}a_{m,\beta}^\frac1px^\beta\right)^T\right|^p\,dx
\sim\max_{l\in\{1,\ldots,m\},\,|\beta|\le s}\left|a_{l,\beta}\right|,
\end{align}
where the positive equivalence constants depend on $Q$.
Let $Q_1$ be the cube centered at  $\mathbf{0}$ and with edge length $8$ and
$\mathcal{P}:=\{Q\in\mathscr{Q}_{2^{-1}}:\ Q\cap Q_{1}
\neq \emptyset\}.$
It is easy to prove $\#\mathcal P<\infty$ and,
for any cube $Q$ satisfying $l(Q)=1$ and $\mathbf{0}\in 6Q$,
$Q\subset Q_1$ and consequently there exists
$Q_0\in \mathcal P$ such that $Q_0\subset Q$.
By these and \eqref{equiv}, we conclude that,
for any cube $Q$ satisfying $\mathbf{0}\in 6Q$ and $|Q|\geq2\kappa$
and for any
$\{a_{i,\beta}\}_{l\in\{1,\ldots,m\},\,|\beta|\le s}\subset\mathbb C$,
\begin{align*}
&\int_{Q}\left|\left(\sum_{|\beta|\le s}a_{1,\beta}x^\beta,
\ldots,\sum_{|\beta|\le s}a_{m,\beta}x^\beta\right)^T\right|^p\,dx\\
&\quad= [l(Q)]^n\int_{\{\frac{y}{ l(Q)}:\, y\in Q\}}
\left|\left(\sum_{|\beta|\le s}a_{1,\beta} [l(Q)]^{|\beta|}x^\beta,
\ldots,\sum_{|\beta|\le s}a_{m,\beta}
\left[l(Q)\right]^{|\beta|}x^\beta\right)^T\right|^p\,dx\nonumber\\
&\quad\geq [l(Q)]^n\min_{Q^*\in\mathcal{P}}\int_{Q^*}
\left|\left(\sum_{|\beta|\le s}a_{1,\beta} [l(Q)]^{|\beta|}x^\beta,
\ldots,\sum_{|\beta|\le s}a_{m,\beta} [l(Q)]^{|\beta|}x^\beta\right)^T\right|^p\,dx\\
&\quad\sim  [l(Q)]^n\max_{l\in\{1,\ldots,m\},\,|\beta|\le s}
\left|a_{l,\beta} [l(Q)]^{|\beta|}\right|^{p}\gtrsim
[l(Q)]^n\max_{l\in\{1,\ldots,m\},\,|\beta|\le s}
\left|a_{l,\beta}\right|^{p}.\nonumber
\end{align*}
Using this, Lemma \ref{8}(i), the definition of $S^{\kappa}_1$,
Corollary \ref{p01}, and \eqref{nywmgn}, we find that, for any
$j\in{\mathbb Z}$ and $k\in{\mathbb N}$ satisfying
$Q_{j,k}\in S^{\kappa}_1$,
\begin{align*}
\int_{\frac12Q_{j}}\left|W^{\frac{1}{p}}(x)\vec g_{j,k}(x)\right|^p\,dx
&\geq\int_{ \frac12Q_{j,k}}
\frac{|\vec g_{j,k}(x)|^p}
{\|A_1(\mathbf{0})^{-1}\|^{p}\|A_1(\mathbf{0})A_{Q_{j,k}}^{-1}\|^{p}
\|A_{Q_{j,k}}W^{-\frac{1}{p}}(x)
\|^{p}}\,dx\\
&\gtrsim \left[l\left(Q_{j,k}\right)\right]^{n-d}
\max_{l\in\{1,\ldots,m\},\,|\beta|\le s}
\left|a^{j,k}_{l,\beta}\right|^p.\nonumber
\end{align*}
From this, Lemma \ref{<sup}, and Theorem \ref{if and only if W},
we deduce that,
for any
$j\in{\mathbb Z}$ and $k\in{\mathbb N}$ satisfying
$Q_{j,k}\in S^{\kappa}_1$,
\begin{align*}
\max_{l\in\{1,\ldots,m\},\,|\beta|\le s}\left|a^{j,k}_{l,\beta}\right|
&\lesssim
\left[l\left(Q_{j,k}\right)\right]^{\frac{d-n}{p}}
\left[\int_{\frac12Q_{j,k}}\left|W^{\frac{1}{p}}(x)\vec g_{j}(x)\right|^p\,dx\right]^{\frac1p}\\
&\le \left[l\left(Q_{j,k}\right)\right]^{\frac{d-n}{p}}
\left[\int_{{\mathbb{R}^n}}\sup_{t\in(0,\infty)}
\left|W^{\frac{1}{p}}(x)\psi_t*\vec g_j(x)\right|^p\,dx\right]^{\frac1p}
\le \kappa^{\frac{d-n}{p}}\left\|\vec g_j\right\|_{H^p_W}.
\end{align*}
By this, the definition of $\vec g_{j,k}$, and $Q_{j,k}\in S^{\kappa}_1$,
we conclude that,
for any $j\in{\mathbb Z}$ and $k\in{\mathbb N}$
satisfying $Q_{j,k}\in S^{\kappa}_1$,
\begin{align*}
\int_{{\mathbb{R}^n}}\left|\vec g_{j,k}(x)h(x)\right|\,dx
&\le\int_{Q_{j,k}^*}\left(\sum_{l=1}^{m}\left|\sum_{|\beta|\le s}a^{j,k}_{l,\beta}x^\beta
\right|^2\right)^{\frac{1}{2}}\left|h(x)\right|\,dx\\
&\le\max_{l\in\{1,\ldots,m\},\,|\beta|\le s}
\left|a^{j,k}_{l,\beta}\right|\int_{Q_{j,k}^*}
\left(1+|x|\right)^s\left|h(x)\right|\,dx
\lesssim \kappa^{\frac{d-n}{p}}\left\|\vec g_j\right\|_{H^p_W},\nonumber
\end{align*}
which, together with $\operatorname{supp}\vec g_{j,k}
\subset Q^*_{j,k}$ and \eqref{jrfzwsw}, further implies that, for any $j\in{\mathbb Z}$,
\begin{align}\label{jc}
\sum_{k\in{\mathbb N},\,Q_{j,k}\in S^{\kappa}_1}
\int_{{\mathbb{R}^n}}\left|\vec g_{j,k}(x)h(x)\right|\,dx
\lesssim \kappa^{\frac{d-n}{p}}\left\|\vec g_j\right\|_{H^p_W}.
\end{align}
From \eqref{bjto0}, we infer that, for any $j\in{\mathbb Z}$,
\begin{align*}
\left\|\vec b_j\right\|_{H^p_W}\lesssim\left\{\int_{\{(M_N)_{\mathbb{A}}(\vec f)>2^j\}}
\left[(M_N)_{\mathbb{A}}\left(\vec f\right)(x)\right]^p\,dx\right\}^{\frac1p}\le\left\|\vec f\right\|_{H^p_W}
\end{align*}
and hence $\|\vec g_j\|_{H^p_W}\lesssim
\|\vec b_j\|_{H^p_W}+\|\vec f\|_{H^p_W}
\lesssim\|\vec f\|_{H^p_W}$.
By this and \eqref{jc}, we conclude that,
for any $\varepsilon\in(0,\infty)$, there exists a positive constant $\kappa$ such that,
for any $j\in{\mathbb Z}$,
\begin{align}\label{jc1}
    \sum_{k\in{\mathbb N},\,Q_{j,k}\in S^{\kappa}_1}
\int_{{\mathbb{R}^n}}\left|\vec g_{j,k}(x)h(x)\right|\,dx<\varepsilon.
\end{align}

Next, we consider the case where $k\in{\mathbb N}$ and $j\in{\mathbb Z}$ such that
$Q_{j,k}\in S^{\kappa}_2$ or $Q_{j,k}\in S_3$.
From Corollary \ref{p01}, \eqref{estimate-g},
the definition of $S^{\kappa}_2$, and
\eqref{jrfzwsw}, we deduce that
\begin{align}\label{jc2}
&\sum_{k\in{\mathbb N},\, Q_{j,k}\in S^{\kappa}_2}\int_{{\mathbb{R}^n}}\left|\vec g_{j,k}(x)h(x)\right|\,dx\\
&\quad\lesssim\sum_{k\in{\mathbb N},\, Q_{j,k}\in S^{\kappa}_2}\int_{Q_{j,k}^*}\left|h(x)\right|
\left\|A_1(\mathbf{0})^{-1}\right\|\left\|A_1(\mathbf{0})A_{Q_{j,k}^*}^{-1}\right\|
\left|A_{Q_{j,k}^*}\vec g_{j,k}(x)\right|\,dx\nonumber\\
&\quad\lesssim2^j\sum_{k\in{\mathbb N},\, Q_{j,k}\in S^{\kappa}_2}\int_{Q_{j,k}^*}\left|h(x)\right|
\max\left\{1,\left[l\left(Q_{j,k}\right)\right]^{\frac{d}{p}}\right\}
\left[1+\frac{|c_{Q_{j,k}}|}{\max\{l(Q_{j,k}),1\}}\right]^{\frac{d}{p}}\,dx\nonumber\\
&\quad\lesssim2^j\sum_{k\in{\mathbb N},\,Q_{j,k}\in S^{\kappa}_2}\int_{Q_{j,k}^*}\left|h(x)\right|\,dx
\lesssim2^j\int_{{\mathbb{R}^n}}\left|h(x)\right|\,dx
\to0\nonumber
\end{align}
as $j\to-\infty$.
Using Corollary \ref{p01}, \eqref{estimate-g},
the definition of $S_3$, and
\eqref{jrfzwsw}, we find that
\begin{align}\label{jc3}
&\sum_{k\in{\mathbb N},\, Q_{j,k}\in S_3}\int_{{\mathbb{R}^n}}\left|\vec g_{j,k}(x)h(x)\right|\,dx\\
&\quad\lesssim2^j\sum_{k\in{\mathbb N},\, Q_{j,k}\in S_3}\int_{Q_{j,k}}\left|h(x)\right|
\max\left\{1,\left[l\left(Q_{j,k}\right)\right]^{\frac{d}{p}}\right\}
\left[1+\frac{|c_{Q_{j,k}}|}{\max\{l(Q_{j,k}),1\}}\right]^{\frac{d}{p}}\,dx\nonumber\\
&\quad\lesssim2^j\sum_{k\in{\mathbb N},\, Q_{j,k}\in S_3}\int_{Q_{j,k}}\left|h(x)\right|
\left(1+|x|\right)^{\frac{2d}{p}}\left[\frac{1+|c_{Q_{j,k}}|}{1+|x|}\right]
^{\frac{2d}{p}}\,dx\nonumber\\
&\quad\lesssim2^j\int_{{\mathbb{R}^n}}\left|h(x)\right|
(1+|x|)^{\frac{2d}{p}}\,dx\to0\nonumber
\end{align}
as $j\to-\infty$.
Combining \eqref{jc1}, \eqref{jc2}, \eqref{jc3}, and \eqref{jc4}, we conclude that
the above claim \eqref{claimjd} holds,
which, together with \eqref{bto0}, further implies that \eqref{shoulian} holds.

For any $j\in{\mathbb Z}$ and $k,i\in{\mathbb N}$, let $\vec c_{j,k}:=P_{j,k}\vec f$ and
\begin{align}\label{cjki}
\vec c_{j,k,i}:=P_{j+1,i}\left(\left[\vec f-\vec c_{j+1,i}\right]\eta_{j,k}\right).
\end{align}
Fix $j\in{\mathbb Z}$. We claim that
\begin{align}\label{yz3}
\vec g_{j+1}-\vec g_{j}&=\vec b_{j}-\vec b_{j+1}=\sum_{k\in{\mathbb N}}\vec b_{j,k}-\sum_{k\in{\mathbb N}}\sum_{i\in{\mathbb N}}\vec b_{j+1,i}\eta_{j,k}
+\sum_{k\in{\mathbb N}}\sum_{i\in{\mathbb N}}\vec c_{j,k,i}\eta_{j+1,i}\\
&=\sum_{k\in{\mathbb N}}\left[\left(\vec f-\vec c_{j,k}\right)\eta_{j,k}
-\sum_{i\in{\mathbb N}}\left(\vec f-\vec c_{j+1,i}\right)\eta_{j+1,i}\eta_{j,k}
+\sum_{i\in{\mathbb N}}\vec c_{j,k,i}\eta_{j+1,i}\right]
=:\sum_{k\in{\mathbb N}}\vec A_{j,k} \nonumber
\end{align}
converge in $(\mathcal{S}')^m$.
Now, we show this claim.
Using \eqref{0125}, we obtain
\begin{align}\label{claim1}
\vec b_j=\sum_{k\in{\mathbb N}}\vec b_{j,k}
\end{align}
in both $H^p_W$ and $(\mathcal{S}')^m$ and
$\vec b_{j+1}=\sum_{k\in{\mathbb N}}\vec b_{j+1,k}$ in $(\mathcal{S}')^m$.
By this and the definition of
$\{\eta_{j,k}\}_{k\in{\mathbb N}}$, we conclude that,
for any $j_1,j_2\in{\mathbb N}$ and $h\in\mathcal{S}$,
\begin{align*}
\lim_{i_1\to\infty}\lim_{i_2\to\infty}
\int_{{\mathbb{R}^n}}\left|\sum_{k=i_1}^{i_1+j_1}\sum_{i=i_2}^{i_2+j_2}\vec b_{j+1,i}(x)
\eta_{j,k}(x)h(x)\right|\,dx&\le
\lim_{i_1\to\infty}\lim_{i_2\to\infty}
\int_{{\mathbb{R}^n}}\left|\sum_{k=i_1}^{i_1+j_1}
\eta_{j,k}(x)\right|\left|\sum_{i=i_2}^{i_2+j_2}\vec b_{j+1,i}(x)h(x)\right|\,dx\nonumber\\
&\le\lim_{i_2\to\infty}
\int_{{\mathbb{R}^n}}\left|\sum_{i=i_2}^{i_2+j_2}\vec b_{j+1,i}(x)h(x)\right|\,dx=0,\nonumber
\end{align*}
which further implies that
\begin{align}\label{claim2}
\sum_{k\in{\mathbb N}}\sum_{i\in{\mathbb N}}\vec b_{j+1,i}\eta_{j,k}
\text{ converges in } (\mathcal{S}')^m.
\end{align}

Fix $i,k\in{\mathbb N}$ satisfying $\vec c_{j,k,i}\neq \vec 0$.
Let $\{e_l^{(i)}\}_{l=1}^M$ be the same as in the proof of
Theorem \ref{l-CZ} associated with $Q_{j+1,i}^*$.
From the definition of $P_{j+1,i}$, we infer that
$\vec c_{j,k,i}\neq \vec 0$ only if
$Q^{*}_{j,k}\cap Q^{*}_{j+1,i}\neq\emptyset$.
Using $O_{j+1}\subset O_j$ and Lemma \ref{wd}(ii), we find that,
if $Q_{j,k}^*\cap Q_{j+1,i}^*\neq\emptyset$,
then
\begin{align}\label{l<l}
l\left(Q_{j+1,i}^*\right)&\sim\mathop\mathrm{dist}
\left(Q_{j+1,i}^*,O_{j+1}^\complement\right)\lesssim
\mathop\mathrm{dist}\left(Q_{j,k}^*,O_{j}^\complement\right)
+l(Q_{j,k}^*)\sim l(Q_{j,k}^*),
\end{align}
which, together with \eqref{betaeta}, \eqref{betaeta2}, \eqref{w-eta},
and \eqref{qqq}, further implies that, for any $l\in\{1,\ldots,M\}$,
\begin{align}\label{bbb}
\sup_{x\in {\mathbb{R}^n}}\left|\partial^\beta \left(e_l^{(i)}\widetilde{\eta}_{j+1,i}\eta_{j,k}\right)(x)\right|
\sim\sup_{x\in {\mathbb{R}^n}}\left|\sum_{\alpha+\gamma\le\beta}\partial^{\alpha}e_l^{(i)}\partial^\gamma
\widetilde{\eta}_{j+1,i}(x)\partial^{\beta-\alpha-\gamma}\eta_{j,k}(x)\right|
\lesssim \left(l_{j+1,i}\right)^{-n-|\beta|}.
\end{align}
By this, $e_l^{(i)}\widetilde{\eta}_{j+1,i}\eta_{j,k}\in\mathcal{S}$,
$\operatorname{supp}\,(e_l^{(i)}\widetilde{\eta}_{j+1,i}\eta_{j,k})\subset
B(c_{j+1,i},\frac{\sqrt n}{2}l_{j+1,i})$, and Lemma \ref{le<},
we conclude that, for any $y\in{\mathbb{R}^n}$ and $t\in(0,\infty)$,
\begin{align}\label{gbdl}
\left|A_{l_{j+1,i}}(y)
\left\langle\vec f,e_l^{(i)}\widetilde{\eta}_{j+1,i}\eta_{j,k}\right\rangle\right|
\lesssim \left(2+\frac{|y-c_{j+1,i}|}{l_{j+1,i}}\right)^{N+n+1}
(M_N)_{\mathbb{A}}\left(\vec f\right)(y).
\end{align}

Due to (b) appearing in the proof of Theorem \ref{l-CZ}, we are able to choose
$y_{j,i}\in 10\sqrt{n}Q_{j+1,i}\cap (O_{j+1})^{\complement}$.
From this, Corollary \ref{p01}, the definition of $P_k$,
\eqref{gbdl},
\eqref{<alpha}, and \eqref{qqq},
we deduce that, for any $x\in Q^{*}_{j+1,i}$,
\begin{align}\label{sb1}
\left|A_{l_{j+1,i}}(x)P_{j+1,i}\left(\vec f\eta_{j,k}\right)(x)\right|
&\le\left\|A_{l_{j+1,i}}(x)\left[A_{l_{j+1,i}}
\left(y_{j,i}\right)\right]^{-1}\right\|
\left|A_{l_{j+1,i}}(y_{j,i})P_{j+1,i}\left(\vec f\eta_{j,k}\right)(x)\right|\\
&\lesssim\sum_{l=1}^{M}\left|A_{l_{j+1,i}}\left(y_{j,i}\right)
\left\langle\vec f,e_l^{(i)}\widetilde{\eta}_{j+1,i}\eta_{j,k}\right\rangle\right|
\left|e_l^{(i)}(x)\right|\nonumber\\
&\lesssim(M_N)_{\mathbb{A}}\left(\vec f\right)\left(y_{j,i}\right)
\sup_{l\in\{1,\ldots,M\},\,x\in Q_{j+1,i}^{*}}\left|e_l^{(i)}(x)\right|\lesssim2^j.\nonumber
\end{align}
Using the definition of $P_k$, $\operatorname{supp}\widetilde{\eta}_{j+1,i}\subset Q^*_{j+1,i}$,
Corollary \ref{p01},
\eqref{estimate of gk} with $P_k=P_{j+1,i}$ and $l_k=l_{j+1,i}$,
\eqref{bbb}, and \eqref{qqq},
we find that, for any $x\in Q^{*}_{j+1,i}$,
\begin{align}\label{sb2}
&\left|A_{l_{j+1,i}}(x)P_{j+1,i}\left(P_{j+1,i}\vec f\eta_{j,k}\right)(x)\right|\\
&\quad\le\sum_{l=1}^{M}\left|A_{l_{j+1,i}}(x)\left\langle P_{j+1,i}\vec f,
e_l^{(i)}\widetilde{\eta}_{j+1,i}\eta_{j,k}\right\rangle e_l^{(i)}(x)\right|\nonumber\\
&\quad\le\sum_{l=1}^{M}\int_{Q^*_{j+1,i}}\left\|A_{l_{j+1,i}}(x)
\left[A_{l_{j+1,i}}(y)\right]^{-1}\right\|
\left|A_{l_{j+1,i}}(y)P_{j+1,i}\vec f(y)\right|
\left|e_l^{(i)}(y)\widetilde{\eta}_{j+1,i}(y)\eta_{j,k}(y)\right|\,dy
\lesssim2^j.\nonumber
\end{align}
By \eqref{l<l}, Corollary \ref{p01}, \eqref{cjki},
\eqref{sb1}, and \eqref{sb2}, we conclude that,
for any $x\in{\mathbb{R}^n}$,
\begin{align}\label{Ajk-1}
\left|A_{l_{j,k}}(x)\vec c_{j,k,i}(x)\eta_{j+1,i}(x)\right|
&\le\left\|A_{l_{j,k}}(x)A_{l_{j+1,i}}(x)^{-1}\right\|\left|A_{l_{j+1,i}}(x)\vec c_{j,k,i}(x)\eta_{j+1,i}(x)\right|
\\&\lesssim\left|A_{l_{j+1,i}}(x)P_{j+1,i}\left(\vec f\eta_{j,k}\right)(x)\eta_{j+1,i}(x)\right|\nonumber\\
&\quad+\left|A_{l_{j+1,i}}(x)P_{j+1,i}\left(P_{j+1,i}\vec f\eta_{j,k}\right)(x)\eta_{j+1,i}(x)\right|\nonumber
\\&\lesssim2^j.\nonumber
\end{align}
Using both the definition of $l_{j,k}$ and \eqref{pro-O}, we find that
$l_{j,k}<|O|^{\frac1n}<\infty$.
From this, \eqref{Ajk-1}, and Corollary \ref{p01},
we infer that
\begin{align}\label{sbts}
\left|\vec c_{j,k,i}(x)\eta_{j+1,i}(x)\right|
&\le\left\|[A_1(\mathbf{0})]^{-1}\right\|
\left\|A_1(\mathbf{0})[A_{l_{j,k}}(x)]^{-1}\right\|
\left|A_{l_{j,k}}(x)\vec c_{j,k,i}(x)\eta_{j+1,i}(x)\right|\\
&\lesssim2^j
\left\|A_1(\mathbf{0})[A_{l_{j,k}}(x)]^{-1}\right\|
\lesssim 2^j\max\left\{1,\left(l_{j,k}\right)^{\frac dp}
\right\}\left(1+\frac{|x|}{\max\{1,l_{j,k}\}}\right)^{\frac{d}{p}}\lesssim(1+|x|)^{\frac{d}{p}},\nonumber
\end{align}
where the implicit positive constants depend on $j$.
Using \eqref{l<l} and \eqref{jrfzwsw},
we find that,
for any $i\in{\mathbb N}$, there exist
a finite number of $k\in{\mathbb N}$ such that
$Q^{*}_{j,k}\cap Q^{*}_{j+1,i}\neq\emptyset$ and hence
$\vec c_{j,k,i}\neq \vec 0$.
Moreover, from \eqref{jrfzwsw},
we deduce that, for any $x\in{\mathbb{R}^n}$, there exist
a finite number of $i\in{\mathbb N}$ such that $\eta_{j+1,i}(x)\neq0$.
These further imply that there exists $K\in{\mathbb N}$
such that, for any $x\in{\mathbb{R}^n}$,
\begin{align}\label{Kcount}
\#\left\{(k,i)\in{\mathbb N}^2:\ \vec c_{j,k,i}(x)\eta_{j+1,i}(x)\neq\vec 0\right\}
\le K.
\end{align}
Let $h\in\mathcal{S}$. Using \eqref{sbts} and \eqref{Kcount},
we conclude that there exists a positive constant $C$ such that, for any $x\in{\mathbb{R}^n}$,
\begin{align*}
\sum_{k\in{\mathbb N}}\sum_{i\in{\mathbb N}}\left|\vec c_{j,k,i}(x)
\eta_{j+1,i}(x)h(x)\right|
&\le\sum_{k\in{\mathbb N}}\sum_{i\in{\mathbb N}}\left|\vec c_{j,k,i}(x)
\eta_{j+1,i}(x)\right||h(x)|\\
&\lesssim K\left|h(x)\right|(1+|x|)^{{\frac{d}{p}}+n+1}(1+|x|)^{-n-1}
<C(1+|x|)^{-n-1}.
\end{align*}
From this and the dominated convergence theorem, it follows that
\begin{align*}
&\lim_{k_1\to\infty}\lim_{i_1\to\infty}
\int_{{\mathbb{R}^n}}\left|\sum_{k=1}^{k_1}\sum_{i=1}^{i_1}\vec c_{j,k,i}(x)
\eta_{j+1,i}(x)h(x)\right|\,dx\\
&\quad=\int_{{\mathbb{R}^n}}\left|\sum_{i\in{\mathbb N}}\sum_{k\in{\mathbb N}}\vec c_{j,k,i}(x)
\eta_{j+1,i}(x)h(x)\right|\,dx\\
&\quad=\int_{{\mathbb{R}^n}}\left|
\sum_{i\in{\mathbb N}}P_{j+1,i}\left(\left(
\vec f-\vec c_{j+1,i}\right)
\sum_{k\in{\mathbb N}}\eta_{j,k}\right)(x)\eta_{j+1,i}(x)h(x)\right|\,dx\\
&\quad=\int_{{\mathbb{R}^n}}\left|\sum_{i\in{\mathbb N}}P_{j+1,i}
\left(\vec f-P_{j+1,i}\vec f\right)
(x)\eta_{j+1,i}(x)h(x)\right|\,dx=0.
\end{align*}
Consequently, $\{\sum_{k=1}^{k_1}\sum_{i=1}^{i_1}\vec c_{j,k,i}\eta_{j+1,i}\}_{k_1,i_1\in\mathbb{N}}$
converges to $\vec 0$ in $(\mathcal{S}')^m$.
Combining this, \eqref{claim1}, and \eqref{claim2},
we find that \eqref{yz3} holds.

Fix $k\in{\mathbb N}$.
Let
\begin{align}\label{yz4}
\vec a_{j,k}:=\lambda_{j,k}^{-1}\vec A_{j,k}
\end{align}
with $\lambda_{j,k}:=c2^j|Q_{j,k}|^{\frac{1}{p}}$,
where $c$ is determined later.
Next, we aim to prove that $\vec a_{j,k}$ is a $(p,\infty,s)_{\mathbb A}$-atom.

By \eqref{l<l}, $\operatorname{supp}\eta_{j+1,i}
\subset Q^{*}_{j+1,i}$, and $\vec c_{j,k,i}\neq\vec 0$ only if
$Q^{*}_{j,k}\cap Q^{*}_{j+1,i}\neq\emptyset$,
we conclude that there exists $c_0\in(\frac54,\infty)$
such that $\operatorname{supp}\,(\sum_{i\in{\mathbb N}}\vec c_{j,k,i}\eta_{j+1,i})\subset c_0Q_{j,k}$,
which, together with $\operatorname{supp}\vec b_{j,k}\subset\operatorname{supp}\eta_{j,k}
\subset Q_{j,k}^*$, further implies that
$\operatorname{supp} \vec a_{j,k}\subset c_0Q_{j,k}$.
This shows that $\vec a_{j,k}$ satisfies Definition \ref{F-atom}(i).

From \eqref{qg},
we infer that
\begin{align}\label{jjxx}
\int_{\mathbb{R}^n}x^\gamma \left[
\left(\vec f-\vec c_{j+1,i}\right)\eta_{j,k}-P_{j+1,i}
\left(\left(\vec f-\vec c_{j+1,i}\right)\eta_{j,k}\right)\right](x)\eta_{j+1,i}(x)\,dx=\vec0
\end{align}
if $|\gamma|\le s$.
Using \eqref{3.4x}, we obtain
$\int_{\mathbb{R}^n}x^\gamma \vec b_{j,k}(x)\,dx=\vec0$ if $|\gamma|\le s$,
which, together with \eqref{jjxx}, further implies that
$\vec a_{j,k}$ satisfies Definition \ref{F-atom}(iii).

Now, we verify that $\vec a_{j,k}$ satisfies
Definition \ref{F-atom}(iv).
By \eqref{betaeta}, Lemma \ref{<sup}, Corollary \ref{p01},
the definition of
$(M_N)_{\mathbb{A}}(\vec f)$, and
\eqref{<alpha}, we find that, for any $x\in
(O^{j+1})^\complement$,
\begin{align}\label{Ajk-3}
\left|A_{l_{j,k}}(x)\vec f(x)\eta_{j,k}(x)
\right|&\le \left|A_{l_{j,k}}(x)\vec f(x)
\right|\le \sup_{t\in(0,l_{j,k})}\left|A_{l_{j,k}}(x)\psi_t*\vec f(x)\right|\\
&\lesssim \sup_{t\in(0,l_{j,k})}\left\|A_{l_{j,k}}(x)A_{t}(x)^{-1}\right\|
\left|A_{t}(x)\psi_t*\vec f(x)\right|
\le (M_N)_{\mathbb{A}}\left(\vec f\right)(x)\lesssim2^j.\nonumber
\end{align}
From \eqref{betaeta}, $\operatorname{supp}\eta_{j,k}\subset Q^{*}_{j,k}$,
and \eqref{estimate of gk}, we deduce that, for any $x\in{\mathbb{R}^n}$,
\begin{align}\label{Ajk-2}
\left|A_{l_{j,k}}(x)\vec c_{j,k}(x)\eta_{j,k}(x)
\right|\lesssim2^j.
\end{align}
Using \eqref{betaeta}, $\operatorname{supp}\eta_{j+1,k}\subset Q^{*}_{j+1,k}$,
\eqref{l<l}, Corollary \ref{p01}, and \eqref{estimate of gk},
we conclude that, for any $x\in{\mathbb{R}^n}$,
\begin{align}\label{Ajk-4}
\left|A_{l_{j,k}}(x)\vec c_{j+1,i}(x)\eta_{j+1,i}\eta_{j,k}(x)\right|
\le\left\|A_{l_{j,k}}(x)A_{l_{j+1,i}}(x)^{-1}
\right\|\left|A_{l_{j+1,i}}(x)\vec c_{j+1,i}(x)
\mathbf{1}_{Q^{*}_{j+1,k}}(x)\right|
\lesssim2^j.
\end{align}
Combining \eqref{Ajk-1}, \eqref{Ajk-3}, \eqref{Ajk-2},
\eqref{Ajk-4}, $\operatorname{supp}\eta_{j+1,i}\subset Q_{j+1,i}^*$,
and \eqref{jrfzwsw}, we find that,
for any $x\in{\mathbb{R}^n}$,
\begin{align*}
\left|A_{l_{j,k}}(x)\vec A_{j,k}(x)\right|
&\le\left|A_{l_{j,k}}(x)\vec f(x)\eta_{j,k}(x)
\mathbf{1}_{(\bigcup_{i\in{\mathbb N}}Q_{j+1,i})^{\complement}}(x)\right|
+\left|A_{l_{j,k}}(x)\vec c_{j,k}(x)\eta_{j,k}(x)\right|\\
&\quad+\sum_{i\in{\mathbb N}}\left|A_{l_{j,k}}(x)
\vec c_{j+1,i}(x)\eta_{j+1,i}(x)\eta_{j,k}(x)\right|
+\sum_{i\in{\mathbb N}}\left|A_{l_{j,k}}(x)\vec c_{j,k,i}(x)\eta_{j+1,i}(x)\right|\nonumber\\
&\lesssim2^j,\nonumber
\end{align*}
where the implicit positive constant is independent of
$j$ and $k$ and we take it as $c$.
This further implies that $\vec a_{j,k}$ satisfies Definition
\ref{F-atom}(iv) and hence is a $(p,\infty,s)_{\mathbb A}$-atom.
By \eqref{shoulian}, \eqref{yz3}, and \eqref{yz4}, we conclude that
$\vec f=\sum_{j\in{\mathbb Z},\,k\in{\mathbb N}}\lambda_{j,k}\vec a_{j,k}$
in $(\mathcal{S}')^m$.
From the definition of $\{\lambda_{j,k}\}_{j\in{\mathbb Z},k\in{\mathbb N}}$,
Lemma \ref{wd}(i),
\eqref{>alpha}, and Theorems \ref{if and only if W}
and \ref{weight and reducing}, we infer that
\begin{align}\label{jkf}
\sum_{j\in{\mathbb Z},\,k\in{\mathbb N}}|\lambda_{j,k}|^p
&\sim\sum_{j\in{\mathbb Z},\,k\in{\mathbb N}}2^{jp}|Q_{j,k}|\le\sum_{j\in{\mathbb Z}}2^{jp}
\left|\left\{x\in{\mathbb{R}^n}:\ (M_N)_{\mathbb{A}}\left(\vec f\right)(x)>2^j\right\}\right|\\
&\sim\int_{{\mathbb{R}^n}}\left[(M_N)_{\mathbb{A}}\left(\vec f\right)(x)\right]^p\,dx
\sim\left\|\vec f\right\|^p_{H^p_W}.\nonumber
\end{align}

By Proposition \ref{dense}, we conclude that there exists a sequence
$\{\vec f_i\}_{i\in{\mathbb Z}_+}\subset (L^1_{\mathrm{loc}})^m$
satisfying $\vec f_0=\vec 0$ and
$\vec f_i\to \vec f$ in $H_W^p$ as $i\to\infty$
and hence, without loss of generality, we may assume that, for any $i\in{\mathbb Z}_+$,
$\|\vec f_{i+1}-\vec f_i\|_{H^p_W}^p\le 2^{-i-1}\|\vec f\|_{H^p_W}^p$.
Then $\vec f=\sum_{i=0}^\infty(\vec f_{i+1}-\vec f_i)$ in $H_W^p$
and consequently in $(\mathcal{S}')^m$.
Applying the atomic decomposition to each $\vec f_{i+1}-\vec f_i$,
we obtain
$\vec f=\sum_{i\in{\mathbb Z}_+,\,j\in{\mathbb Z},\,k\in{\mathbb N}}\lambda_{j,k,i}\vec a_{j,k,i}
$ in $(\mathcal{S}')^m$ and
$$\sum_{i\in{\mathbb Z}_+,\,j\in{\mathbb Z},\,k\in{\mathbb N}}\left|\lambda_{j,k,i}\right|^p
\lesssim\sum_{i\in{\mathbb Z}_+}
\left\|\vec f_{i+1}-\vec f_i\right\|_{H^p_W}^p
\lesssim\left\|\vec f\right\|^p_{H^p_W}.
$$
From the fact that
$\{\lambda_{j,k,i}\}_{i\in{\mathbb Z}_+,\,j\in{\mathbb Z},\,k\in{\mathbb N}}\in l^p$
and $\{\vec a_{j,k,i}\}_{i\in{\mathbb Z}_+,\,j\in{\mathbb Z},\,k\in{\mathbb N}}$
are $(p,\infty,s)_{\mathbb A}$-atoms,
statement (i) of this theorem, and the uniqueness of the limit,
we deduce that $\vec f=\sum_{i\in{\mathbb Z}_+,\,j\in{\mathbb Z},\,k\in{\mathbb N}}\lambda_{j,k,i}\vec a^i_{j,k}
$ in $H^p_W$.
This finishes the proof of (ii) and hence Theorem \ref{W-F-atom}.
\end{proof}

\begin{remark}\label{wap}
In the proofs of Theorem \ref{W-F-atom} and Lemma \ref{dense},
when we use Corollary \ref{p01}, we need $W\in A_p$
which is a proper subset of $A_{p,\infty}$.
Specially, in Theorem \ref{W-F-atom}(i),
we assume $W\in A_p$. This is because
that, in the proof of \eqref{PO}, we used the estimate that,
for any $t\in(0,1)$ and $x\in \mathbb{R}^n$,
$\|A_1(x)[A_t(x)]^{-1}\|\lesssim 1$ and the implicit positive
constant is independent of
$t$ and $x$, which when $W\in A_{p,\infty}$ becomes
$\|A_1(x)[A_t(x)]^{-1}\|\lesssim\frac{1}{t^d}$ with
some $d\in[0,\infty)$ and we then cannot effectively eliminate
this upper bound $\frac 1 {t^d}$.
Of course, there exist some other similar difficulties in the proof of
Theorem \ref{W-F-atom}(i) if $W\in A_{p,\infty}$.
A challenging question is whether or not the assumption $W\in A_p$
in Theorem \ref{W-F-atom}(i) can be weakened to $W\in A_{p,\infty}$,
which seems to need an essentially new approach.
\end{remark}

Now, we prove that, if $p\in(1,\infty)$, the matrix-weighted Lebesgue space
$L^p_W$ and the matrix-weighted Hardy space $H^p_W$ coincide with equivalent norms.

\begin{theorem}\label{L-H}
Let $p\in(1,\infty)$ and $W\in A_p$ be a matrix weight.
Then $H^p_W=L^p_W$ with equivalent norms.
\end{theorem}

\begin{proof}
Let $\psi\in\mathcal{S}$ satisfy
$\operatorname{supp}\psi\subset B(\mathbf{0},1)$
and $\int_{\mathbb R^n}\psi(x)\,dx\neq 0$.

We first show $L^p_W\subset H^p_W$.
Using the H\"older inequality, we find that,
for any $\vec f\in L^p_W$ and $\varphi\in\mathcal{S}$,
\begin{align*}
\int_{{\mathbb{R}^n}}\left|\vec f(x)\varphi(x)\right|\,dx
&\lesssim\left\|\vec f\right\|_{L^p_W}
\left[\int_{\mathbb{R}^n}\left\|W^{-\frac1p}(x)\varphi(x)\right\|^{p'}\,dx\right]^{\frac{1}{p'}},
\end{align*}
and, moreover, from \cite[Corollary 2.15]{bhyy1} and Lemma \ref{sharp estimate},
we infer that, for any $\varphi\in\mathcal S$,
\begin{align}\label{pp'}
&\int_{\mathbb{R}^n}\left\|W^{-\frac1p}(x)\varphi(x)\right\|^{p'}\,dx\\
&\quad\lesssim
\sum_{Q\in\mathscr{Q}_1}\int_Q\left\|W^{-\frac1p}(x)\right\|^{p'}
\left(1+|x|\right)^{-\frac{dp'}{p}-n-1}\,dx
\sim\sum_{Q\in\mathscr{Q}_1}
\left\|A_Q^{-1}\right\|^{p'}
\left|c_Q\right|^{-\frac{dp'}{p}-n-1}\nonumber\\
&\quad\lesssim\sum_{Q\in\mathscr{Q}_1}
\left\|\left[A_1(\mathbf{0})\right]^{-1}\right\|^{p'}
\left\|A_1(\mathbf{0})A_Q^{-1}\right\|^{p'}
\left|c_Q\right|^{-\frac{dp'}{p}-n-1}
\lesssim\sum_{Q\in\mathscr{Q}_1}\left|c_Q\right|^{-n-1}<\infty,\nonumber
\end{align}
where $\mathscr Q_1$ is as in \eqref{deQt} with $t=1$,
$\{A_Q\}_{Q\in\mathscr Q_1}$ is a family of reducing operators of order $p$ for $W$,
and $d$ is the same as in Lemma \ref{p01};
consequently, $\vec f\in (\mathcal S')^m$.
Recall that, for any $\vec f\in(\mathscr{M})^m$,
the \emph{matrix-weighted maximal
function} $\mathcal{M}_{W,p}\vec f$ is defined by setting,
for any $x\in\mathbb{R}^n$,
$\mathcal{M}_{W,p}\vec f(x):=
\sup_{t\in(0,\infty)}
\fint_{B(x,t)}|W^{\frac1p}(x)W^{-\frac1p}(y)
\vec f(y)|\,dy.$
From Theorem \ref{if and only if W}, the definition of $\psi$,
and \cite[Theorem 1.3]{IM19}, we deduce that,
for any $\vec f\in L^p_W\subset (\mathcal S')^m$,
\begin{align*}
\left\|\vec f\right\|_{H^p_W}
&\sim\left[\int_{{\mathbb{R}^n}}\sup_{t\in(0,\infty)}
\left|W^{\frac1p}(x)\psi_t*\vec f(x)\right|^p\,dx\right]^{\frac{1}{p}}\\
&\lesssim\left\{\int_{{\mathbb{R}^n}}\left[\mathcal{M}_{W,p}\left(W^{-\frac1p}\vec f\right)(x)
\right]^p\,dx\right\}^{\frac{1}{p}}\lesssim\left[\int_{{\mathbb{R}^n}}\left|W^{\frac1p}(x)\vec f(x)\right|^p\,dx\right]^{\frac{1}{p}}
=\left\|\vec f\right\|_{L^p_W}\nonumber
\end{align*}
and hence $L^p_W\subset H^p_W$.

Next, we prove $H^p_W\subset L^p_W$.
By the H\"older inequality, \cite[Corollary 2.15]{bhyy1}, and Theorem \ref{if and only if W},
we find that, for any cube $Q$ and any $\vec f\in H^p_W$,
\begin{align*}
&\int_Q\sup_{t\in(0,1)}\left|\psi_t*\vec f(x)\right|\,dx
\\&\quad\le\left[\int_Q\left\|W^{-\frac1p}(x)\right\|^{p'}\,dx\right]^{\frac{1}{p'}}
\left[\int_Q\sup_{t\in(0,\infty)}
\left|W^{\frac1p}(x)\psi_t*\vec f(x)\right|^p\,dx\right]^{\frac1p}
\lesssim\left\|A_Q^{-1}\right\|
\left\|\vec f\right\|_{H^p_W}<\infty,
\end{align*}
where $A_Q$ is a reducing operator of order $p$ for $W$,
and consequently $\sup_{t\in(0,1)}|\psi_t*\vec f|\in L^1_{\mathrm{loc}}$.
From \cite[Lemma 7]{art} with $\Omega:=\mathbb{R}^n$ and
$B_x(\Omega):=\{\psi_t:\ t\in(0,1)\}$, we infer that
there exists $\vec f_{0}\in (L^1_{\mathrm{loc}})^m$ such that
$\vec f=\vec f_{0}$ in $[(C^\infty_{\mathrm c})']^m$.
Using \eqref{pp'}, Lemma \ref{<sup}, and Theorem \ref{weight and reducing},
we conclude that, for any $\varphi\in\mathcal S$,
\begin{align*}
\int_{\mathbb{R}^n}\left|\vec f_{0}(x)\varphi(x)\right|\,dx
&\leq\left[\int_{\mathbb{R}^n}\left|W^{\frac1p}(x)
\vec f_{0}(x)\right|^p\,dx\right]^{\frac1p}
\left[\int_{\mathbb{R}^n}\left\|W^{-\frac1p}(x)\varphi(x)
\right\|^{p'}\right]^{\frac{1}{p'}}\\
&\lesssim\left[\int_{\mathbb{R}^n}\sup_{t\in(0,\infty)}\left|W^{\frac1p}(x)
\psi_t*\vec f(x)\right|^p\,dx\right]^{\frac1p}
\sim\left\|\vec f\right\|_{H^p_W}<\infty
\end{align*}
and hence $\vec f_0\in(\mathcal{S}')^m$.
This, together with the well-known fact that $C^\infty_{\mathrm c}$ is dense in $\mathcal{S}$,
further implies that $\vec f=\vec f_{0}$ in $(\mathcal S')^m$.
From this, Lemma \ref{<sup}, and Theorem \ref{if and only if W}, we deduce that,
for any $\vec f\in H^p_W$,
\begin{align*}
\left\|\vec f_0\right\|_{L^p_W}=
\left[\int_{{\mathbb{R}^n}}\left|W^{\frac1p}(x)\vec f_0(x)\right|^p\,dx\right]^{\frac{1}{p}}
\le\left[\int_{{\mathbb{R}^n}}\sup_{t\in(0,\infty)}
\left|W^{\frac1p}(x)\psi_t*\vec f_0(x)\right|^p\,dx\right]^{\frac{1}{p}}
\sim\left\|\vec f_0\right\|_{H^p_W}=\left\|\vec f\right\|_{H^p_W},
\end{align*}
and consequently $H^p_W\subset L^p_W$,
which completes the proof of Theorem \ref{L-H}.
\end{proof}

\section{Finite Atomic Characterization and
Its Applications\label{fatom}}

In this section, we establish the finite atomic
characterization of $H^p_W$ and
the boundedness of a sublinear operator from $H^p_W$
to any $\gamma$-quasi-Banach space.
We first give the definition of the finite atomic
matrix-weighted Hardy space.

\begin{definition}\label{de-fin}
Let $p\in(0,1]$, $q\in[1,\infty]$, $s\in\mathbb Z_+$, and $W$ be a matrix weight.
The \emph{finite atomic matrix-weighted Hardy space}
$H^{p,q,s}_{W,\rm{fin}}$ is defined to be
the set of all finite linear combinations of $(p,q,s)_W$-atoms
equipped with the quasi-norm
\begin{align*}
\left\|\vec f\right\|_{H^{p,q,s}_{W,\rm{fin}}}:=\inf\left\{
\left(\sum_{k=1}^{N}|\lambda_k|^p\right)^{\frac1p}:\
N\in\mathbb N,\,\vec f=\sum_{k=1}^{N}\lambda_k\vec a_k,\,
\left\{\lambda_k\right\}_{k=1}^N\subset[0,\infty)
\right\},
\end{align*}
where the infimum is taken over all finite linear combinations of
$\vec f$ in terms of $(p,q,s)_W$-atoms.
\end{definition}

Now, borrowing some ideas from the proofs of
\cite[Lemma 7.10]{cw14} and \cite[Theorem 3.1(ii)]{msv08}
(see also \cite[Theorem 6.2(ii)]{blyz}),
we establish the following
finite atomic characterization of $H^p_W$.
Here, and thereafter,
let $\mathcal C$ denote the set of all continuous functions on $\mathbb R^n$.

\begin{theorem}\label{finatom}
Let $p\in(0,1]$, $W\in A_p$,
and $s\in(\lfloor n(\frac{1}{p}-1)\rfloor,\infty)\cap\mathbb Z_+$.
Then $\|\cdot\|_{H^{p,\infty,s}_{W,\rm{fin}}}$ and
$\|\cdot\|_{H^p_W}$ are equivalent quasi-norms
on the space $H^{p,\infty,s}_{W,\rm{fin}}\cap\mathcal C$.
\end{theorem}

\begin{proof}
From Theorem \ref{W-F-atom}(i) and
the definitions of both $H^{p,\infty,s}_{W,\rm{fin}}$ and
$H^{p}_W$, we infer that
$(H^{p,\infty,s}_{W,\rm{fin}}\cap\mathcal C)\subset H^p_W$
and, for any $\vec f\in H^{p,\infty,s}_{W,\rm{fin}}\cap\mathcal C$,
$\|\vec f\|_{H^p_W}\lesssim\|\vec f\|_{H^{p,\infty,s}_{W,\rm{fin}}}$.
Therefore, next we only need to show that,
for any $\vec f\in H^{p,\infty,s}_{W,\rm{fin}}\cap\mathcal C$,
the reverse of this inequality also holds.
By the homogeneity of both $\|\cdot\|_{H^{p,\infty,s}_{W,\rm{fin}}}$ and
$\|\cdot\|_{H^p_W}$, without loss of generality,
we may assume that $\vec f\in H^{p,\infty,s}_{W,\rm{fin}}\cap\mathcal C$
and $\|\vec f\|_{H^p_W}=1$.
Since $\vec f$ is a finite linear
combination of $(p,\infty,s)_W$-atoms, it follows that there exists
$R\in(1,\infty)$ such that
$\operatorname{supp}\vec f\subset B(\mathbf{0}, R)$.
Let $Q_{(R)}$ be the cube centered at $\mathbf{0}$ with edge length $8R$.
We first estimate $(M_N)_{\mathbb A}(\vec f)(x)$
with $x\in [B(\mathbf 0,4R)]^\complement$.
Let $d\in[\![d_{p,\infty}^{\mathrm{lower}}(W),n)$,
$N_1\in[\lfloor\frac np+\frac dp\rfloor+1,\infty)\cap\mathbb N$,
$N\in(\frac{d}{p}+N_1,\infty)\cap\mathbb N$,
$\varphi\in\mathcal S_{N}$, and
$x\in [B(\mathbf 0,4R)]^\complement$.
Assume that $\theta\in C^\infty$ satisfies $\theta=1$ on $B(\mathbf0,1)$ and
$\operatorname{supp}\theta\subset B(\mathbf{0},2)$.
Now, we estimate $\varphi_t*\vec f$ with $t\in(0,\infty)$.
If $t\geq R$, we then obtain
\begin{align}\label{nth1}
\varphi_t*\vec f(x)=\int_{\mathbb R^n}\vec f(y)\frac{1}{t^{n}}
\varphi\left(\frac{x-y}{t}\right)\theta\left(-\frac yR\right)
\,dy=\psi_R*\vec f(\mathbf{0}),
\end{align}
where, for any $z\in\mathbb R^n$,
$\psi(z):=(\frac Rt)^n\varphi(\frac{x+Rz}{t})\theta(z)$.
Using $x\in [B(\mathbf 0,4R)]^\complement$ and $N\in[\frac dp-n-1,\infty)\cap\mathbb N$, we conclude that
\begin{align*}
\|\psi\|_{\mathcal S_{N_1}}&
\lesssim\sup_{\{\alpha\in\mathbb Z^n_+:\,|\alpha|\leq N_1+1\}}
\sup_{z\in B(\mathbf 0,2)}
\left(\frac Rt\right)^n\sum_{\beta\leq\alpha}
\left|\partial^\beta\varphi\left(\frac{x+Rz}{t}\right)
\left(\frac{R}{t}\right)^{|\beta|}\partial^{\alpha-\beta}
\theta(z)\right|\\
&\lesssim\|\varphi\|_{\mathcal S_{N}}\sup_{z\in B(\mathbf 0,2)}
\left(1+\frac{|x+Rz|}{t}\right)^{-N-n-1}
\lesssim\left(1+\frac{|x|}{t}\right)^{-N-n-1}
\lesssim\left(1+\frac{|x|}{t}\right)^{-\frac dp}.
\end{align*}
This, together with \eqref{nth1}, Corollary \ref{p01},
the definition of
$(M_{1,N_1})_{\mathbb A}(\vec f)$, and
Theorem \ref{weight and reducing}, further implies that
\begin{align}\label{wts1}
\left|A_t(x)\varphi_t*\vec f(x)\right|
&\leq\inf_{y\in B(\mathbf{0},R)}
\left\|A_t(x)[A_R(y)]^{-1}\right\|
\left|A_R(y)\psi_R*\vec f(\mathbf{0})\right|\lesssim
\inf_{y\in B(\mathbf{0},R)}\left(M_{1,N_1}
\right)_{\mathbb A}\left(\vec f\right)(y)\\
&\lesssim R^{-\frac np}\left\|\left(M_{1,N_1}\right)_{\mathbb A}\left(\vec f\right)
\right\|_{L^p}\sim R^{-\frac np}
\left\|\vec f\right\|_{H^p_W}=R^{-\frac np}.\nonumber
\end{align}
Next, we consider the case $t<R$.
Let $u\in B(\mathbf 0,\frac R2)$ and
$x\in [B(\mathbf 0,4R)]^\complement$.
Then we have
\begin{align}\label{nth2}
\varphi_t*\vec f(x)=\int_{\mathbb R^n}\vec f(y)
\varphi_t(x-y)\theta\left(\frac yR\right)\,dy
=\phi_t*\vec f(u),
\end{align}
where $\phi(z):=\varphi(\frac{x-u}{t}+z)\theta(\frac uR-\frac tRz)$.
From $N\in(\frac dp+N_1,\infty)\cap \mathbb N$,
$u\in B(\mathbf 0,\frac R2)$,
and $x\in [B(\mathbf 0,4R)]^\complement$,
we deduce that
\begin{align*}
\|\phi\|_{\mathcal S_{N_1}}&
\lesssim\sup_{\alpha\in\mathbb Z^n_+,\,|\alpha|\leq N_1+1}
\sup_{z\in \mathbb R^n}(1+|z|)^{N_1+n+1}
\sum_{\beta\leq\alpha}\left|\partial^\beta\varphi\left(
\frac{x-u}{t}+z\right)\right|
\left(\frac tR\right)^{|\alpha|-|\beta|}\left|
\partial^{\alpha-\beta}\theta\left(\frac uR-\frac tRz\right)
\right|\\
&\lesssim\|\varphi\|_{S_{N}}
\sup_{z\in B(\mathbf 0,\frac{5R}{2t})}(1+|z|)^{N_1+n+1}\left(1+
\left|\frac{x-u}{t}+z\right|\right)^{-N-n-1}\lesssim
\left(1+\frac{|x|}{t}\right)^{N_1-N}
\lesssim\left(1+\frac{|x|}{t}\right)^{-\frac dp},
\end{align*}
which, together with \eqref{nth2}, Corollary \ref{p01},
$x\in [B(\mathbf 0,4R)]^\complement$,
the definition of
$(M_{1,N_1})_{\mathbb A}(\vec f)$, and
Theorem \ref{weight and reducing}, further implies that
\begin{align}\label{wts2}
\left|A_t(x)\varphi_t*\vec f(x)\right|
&\leq\inf_{y\in B(\mathbf{0},\frac R2)}\left\|A_t(x)[A_t(y)]^{-1}\right\|
\left|A_t(y)\phi_t*\vec f(y)\right|\lesssim
\inf_{y\in B(\mathbf{0},\frac R2)}\left(M_{1,N_1}
\right)_{\mathbb A}\left(\vec f\right)(y)\\
&\lesssim R^{-\frac np}\left\|\left(M_{1,N_1}\right)_{\mathbb A}\left(\vec f\right)
\right\|_{L^p}\sim R^{-\frac np}
\left\|\vec f\right\|_{H^p_W}=R^{-\frac np}.\nonumber
\end{align}
Combining \eqref{wts1}, \eqref{wts2}, and the definition of
$(M_N)_{\mathbb A}(\vec f)$, we conclude that there exists
a positive constant $\widetilde C_1$ such that,
for any $x\in [B(\mathbf 0,4R)]^\complement$,
\begin{align}\label{loul}
\left(M_N\right)_{\mathbb A}\left(\vec f\right)(x)
\le \widetilde C_1R^{-\frac np}.
\end{align}
Let all the symbols be the same as in
the proof of Theorem \ref{W-F-atom}(ii).
Let $j_0:=\lceil\log_2(\widetilde C_1R^{-\frac np})\rceil$. Then,
from \eqref{loul} and the definition of $O_j$, we infer that, for any $j\geq j_0$,
$O_j\subset B(\mathbf{0},4R)$.
By the proof of Theorem \ref{W-F-atom} and $\vec f\in(L^1_{\rm{loc}})^m\cap H^p_W$,
we conclude that
$$\vec f=\vec g_{j_0}+
\sum_{j=j_0}^{j_1-1}\sum_{k\in\mathbb N}\lambda_{j,k}\vec a_{j,k}
+\vec b_{j_1}=:\vec g_{j_0}+\vec l_{j_1}+\vec b_{j_1}
$$
both in $(\mathcal S')^m$ and almost everywhere, where $j_1\in\mathbb Z$
is determined later.

We first prove that $\vec b_{j_1}$ is
one $(p,\infty,s)_{\mathbb A}$-atom.
From the assumption that $\vec f\in \mathcal C$ has compact support
(which is guaranteed by the assumption $\vec f\in H^{p,\infty,s}_{W,\rm{fin}}$),
we deduce that, for any $\varepsilon_1\in(0,\infty)$, there exists a positive constant $\delta_1$
such that, for any $x,y\in\mathbb R^n$ with $|x-y|<\delta_1$,
$|\vec f(x)-\vec f(y)|<\varepsilon_1$.
In addition, by the fact that, for any $j\in\mathbb Z$ and
$k\in\mathbb N$, $Q_{j,k}\subset O_j$ and by \eqref{pro-O},
we find that there exists $j_1\in\mathbb Z$ such that,
for any $k\in\mathbb N$, $l(Q_{j_1,k})<|O_{j_1}|^{\frac1n}<
\frac{1}{\sqrt{n}}\delta_1$.
For any $k\in\mathbb N$, let
$\vec f_{j_1,k}:=[\vec f-\vec f(c_{j_1,k})]
\mathbf{1}_{Q^*_{j_1,k}}$;  then $|\vec f_{j_1,k}|<\varepsilon_1$.
Using the definition of $(M_N)_{\mathbb A}(\vec f)$, we obtain
$(M_N)_{\mathbb A}(\vec f_{j_1,k})<\varepsilon_1$ with $W\equiv I_m$.
From this, \eqref{hnt}, and \eqref{estimate of gk}, we infer that,
for any $x\in Q^*_{j_1,k}$,
$$\left|P_{j_1,k}\left(\vec f_{j_1,k}\right)(x)\right|
\lesssim\sup_{y\in\mathbb R^n}
\left(M_N\right)_{\mathbb A}\left(\vec f_{j_1,k}\right)(y)<\varepsilon_1,$$
which, together with $\operatorname{supp}\eta_{j_1,k}\subset Q^*_{j_1,k}$,
further implies that
$|[\vec f-P_{j_1,k}(\vec f)]\eta_{j_1,k}|\leq
|\vec f_{j_1,k}-P_{j_1,k}(\vec f_{j_1,k})|\lesssim\varepsilon_1$.
By this, the definition of $\vec b_{j_1}$, and \eqref{jrfzwsw},
we conclude that there exists
a positive constant $\widetilde C_2$ such that
\begin{align}\label{bad}
\left|\vec b_{j_1}\right|\leq\sum_{k\in\mathbb N}
\left|\left[\vec f-P_{j_1,k}\left(\vec f\right)\right]
\eta_{j_1,k}\right|
\le \widetilde C_2\varepsilon_1.
\end{align}
Let $\varepsilon_1:=\widetilde C_2^{-1}|Q_{(R)}|^{-\frac1p}$.
From $\operatorname{supp}\vec b_{j_1}\subset O_{j_1}
\subset O_{j_0}\subset B(\mathbf{0},4R)$,
we deduce that, to show that $\vec b_{j_1}$ is one $(p,\infty,s)_{\mathbb A}$-atom,
it is sufficient to prove that $\vec b_{j_1}$
has the vanishing moments up to $s$.
Using \eqref{bad}, we find that,
for any $\gamma\in\mathbb Z^n_+$,
\begin{align*}
\int_{\mathbb R^n}\left|x^\gamma\vec b_{j_1}(x)\right|\,dx&=
\int_{B(\mathbf 0,4R)}\left|x^\gamma
\sum_{k\in\mathbb N}\left[\vec f(x)-P_{j_1,k}\left(\vec f\right)(x)\right]\eta_{j_1,k}(x)\right|\,dx\\
&\lesssim\int_{B(\mathbf 0,4R)}\sum_{k\in\mathbb N}\left|
\left[\vec f(x)-P_{j_1,k}\left(\vec f\right)(x)\right]\eta_{j_1,k}(x)\right|\,dx
\lesssim\varepsilon R^n<\infty.
\end{align*}
Thus, $x^\gamma
\sum_{k\in\mathbb N}[\vec f(x)-P_{j_1,k}(\vec f)(x)]\eta_{j_1,k}(x)$ converges
absolutely in $L^1$, allowing us to interchange
the summation and the integral in the first inequality of the above formulae, which, together with \eqref{3.4x}, further implies that
$\vec b_{j_1}$ has the vanishing moments up to $s$.
This shows that $\vec b_{j_1}$ is one $(p,\infty,s)_{\mathbb A}$-atom.

Now, we focus on $\vec l_{j_1}$.
From the assumption that $\vec f\in \mathcal C$ has compact support
(which is guaranteed by the assumption $\vec f\in H^{p,\infty,s}_{W,\rm{fin}}$),
we infer that, for any $\varepsilon_2\in(0,\infty)$, there exists a positive constant $\delta_2$
such that, for any $x,y\in\mathbb R^n$ with $|x-y|<\delta_2$,
$|\vec f(x)-\vec f(y)|<\varepsilon_2$.
Let $$F_{1}:=\{(j,k)\in\mathbb Z\times\mathbb N:\ l(Q_{j,k})<\delta_2,\
j_0+1\leq j\leq j_1-1\}$$
and $$F_{2}:=\{(j,k)\in\mathbb Z\times\mathbb N:\ l(Q_{j,k})\geq\delta_2,\
j_0+1\leq j\leq j_1-1\}.$$
From \eqref{jrfzwsw} and the fact that, for any $j_0+1\leq j\leq j_1-1$,
$|O_{j}|\lesssim R^n$,
we deduce that $\#F_2<\infty$.
Applying an argument similar to that used in the estimation of $\vec b_{j_1}$,
we conclude that, for any $(j,k)\in F_1$,
$|\vec A_{j,k}|<|\lambda_{j,k}\vec a_{j,k}|<\varepsilon_2$
(see also the proof of \cite[Theorem 3.1(ii)]{msv08}).
Then, from \eqref{jrfzwsw}, we infer that
there exists a positive constant $\widetilde C_3$ such that
\begin{align}\label{breaking}
\left\|\,\left|A_{Q_{(R)}}\sum_{(j,k)\in F_1}
\vec A_{j,k}\right|\,\right\|_{L^\infty}
&\leq\left\|A_{Q_{(R)}}\right\|
\left\|\sum_{(j,k)\in F_1}
\left|\vec A_{j,k}\right|\right\|_{L^\infty}
\le \widetilde C_3\varepsilon_2.
\end{align}
Let $\varepsilon_2:=\widetilde C_3^{-1}|Q_{(R)}|^{-\frac1p}$.
By $\operatorname{supp}\vec g_{j_0}\subset[\operatorname{supp}\vec f
\cup O_{j_0}]\subset B(\mathbf{0},4R)$
and $\operatorname{supp}\vec b_{j_1}\subset B(\mathbf{0},4R)$,
we conclude that
$$\operatorname{supp}\vec l_{j_1}\subset \left[\operatorname{supp}\vec f
\cup \operatorname{supp}\vec g_{j_0}
\cup \operatorname{supp}\vec b_{j_1}\right]
\subset B(\mathbf{0},4R).$$
This, together with the fact that, for any $(j,k)\in F_2$,
$\operatorname{supp}\vec a_{j,k}\subset
Q^*_{j,k}\subset B(\mathbf 0,4R)$, further implies that,
to prove $\vec l_{j_1,\varepsilon}:=\sum_{(j,k)\in F_1}
\vec A_{j,k}$ is one $(p,\infty,s)_{\mathbb A}$-atom,
it is sufficient to show
the vanishing moments of $\vec l_{j_1,\varepsilon}$.
From \eqref{breaking} and $\operatorname{supp}\vec l_{j_1}
\subset B(\mathbf{0},4R)$, we deduce that,
for any $\gamma\in\mathbb Z^n_+$,
\begin{align*}
\int_{\mathbb R^n}\left|x^\gamma\right|
\sum_{(j,k)\in F_1}
\lambda_{j,k}\left|\vec a_{j,k}(x)\right|\,dx
\leq R^{n}\left\|\sum_{(j,k)\in F_1}
\left|\vec A_{j,k}\right|\right\|_{L^\infty}
<\infty
\end{align*}
and, moreover, applying an argument similar to that used in the proof
of the vanishing moments of $\vec b_{j_1}$,
we conclude that $\vec l_{j_1,\varepsilon}$ has the vanishing moments up to $s$.
This proves that $\vec l_{j_1,\varepsilon}$ is one $(p,\infty,s)_{\mathbb A}$-atom.

From the vanishing moments of $\vec f$,
$\vec b_{j_1}$, $\vec l_{j_1,\varepsilon}$,
and $\vec a_{j,k}$ with $(j,k)\in F_2$,
we infer that $\vec g_{j_0}$ also has the vanishing moments up to $s$.
Thus, to show $\vec g_{j_0}$ is one $(p,\infty,s)_{\mathbb A}$-atom,
it is sufficient to verify
the size condition of $\vec g_{j_0}$.
By Corollary \ref{p01}, \eqref{estimate-g}, and \eqref{jrfzwsw},
we find that, for any $x\in O_{j_0}$,
\begin{align}\label{llk1}
\left|A_{Q_{(R)}}\vec g_{j_0}(x)\right|\leq
\sum_{k\in\mathbb N}\left\|A_{Q_{(R)}}A_{Q_{j_0,k}}^{-1}\right\|
\left|A_{Q_{j,k}}\vec g_{j,k}(x)\right|
\lesssim2^{j_0}.
\end{align}
From \eqref{g0}, Lemma \ref{<sup}, Corollary \ref{p01}, and the definition of
$(M_N)_{\mathbb A}(\vec f)$,
we deduce that,
for any $x\in O_{j_0}^\complement\cap B(\mathbf{0},4R)$,
\begin{align}\label{llk2}
\left|A_{Q_{(R)}}\vec g_{j_0}(x)\right|\leq
\sup_{t\in(0,R)}\left\|A_{Q_{(R)}}A_{t}^{-1}(x)\right\|
\left|A_{t}(x)\varphi_t*\vec f(x)\right|
\lesssim\left(M_N\right)_{\mathbb A}\left(\vec f\right)(x)
\lesssim2^{j_0}.
\end{align}
By \eqref{g0}, we find that, for any
$x\in [O_{j_0}^\complement\cap B(\mathbf{0},4R)]^\complement
\subset (\operatorname{supp}\vec f)^\complement$,
$|\vec g_{j_0}(x)|=|\vec f(x)|=0$.
Using this, \eqref{llk1}, and \eqref{llk2},
we conclude that $|A_{Q_{(R)}}\vec g_{j_0}|\lesssim2^{j_0}
\sim|Q_{(R)}|^{-\frac1p}$.
This proves that $\vec g_{j_0}$ is one $(p,\infty,s)_{\mathbb A}$-atom.

Now, we obtain a finite linear
combination of $(p,\infty,s)_{\mathbb A}$-atoms of $\vec f$, that is,
$$\vec f=\vec g_{j_0}+\vec l_{j_1,\varepsilon}+
\sum_{(j,k)\in F_2}\lambda_{j,k}\vec a_{j,k}+\vec b_{j_1},$$
which, together with
Remark \ref{kfc}(i), further implies that it is also a finite linear
combination of $(p,\infty,s)_{W}$-atoms of $\vec f$.
Thus, from the definition of
$H^{p,\infty,s}_{W,\rm{fin}}$ and \eqref{jkf}, we infer that
\begin{align*}
\left\|\vec f\right\|_{H^{p,\infty,s}_{W,\rm{fin}}}
\lesssim\left(\sum_{(j,k)\in F_2}\left|\lambda_{j,k}\right|^p+3\right)^{\frac1p}
\lesssim\left(\sum_{j\in\mathbb Z}\sum_{k\in\mathbb N}
\left|\lambda_{j,k}\right|^p\right)^{\frac1p}+1
\lesssim\left\|\vec f\right\|_{H^p_W}.
\end{align*}
This finishes the proof of Theorem \ref{finatom}.
\end{proof}

As an application of Theorem \ref{finatom},
we establish the boundedness on $H^p_W$
of quasi-Banach-valued sublinear operators.
We first present the definition of
$\gamma$-quasi-Banach spaces.

\begin{definition}\label{dth}
Let $\gamma\in(0,1]$.
\begin{enumerate}
\item[{\rm(i)}]
A \emph{quasi-Banach space} $\mathcal B_\gamma$ with the
quasi-norm $\|\cdot\|_{\mathcal B_\gamma}$
is called a $\gamma$-quasi-Banach space if
$\|f+g\|_{\mathcal B_\gamma}^\gamma
\leq\|f\|_{\mathcal B_\gamma}^\gamma
+\|g\|_{\mathcal B_\gamma}^\gamma$
for any $f,g\in \mathcal B_\gamma$.
\item[{\rm(ii)}]
For any given linear space $\mathcal X$
and any given $\gamma$-quasi-Banach space $\mathcal B_\gamma$,
an operator $T$ mapping $\mathcal X$ to $\mathcal B_\gamma$
is said to be \emph{$\mathcal B_\gamma$-sublinear} if, for any $f,g\in \mathcal X$
and $\lambda,\nu\in\mathbb C$,
$$\left\|T(\lambda f+\nu g)\right\|_{\mathcal B_\gamma}
\leq\left[|\lambda|^\gamma\left\|T(f)\right\|_{\mathcal B_\gamma}^\gamma
+|\nu|^\gamma\|T(g)\|_{\mathcal B_\gamma}^\gamma\right]^\frac1\gamma
$$
and $\|T(f)-T(g)\|_{\mathcal B_\gamma}\leq\|T(f-g)\|_{\mathcal B_\gamma}$.
\end{enumerate}
\end{definition}

Next, we establish a criterion on the boundedness of sublinear operators
from $H^p_W$ to any $\gamma$-quasi-Banach space $\mathcal B_\gamma$.

\begin{theorem}\label{subl}
Let $0<p\leq\gamma\leq1$, $s\in(\lfloor n(\frac{1}{p}-1)\rfloor,
\infty)\cap\mathbb Z_+$, $W\in A_p$, and
$\mathcal B_\gamma$ be a $\gamma$-quasi-Banach space.
Then the following two assertions are equivalent.
\begin{enumerate}
\item[{\rm(i)}]
$T$ is a $\mathcal B_\gamma$-sublinear operator defined on all
the continuous $(p,\infty,s)_W$-atoms satisfying that
\begin{align}\label{supa}
\sup\left\{\left\|T(\vec a)\right\|
_{\mathcal B_{\gamma}}:\ \vec a\text{ is a continuous }
(p,\infty,s)_W\text{-atom}\right\}<\infty.
\end{align}
\item[{\rm(ii)}] $T$ has a unique bounded $\mathcal B_\gamma$-sublinear
extension from $H^p_W$ to $\mathcal B_\gamma$.
\end{enumerate}
\end{theorem}

\begin{proof} The proof that (ii) implies (i) is
a consequence of Theorem \ref{W-F-atom}(i).

To show that (i) implies (ii), let $\vec f\in H^{p,\infty,s}_{W,\rm{fin}}\cap\mathcal C$.
Using the definition of $H^{p,\infty,s}_{W,\rm{fin}}$ and
Theorem \ref{finatom}, we conclude that there
exist $N\in\mathbb N$, $\{\lambda_k\}_{k=1}^N\subset\mathbb C$, and
a set $\{\vec a_k\}_{k=1}^N$ of $(p,\infty,s)_W$-atoms
such that $\vec f=\sum_{k=1}^{N}\lambda_k\vec a_k$
and $(\sum_{k=1}^N|\lambda|^p)^{\frac1p}\lesssim\|\vec f\|_{H^p_W}$.
From this, the definition of $\mathcal B_\gamma$, and
\eqref{supa}, we deduce that
$$
\left\|T\left(\vec f\right)\right\|_{\mathcal B_\gamma}
\leq\left[\sum_{k=1}^N\left|\lambda_k\right|^\gamma
\left\|T(\vec a)\right\|_{\mathcal B_{\gamma}}
^\gamma\right]^{\frac1\gamma}\lesssim
\left(\sum_{k=1}^N|\lambda_k|^p\right)^{\frac1p}\lesssim\left\|\vec f\right\|_{H^p_W}.
$$
Thus, to extend the boundedness of $T$ to the whole $H^p_W$, by Definition \ref{dth}
we find that it is sufficient to show
$H^{p,\infty,s}_{W,\rm{fin}}\cap\mathcal C$
is dense in $H^p_W$. Using Theorem \ref{W-F-atom}, we can easily find
that $H^{p,\infty,s}_{W,\rm{fin}}$ is dense in $H^p_W$.
Thus, it remains to prove that $H^{p,\infty,s}_{W,\rm{fin}}\cap\mathcal C$
is dense in $H^{p,\infty,s}_{W,\rm{fin}}$.
Let $\vec f_0\in H^{p,\infty,s}_{W,\rm{fin}}$ and $\varphi\in\mathcal S$
satisfy $\operatorname{supp}\varphi\subset B(\mathbf 0,1)$
and $\int_{\mathbb R^n}\varphi(x)\,dx=1$.
Since $\vec f_0$ is a finite linear combination
of $(p,\infty,q)_W$-atoms, it follows that there exists $R\in(0,\infty)$
such that $\operatorname{supp}\vec f_0\subset B(\mathbf 0,R)$
and $\vec f_0$ has the vanishing moments up to $s$.
From this and the definition of $\varphi$, we infer that
$\operatorname{supp}\,(\varphi_t*\vec f)\subset B(\mathbf 0,R+t)$
and $\varphi_t*\vec f$ has the vanishing moments up to $s$.
For any $t\in(0,\infty)$,
let $Q_{R+t}$ denote the cube centered at $\mathbf 0$ with edge length $2(R+t)$.
Let $\mathbb A:=\{A_Q\}_{\mathrm{cube}\,Q}$ be a family of
reducing operators of order $p$ for $W$.
By the definition of atoms, we find that,
for any $(p,\infty,q)_{\mathbb A}$-atom
$\vec a_0$ supported in the cube $Q_0$,
we obtain $\|\vec a_0\|_{L^\infty}\leq
\|A_{Q_0}^{-1}\|\,\|A_{Q_0}\vec a_0\|_{L^\infty}<\infty$,
which, together with Remark \ref{kfc}(i), further implies that
$\|\vec f_0\|_{L^\infty}<\infty$.
Combining this and the definition of $\varphi$, we find that,
for any $t\in(0,\infty)$,
$$
\left\|A_{Q_{R+t}}\varphi_t*\vec f\right\|_{L^\infty}
\leq\left\|A_{Q_{R+t}}\right\|\left\|\varphi_t*\vec f\right\|_{L^\infty}
<\infty,
$$
which, together with Remark \ref{kfc}(i), further implies that
$\varphi_t*\vec f\in H^{p,\infty,s}_{W,\rm{fin}}\cap \mathcal C$.
From the well-known property of approximations of the identity
(see, for instance, \cite[Theorem 1.2.19(1)]{g14c}),
it follows that, for any $\varepsilon\in(0,\infty)$,
there exists $t_0\in(0,\infty)$ such that, for any $t\in(0,\min\{1,t_0\})$
and $q\in(\max\{1,\frac{r_Wp}{r_W-1}\},\infty]$,
$$
\left\|A_{Q_{R+1}}\left(\vec f_0-\varphi_t*\vec f_0\right)\right\|_{L^q}
\leq\left\|A_{Q_{R+1}}\right\|\left\|\vec f_0-\varphi_t*\vec f_0\right\|_{L^q}
\leq\varepsilon,
$$
and hence $\varepsilon^{-1}|Q_{R+1}|^{\frac1q-\frac1p}
(\vec f_0-\varphi_t*\vec f_0)$ is a $(p,q,s)_{\mathbb A}$-atom.
Using Theorem \ref{W-F-atom}(i) and Remark \ref{kfc}(i), we find that
$\|\vec f_0-\varphi_t*\vec f_0\|_{H^p_W}\lesssim\varepsilon$.
This proves that $H^{p,\infty,s}_{W,\rm{fin}}\cap\mathcal C$
is dense in $H^p_W$, which then shows that (i) implies (ii) and
hence completes the proof of Theorem \ref{subl}.
\end{proof}

\begin{remark}
Let all the assumptions be the same as in Theorem \ref{subl}
and, in addition, let $\mathbb A:=\{A_Q\}_{Q\in\mathscr{Q}}$
be a family of reducing operators of order $p$ for $W$.
\begin{enumerate}
\item[(i)] From Remark \ref{kfc}(i), it follows that Theorem \ref{subl} still holds if all
the continuous $(p,\infty,s)_W$-atoms are replaced by
all the continuous $(p,\infty,s)_{\mathbb A}$-atoms.

\item[(ii)] Assume that
$\|T(\vec a)\|_{\mathcal B_\gamma}<\infty$ for any
$(p,\infty,s)_W$-atom $\vec a$.
Denote by $T_0$ the restriction of $T$ to all the continuous
$(p,\infty,s)_W$-atoms. Then, by Theorem \ref{subl},
$T_0$ has an extension, denoted by $\widetilde T_0$,
such that $\widetilde T_0$ is bounded from $H^p_W$ to $\mathcal B_\gamma$.
However, $\widetilde T_0$ may not coincide with $T$ on all the
$(p,\infty,s)_W$-atoms (see, for instance, \cite[Theorem 2]{b08}).

\item[(iii)]
Let $q\in(\max\{1,\frac{r_Wp}{r_W-1}\},\infty]$.
If, similarly to \cite[Theorem 3.1(i)]{msv08},
we could establish a finite atomic characterization of $H^p_W$
on non-continuous atoms, that is,
$\|\cdot\|_{H^{p,q,s}_{W,\rm{fin}}}$ and
$\|\cdot\|_{H^p_W}$ are equivalent quasi-norms
on the finite atomic matrix-weighted Hardy space $H^{p,q,s}_{W,\rm{fin}}$,
then the extension of
$T$, where $T$ satisfies $\|T(\vec a)\|_{\mathcal B_\gamma}\lesssim1$ for any
$(p,q,s)_W$-atom $\vec a$ with the implicit positive constant independent of $\vec a$, coincides with $T$
itself on all the $(p,q,s)_W$-atoms.
However, since $\vec l_{j_1}$ in Theorem \ref{finatom} may not converge in
$L^q$, we are unable to obtain such a characterization. Therefore, a
challenging question is whether or not $\|\cdot\|_{H^{p,q,s}_{W,\rm{fin}}}$ and
$\|\cdot\|_{H^p_W}$ are equivalent quasi-norms
on   $H^{p,q,s}_{W,\rm{fin}}$.
\end{enumerate}
\end{remark}

\section{Boundedness of Calder\'on--Zygmund Operators on $H^p_W$\label{cz}}

In this subsection, we establish the boundedness
of Calder\'on--Zygmund operators on $H^p_W$.
We first present the concept of the $s$-order standard kernel
(see, for instance, \cite[Chapter III]{S93}).
In what follows, for any $\gamma=(\gamma_1,\ldots,\gamma_n)\in
{\mathbb Z}_+^n$,
any $\gamma$-order differentiable function $F(\cdot,\cdot)$
on ${\mathbb{R}^n}\times {\mathbb{R}^n}$, and any $(x,y)\in {\mathbb{R}^n}\times {\mathbb{R}^n}$, let
$$
\partial_{(1)}^{\gamma}F(x,y):=\frac{\partial^{|\gamma|}}
{\partial x_1^{\gamma_1}\cdots\partial x_n^{\gamma_n}}F(x,y)
\ \ \text{and}\ \
\partial_{(2)}^{\gamma}F(x,y):=\frac{\partial^{|\gamma|}}
{\partial y_1^{\gamma_1}\cdots\partial y_n^{\gamma_n}}F(x,y).
$$

\begin{definition}\label{def-s-k}
Let $s\in{\mathbb Z}_+$ and $\delta\in(0,1]$. A measurable function $K$
on ${\mathbb{R}^n}\times {\mathbb{R}^n}\setminus\{(x,x):\ x\in{\mathbb{R}^n}\}$
is called an \emph{$(s,\delta)$-type standard kernel} if
there exists a positive constant $C$
such that, for any $\gamma\in{\mathbb Z}_+^n$ with $|\gamma|\le s$,
the followings hold:
\begin{itemize}
\item[\rm (i)]
for any $x,y\in{\mathbb{R}^n}$ with $x\neq y$,
\begin{align}\label{size-s'}
\left|\partial_{(2)}^{\gamma}K(x,y)\right|\le
\frac{C}{|x-y|^{n+|\gamma|}};
\end{align}

\item[\rm (ii)]
\eqref{size-s'} still holds for the first variable of $K$;

\item[\rm (iii)]
for any $x,y,z\in{\mathbb{R}^n}$
with $x\neq y$ and $|x-y|\ge2|y-z|$,
\begin{align}\label{regular2-s}
\left|\partial_{(2)}^{\gamma}K(x,y)-\partial_{(2)}^{\gamma}K(x,z)\right|
\le C\frac{|y-z|^\delta}{|x-y|^{n+|\gamma|+\delta}};
\end{align}

\item[\rm (iv)]
\eqref{regular2-s} still holds for the first variable of $K$.
\end{itemize}
\end{definition}

Next, we present the definition of Calder\'on--Zygmund operators.

\begin{definition}\label{defin-C-Z-s}
Let $s\in{\mathbb Z}_+$ and $\delta\in(0,1]$.
A linear operator $T$ is called an
\emph{$(s,\delta)$-type Calder\'on--Zygmund operator} if $T$ is bounded on $L^2$ and
there exists an $(s,\delta)$-type standard kernel $K$ such that,
for any given $f\in L^2$
and for almost every $x\in{\mathbb{R}^n}$,
$T(f)(x)=\lim_{\eta\to0^+}T_\eta (f)(x),$
where, for any $\eta\in(0,\infty)$,
\begin{align*}
T_\eta (f)(x)
:=\int_{{\mathbb{R}^n}\setminus B(x,\eta)}K(x,y)f(y)\,dy.
\end{align*}
\end{definition}

\begin{remark}\label{rem-2.15}
Let $T$ be the same as in Definition \ref{defin-C-Z-s}.
From a statement in \cite[p.\,102]{Duo01}, we deduce that, for any $q\in(1,\infty)$,
$T$ is bounded on $L^q$ and, moreover, for any $f\in L^q$,
$T(f)=\lim_{\eta\to0^+}T_\eta (f)$
both almost everywhere on ${\mathbb{R}^n}$ and in $L^q$.
\end{remark}

Now, we recall the concept of the well-known
vanishing moments on $T$
(see, for instance, \cite[p.\,23]{mc1997}).

\begin{definition}\label{Def-T-s-v}
Let $s\in{\mathbb Z}_+$ and $\delta\in(0,1]$. An
$(s,\delta)$-type Calder\'on--Zygmund operator $T$
is said to have the \emph{vanishing moments up to
order $s$} if, for any function $a\in L^2$ having compact support
and satisfying that,
for any $\gamma\in{\mathbb Z}_+^n$ with $|\gamma|\le s$,
$\int_{{\mathbb{R}^n}} a(x)x^\gamma\,dx=0$, it holds that
$T^*(x^{\gamma}):=\int_{{\mathbb{R}^n}} T(a)(x)x^\gamma\,dx=0.$
\end{definition}

\begin{theorem}\label{CZ}
Let $p\in(0,1]$, $s\in{\mathbb Z}_+\cap[\lfloor n(\frac{1}{p}-1)\rfloor,\infty)$,
$\delta\in(0,1]$, and $W\in A_p$.
Let $T$ be an $(s,\delta)$-type Calder\'on--Zygmund operator.
Then there exists an operator $\widetilde T_1$ bounded from $H^p_W$
to $L^p_W$ that agrees with $T$ on all the
continuous $(p,\infty,s)_W$-atoms.
If, in addition, $T$ has the vanishing moments up to order $s$,
then there exists an operator $\widetilde T_2$ bounded from $H^p_W$
to $H^p_W$ that agrees with $T$ on all the
continuous $(p,\infty,s)_W$-atoms.
\end{theorem}

\begin{proof}
We first prove that $T$ can  extend to $\widetilde{T}_1$,
which is bounded from $H^p_W$ to $L^p_W$.
Let $\mathbb A:=\{A_Q\}_{Q\in\mathscr{Q}}$
be a family of reducing operators of order $p$ for $W$,
$\vec a$ be a $(p,\infty,s)_W$-atom supported in the cube $Q$,
and $r\in(1,r_W)$, where $r_W$ is the same as in \eqref{r_w}.
By Remark \ref{kfc}(i), we find that $\vec a$ is also
a harmless positive constant multiple of one $(p,\infty,s)_{\mathbb A}$-atom.
From Corollary \ref{p01}, the H\"older inequality, Lemma \ref{86},
the fact that $T$ is bounded on $L^{r'p}$ (which is a consequence of the assumptions),
and both (i) and (iv) of Definition \ref{F-atom},
we infer that
\begin{align}\label{inQ}
\int_{2\sqrt{n}Q}\left\|W^{\frac1p}(x)T\vec a(x)\right\|^p\,dx
&\le\int_{2\sqrt{n}Q}\left\|W^{\frac1p}(x)A_{2\sqrt{n}Q}^{-1}\right\|^p
\left\|A_{2\sqrt{n}Q}A_{Q}^{-1}\right\|^p\left|A_QT\vec a(x)\right|^p\,dx\\
&\le\left[\int_{2\sqrt{n}Q}\left\|W^{\frac1p}(x)
A_{2\sqrt{n}Q}^{-1}\right\|^{rp}\,dx\right]^{\frac1r}
\left\|A_QT\vec a\right\|_{L^{r'p}}^p\nonumber\\
&\lesssim|Q|^{\frac1r}\left\|A_Q\vec a\right\|_{L^{r'p}}^p
\lesssim 1.\nonumber
\end{align}
By Definition \ref{F-atom}(iii), we conclude that there exists
$\xi_y\in Q$ such that
\begin{align*}
&\int_{(2\sqrt{n}Q)^\complement}\left|W^{\frac1p}(x)T\vec a(x)\right|^p\,dx\\
&\quad\le\int_{(2\sqrt{n}Q)^\complement}\left\|W^{\frac1p}(x)A_Q^{-1}\right\|^p\int_Q
\left|\left[k(x,y)-\sum_{|\beta|\le s}
\frac{\partial^\beta_xk(x,c_Q)}{\beta !}(y-c_Q)^\beta\right]A_Q\vec a(y)\right|^p\,dy\,dx\nonumber\\
&\quad=\int_{(2\sqrt{n}Q)^\complement}\left\|W^{\frac1p}(x)A_Q^{-1}\right\|^p\int_Q
\left|\sum_{|\beta|=s}\frac{\partial^\beta_xk(x,c_Q)-\partial^\beta_xk(x,\xi_y)}{\beta!}
(y-c_Q)^\beta A_Q\vec a(y)\right|^p\,dy\,dx,
\end{align*}
which, together with \eqref{regular2-s}, both (i) and (iv) of Definition \ref{F-atom},
Corollary \ref{p01}, and Lemma \ref{86},
further implies that
\begin{align}\label{outQ}
&\int_{(2\sqrt{n}Q)^\complement}\left|W^{\frac1p}(x)T\vec a(x)\right|^p\,dx\\
&\quad\le\int_{(2\sqrt{n}Q)^\complement}\left\|W^{\frac1p}(x)A_Q^{-1}\right\|^p\int_Q
\left|A_Q\vec a(y)\right|^p\frac{|y-c_Q|^{s+\delta}}{|x-c_Q|^{n+s+\delta}}\,dy\,dx\nonumber\\
&\quad\lesssim\sum_{i=0}^{\infty}\int_{2^{i+1}Q2^iQ}
\left\|W^{\frac1p}(x)A_{2^{i+1}Q}^{-1}\right\|^p
\left\|A_{2^{i+1}Q}A_{Q}^{-1}\right\|^p
\frac1{(2^il)^{n+s+\delta}}\,dx
|Q|^{\frac{s+\delta}{n}}\nonumber\\
&\quad\lesssim\sum_{i=0}^{\infty}2^{-(n+s+\delta)i}\int_{2^{i+1}Q}
\left\|W^{\frac1p}(x)A_{2^{i+1}Q}^{-1}\right\|^p\,dx|Q|^{-1}\lesssim1.\nonumber
\end{align}
Using \eqref{inQ}, \eqref{outQ}, and Theorem \ref{subl},
we obtain the desired result.

Next,  further assuming that $T$ has the vanishing moments up to order $s$,
we show that $T$ can extend to $\widetilde{T}_2$,
which is bounded from $H^p_W$ to $H^p_W$.
By the assumption that $T$ is bounded on $L^2$
and both (i) and (iv) of Definition \ref{F-atom}, we conclude that
\begin{align}\label{TaL2}
\left\|T\left(A_Q\vec a\right)\right\|_{L^2}\lesssim\left\|A_Q\vec a\right\|_{L^2}
\lesssim\left|Q\right|^{\frac{1}{2}-\frac{1}{p}}.
\end{align}
Let $r_Q:=\frac{\sqrt{n}}2l(Q)$
and $\psi\in\mathcal{S}$ satisfy
$\operatorname{supp}\psi\subset B(\mathbf{0},1)$
and $\int_{\mathbb R^n}\psi(x)\,dx\neq 0$.
From the vanishing moments of $T$, we deduce that,
for any $t\in(0,\infty)$ and $x\in B(c_Q,4r_Q)^\complement$,
\begin{align}\label{rens}
&\frac{1}{t^n}\int_{{\mathbb{R}^n}}\left|\psi\left(\frac{x-y}{t}\right)
W^{\frac{1}{p}}(x)T\vec a(y)\right|\,dy\\
&\quad\le\left\|W^{\frac{1}{p}}(x)A_Q^{-1}\right\|\frac{1}{t^n}
\int_{{\mathbb{R}^n}}\left|\psi\left(\frac{x-y}{t}\right)
T(A_Q\vec a)(y)\right|\,dy.\nonumber\\
&\quad\le\left\|W^{\frac{1}{p}}(x)A_Q^{-1}\right\|\frac1{t^n}\int_{{\mathbb{R}^n}}
\left|\psi\left(\frac{x-y}{t}\right)-
\sum_{|\beta|\le s}
\frac{\partial^\beta\psi(\frac{x-c_Q}{t})}{\beta!}\left(\frac{y-c_Q}{t}\right)^\beta\right||T(A_Q\vec a)(y)|\,dy\nonumber\\
&\quad=\left\|W^{\frac{1}{p}}(x)A_Q^{-1}\right\|\frac1{t^n}\left(\int_{|y-c_Q|<2r_Q}+\int_{2r_Q\le|y-c_Q|
<\frac{|x-c_Q|}{2}}+\int_{|y-c_Q|\geq\frac{|x-c_Q|}{2}}\right)\nonumber\\\nonumber
&\quad\hspace*{12pt}\times\left|\psi\left(\frac{x-y}{t}\right)-\sum_{|\beta|\le s}
\frac{\partial^\beta\psi(\frac{x-c_Q}{t})}{\beta!}
\left(\frac{y-c_Q}{t}\right)^\beta\right||T(A_Q\vec a)(y)|\,dy\nonumber\\
&\quad=:\left\|W^{\frac{1}{p}}(x)A_Q^{-1}\right\|
\left(\mathrm{I_1}+\mathrm{I_2}+\mathrm{I_3}\right).\nonumber
\end{align}

Using \eqref{TaL2} and repeating the estimation of
\cite[(6.14), (6.15), and (6.16)]{zyyw} with $a_j$ and $\gamma$
replaced, respectively, by $A_Q\vec a$ and $s+\delta$, we obtain
\begin{align}\label{Eq79}
\mathrm{I_1}\lesssim\frac{[l(Q)]^{s+1}}{|x-c_Q|^{n+s+1}}
\left\|T\left(A_Q\vec a\right)\right\|_{L^2}|Q|^{1/2}
\lesssim\frac{[l(Q)]^{n+s+1-\frac{n}{p}}}{|x-c_Q|^{n+s+1}},
\end{align}
\begin{align}\label{Eq710}
\mathrm{I_2}
\lesssim\frac{[l(Q)]^{s+\delta}}{|x-c_Q|^{n+s+1}}\int_{2r_Q\le|y-c_Q|<\frac{|x-c_Q|}{2}}
\frac1{|y-c_Q|^{n+\delta}}\,dy\left\|T\left(A_Q\vec a\right)\right\|_{L^2}|Q|^{1/2}
\lesssim\frac{[l(Q)]^{n+s+\delta-\frac{n}{p}}}{|x-c_Q|^{n+s+\delta}},
\end{align}
and
\begin{align}\label{Eq711}
\mathrm{I_3}&\lesssim\left\|T\left(A_Q\vec a\right)\right\|_{L^2}\left|Q\right|^{1/2}
\left[\frac{[l(Q)]^{s+\delta}}{|x-c_Q|^{n+s+\delta}}
\int_{|y-c_Q|\geq\frac{|x-c_Q|}{2}}|\psi_t(x-y)|\,dy\right.\\
&\hspace*{12pt}\left.+\sum_{|\beta|\le s}[l(Q)]^{s+\delta}
\int_{|y-c_Q|\geq\frac{|x-c_Q|}{2}}\frac 1{|x-c_Q|^{n+|\beta|}}
\frac1{|y-c_Q|^{n+s+\delta-|\beta|}}\,dy\right]\nonumber\\
&\lesssim\frac{[l(Q)]^{n+s+\delta-\frac{n}{p}}}{|x-c_Q|^{n+s+\delta}}.\nonumber
\end{align}
Combining \eqref{rens}, \eqref{Eq79}, \eqref{Eq710}, and \eqref{Eq711},
we conclude that, for any $x\in [B(c_Q,4r_Q)]^\complement$,
\begin{align*}
M^p_{W}(T\vec a,\psi)(x)\lesssim\left\|W^{\frac{1}{p}}(x)A_Q^{-1}\right\|
\frac{[l(Q)]^{n+s+\delta-\frac{n}{p}}}{|x-c_Q|^{n+s+\delta}}.
\end{align*}
From this, Corollary \ref{p01}, Lemma \ref{86},
and $s\in{\mathbb Z}_+\cap[\lfloor n(\frac{1}{p}-1)\rfloor,\infty)$, we infer that
\begin{align}\label{outball-T}
&\int_{[B(c_Q,4r_Q)]^\complement}\left[M^p_{W}(T\vec a,\psi)(x)\right]^p\,dx\\
&\quad\lesssim\sum_{i=1}^{\infty}\int_{2^{i+2}Q\setminus2^{i+1}Q}
\left\|W^{\frac{1}{p}}(x)A_{2^{i+2}Q}^{-1}\right\|^p
\left\|A_{2^{i+2}Q}A_{Q}^{-1}\right\|^p
\frac{[l(Q)]^{(n+s+\delta-\frac{n}{p})p}}{|x-c_Q|^{(n+s+\delta)p}}\,dx\nonumber\\
&\quad\lesssim\sum_{i=1}^{\infty}\int_{2^{i+2}Q\setminus2^{i+1}Q}
\left\|W^{\frac{1}{p}}(x)A_{2^{i+2}Q}^{-1}\right\|^p\,dx
\frac{[l(Q)]^{(n+s+\delta-\frac{n}{p})p}}{[2^il(Q)]^{(n+s+\delta)p}}
\lesssim\sum_{i=1}^{\infty}2^{(n-np-sp-\delta p)i}\lesssim1.\nonumber
\end{align}
Let $r\in(0,\min\{r_W,\frac{1}{1-p}\})$.
Using the H\"older inequality, Corollary \ref{p01}, Lemma \ref{86},
\cite[Corollary 2.1.12]{g14c},
the facts that $T$ and $\mathcal{M}$ are bounded on
$L^{\frac{pr}{r-1}}$, and both
(i) and (iv) of Definition \ref{F-atom},
we find that
\begin{align}\label{inball-T}
&\int_{B(c_Q,4r_Q)}\left[M^p_W(T\vec a,\psi)(x)\right]^p\,dx\\
&\quad\le\int_{4\sqrt{n}Q}\left\|W^{\frac{1}{p}}(x)A_{4\sqrt{n}Q}^{-1}\right\|^p
\left\|A_{4\sqrt{n}Q}A_Q^{-1}\right\|^p\sup_{t\in(0,\infty)}
\left|A_{Q}(\psi_t*T\vec a)(x)\right|^p\,dx\nonumber\\
&\quad\le\left[\int_{4\sqrt{n}Q}\left\|W^{\frac{1}{p}}(x)
A_{4\sqrt{n}Q}^{-1}\right\|^{pr}\,dx\right]^{\frac{1}{r}}
\left[\int_{4\sqrt{n}Q}\sup_{t\in(0,\infty)}
\left(\psi_t*\left|TA_{Q}\vec a\right|\right)^{\frac{pr}{r-1}}(x)\,dx\right]^{\frac{r-1}{r}}\nonumber\\
&\quad\lesssim|Q|^{\frac{1}{r}}\left\{\int_{{\mathbb{R}^n}}
\left[\mathcal{M}\left(\left|TA_{Q}\vec a\right|\right)\right]^{\frac{pr}{r-1}}(x)
\,dx\right\}^{\frac{r-1}{r}}
\lesssim|Q|^{\frac{1}{r}}\left[\int_{{\mathbb{R}^n}}
\left|A_{Q}\vec a(x)\right|^
{\frac{pr}{r-1}}\,dx\right]^{\frac{r-1}{r}}
\lesssim 1.\nonumber
\end{align}
Combining \eqref{outball-T}, \eqref{inball-T}, and Theorem \ref{subl},
we obtain the desired conclusion.
This finishes the proof of Theorem \ref{CZ}.
\end{proof}

\begin{remark}
Let $m=1$, $p\in(0,1]$, $s\in[\lfloor n(\frac{1}{p}-1)\rfloor,\infty)\cap{\mathbb Z}_+$,
and $W\equiv1$.
Using \cite[p.\,23,\ Proposition 4]{mc1997}, we conclude that
a bounded linear operator $T:\ L^2\to L^2$ whose kernel satisfies \eqref{size-s'}
has the vanishing moments up to order $s$
if and only if $T$ extends to a bounded linear operator on $H^p_W$,
which, in some sense, further implies the necessity of the assumption that
$T$ has the vanishing moments up to order $s$ in Theorem \ref{CZ}.
\end{remark}

\bigskip

\noindent Fan Bu, Yiqun Chen, Dachun Yang (Corresponding author) and Wen Yuan

\medskip

\noindent Laboratory of Mathematics and Complex Systems (Ministry of Education of China),
School of Mathematical Sciences, Beijing Normal University, Beijing 100875, The
People's Republic of China

\smallskip

\noindent{\it E-mails:} \texttt{fanbu@mail.bnu.edu.cn} (F. Bu)

\noindent\phantom{{\it E-mails:} }\texttt{yiqunchen@mail.bnu.edu.cn} (Y. Chen)

\noindent\phantom{{\it E-mails:} }\texttt{dcyang@bnu.edu.cn} (D. Yang)

\noindent\phantom{{\it E-mails:} }\texttt{wenyuan@bnu.edu.cn} (W. Yuan)

\end{document}